\documentclass[11pt]{amsart}
\usepackage{amscd}
\usepackage{amsmath}
\usepackage{amssymb}
\usepackage{latexsym}
\usepackage{lscape}
\usepackage{xypic}

\input
xy \xyoption{all}

\newtheorem{thm}{Theorem}[section]
\newtheorem{prop}[thm]{Proposition}
\newtheorem{lem}[thm]{Lemma}
\newtheorem{lem-def}[thm]{Lemma-Definition}
\newtheorem{cor}[thm]{Corollary}
\newtheorem*{prop*}{Proposition}

\theoremstyle{remark}
\newtheorem{rmk}[thm]{Remark}
\newtheorem{ex}[thm]{Example}

\theoremstyle{definition}
\newtheorem{dfn}[thm]{Definition}

\numberwithin{equation}{section}

\newcommand{\frake}{{\mathfrak
e}}
\newcommand{\frakf}{{\mathfrak
f}}

\newcommand{\frakp}{{\mathfrak
p}}

\newcommand{\fraks}{{\mathfrak
s}}

\newcommand{\bbA}{{\mathbb
A}}

\newcommand{\bbC}{{\mathbb
C}}

\newcommand{\bbF}{{\mathbb
F}}
\newcommand{\bbG}{{\mathbb
G}}

\newcommand{\bbK}{{\mathbb
K}}
\newcommand{\bbL}{{\mathbb
L}}

\newcommand{\bbN}{{\mathbb
N}}

\newcommand{\bbT}{{\mathbb
T}}
\newcommand{\bbU}{{\mathbb
U}}

\newcommand{\bbZ}{{\mathbb
Z}}

\newcommand{\calC}{{\mathcal
C}}
\newcommand{\calD}{{\mathcal
D}}

\newcommand{\calF}{{\mathcal
F}}
\newcommand{\calG}{{\mathcal
G}}
\newcommand{\calH}{{\mathcal
H}}

\newcommand{\calL}{{\mathcal
L}}

\newcommand{\calO}{{\mathcal
O}}
\newcommand{\calP}{{\mathcal
P}}
\newcommand{\calQ}{{\mathcal
Q}}

\newcommand{\calT}{{\mathcal
T}}

\newcommand{\calZ}{{\mathcal
Z}}
\newcommand{\End}{{\mathrm{End}}}
\newcommand{\Hom}{{\mathrm{Hom}}}
\newcommand{\ind}{{\mathrm{Ind}}}
\newcommand{\pic}{{\mathrm{Pic}}}

\newcommand{\rep}{{\mathrm{Rep}}}
\newcommand{\res}{{\mathrm{Res}}}

\newcommand{\spec}{{\mathrm{Spec}}}
\newcommand{\tr}{{\mathrm{Tr}}}
\newcommand{\Cl}{{\mathrm{Cl}}}
\newcommand{\Gr}{{\mathrm{Gr}}}
\newcommand{\Mod}{{\mathrm{-Mod}}}
\newcommand{\GL}{{\mathbb
G\mathbb L}}
\newcommand{\gl}{{\mathfrak
g\mathfrak l}}

\newcommand{\nc}{\newcommand}
\nc{\on}{\operatorname} \nc{\ch}{\mbox{ch}} \nc{\Z}{{\mathbb Z}}
\nc{\C}{{\mathbb C}} \nc{\pone}{{\mathbb P}^1} \nc{\pa}{\partial}
\nc{\F}{{\mathcal F}} \nc{\arr}{\rightarrow}
\nc{\larr}{\longrightarrow} \nc{\al}{\alpha} \nc{\ri}{\rangle}
\nc{\lef}{\langle} \nc{\W}{{\mathcal W}} \nc{\la}{\lambda}
\nc{\ep}{\epsilon} \nc{\su}{\widehat{{\mathfrak s}{\mathfrak
l}}_2} \nc{\sw}{{\mathfrak s}{\mathfrak l}} \nc{\g}{{\mathfrak g}}
\nc{\h}{{\mathfrak h}} \nc{\n}{{\mathfrak n}}
\nc{\N}{\widehat{\n}} \nc{\G}{\widehat{\g}} \nc{\De}{\Delta}
\nc{\gt}{\widetilde{\g}} \nc{\Ga}{\Gamma} \nc{\one}{{\mathbf 1}}
\nc{\z}{{\mathfrak Z}} \nc{\La}{\Lambda} \nc{\wt}{\widetilde}
\nc{\wh}{\widehat} \nc{\cri}{_{\kappa_c}} \nc{\kk}{h^\vee}
\nc{\sun}{\widehat{\sw}_N} \nc{\si}{\sigma} \nc{\el}{\ell}
\nc{\bi}{\bibitem} \nc{\om}{\omega} \nc{\ol}{\overline}
\nc{\ds}{\displaystyle} \nc{\dzz}{\frac{dz}{z}}
\nc{\Res}{\on{Res}} \nc{\mc}{\mathcal} \nc{\Cal}{\mathcal}
\nc{\bb}{{\mathfrak b}} \nc{\ot}{\otimes} \nc{\R}{{\mc R}}
\nc{\yy}{{\mc Y}} \nc{\ga}{\gamma}

\nc{\us}{\underset} \nc{\opl}{\oplus} \nc{\beq}{\begin{equation}}
\nc{\Fq}{{\mathcal F}} \nc{\Mq}{{\mathcal M}} \nc{\Rep}{\on{Rep}}
\nc{\sssec}{\subsubsection} \nc{\ssec}{\subsection}
\nc{\lan}{\langle} \nc{\ran}{\rangle}

\nc{\D}{\mathcal D} \nc{\Vect}{\on{Vect}} \nc{\ghat}{\G}
\nc{\T}{\mc T} \nc{\Tloc}{\T^\g_{\on{loc}}} \nc{\vac}{|0\ran}
\nc{\Wick}{{\mb :}} \nc{\mb}{\mathbf} \nc{\delz}{\partial_z}
\nc{\K}{{\cali K}} \nc{\cali}{\mathcal} \nc{\li}{\mathfrak l}
\nc{\lt}{\widetilde{\li}} \nc{\astar}{a^*} \nc{\cA}{{\mc A}}
\nc{\ka}{\kappa}

\nc{\OO}{{\mc O}} \nc{\AutO}{\on{Aut}\OO} \nc{\DerO}{\on{Der}\OO}
\nc{\DerpO}{\on{Der}_+\OO} \nc{\Au}{{\mc A}ut} \nc{\mf}{\mathfrak}
\nc{\V}{{\mathbb V}} \nc{\hh}{\wh{\h}}

\nc{\pp}{{\mathfrak p}} \nc{\mm}{{\mathfrak m}}
\nc{\rr}{{\mathfrak r}} \nc{\ket}{\rangle} \nc{\zz}{{\mathfrak z}}
\nc{\gr}{\on{gr}} \nc{\Spe}{\on{Spec}} \nc{\rv}{\crho}
\nc{\can}{\on{can}} \nc{\CC}{\on{Op}_G(D))} \nc{\Op}{\on{Op}_G(D)}
\nc{\MOp}{\on{MOp}_G(D)} \nc{\Db}{{\mathbb D}} \nc{\ww}{w}

\nc{\af}{{\mathbb A}^1} \nc{\bs}{\backslash} \nc{\laa}{(\la_i)}
\nc{\zn}{(z_i)}

\nc{\cla}{\check{\la}} \nc{\cmu}{\check{\mu}}
\nc{\crho}{\check{\rho}} \nc{\chal}{\check{\al}}
\nc{\cc}{{\mathfrak c}}

\nc{\M}{{\mathbb M}}

\nc{\ZZ}{{\mc Z}}

\nc{\UU}{{\mathbb U}}

\nc{\Conn}{\on{Conn}(\Omega^{\crho})}
\nc{\Con}{\on{Conn}(\Omega^{-\rho})}
\nc{\Co}{\on{Conn}(\Omega^{\rho})}

\nc{\ppart}{(\!(t)\!)} \nc{\zpart}{(\!(z)\!)}
\nc{\ppzi}{(\!(t-z_i)\!)} \nc{\ppinf}{(\!(t^{-1})\!)}
\nc{\Ind}{\on{Ind}} \nc{\I}{{\mathbb I}} \nc{\ppars}{(\!(s)\!)}
\nc{\QCoh}{\on{QCoh}}

\begin{document}
\title{Gerbal
representations of double loop groups}\thanks{Supported by DARPA
and AFOSR through the grant FA9550-07-1-0543}

\author{Edward Frenkel}

\address{Department of
Mathematics, University of California, Berkeley, CA 94720, USA}

\author{Xinwen
Zhu}

\date{October 2008}

\begin{abstract}
A crucial role in representation theory of loop groups of reductive
Lie groups and their Lie algebras is played by their non-trivial
second cohomology classes which give rise to their central extensions
(the affine Kac--Moody groups and Lie algebras). Loop groups embed
into the group $GL_\infty$ of continuous automorphisms of $\C\ppart$,
and these classes come from a second cohomology class of
$GL_\infty$. In a similar way, double loop groups embed into a group
of automorphisms of $\C\ppart\ppars$, denoted by $GL_{\infty,\infty}$,
which has a non-trivial third cohomology. In this paper we explain how
to realize a third cohomology class in representation theory of a
group: it naturally arises when we consider representations on
categories rather than vector spaces. We call them ``gerbal
representations.'' We then construct a gerbal representation of
$GL_{\infty,\infty}$ (and hence of double loop groups), realizing its
non-trivial third cohomology class, on a category of modules over an
infinite-dimensional Clifford algebra. This is a two-dimensional
analogue of the fermionic Fock representations of the ordinary loop
groups.
\end{abstract}

\maketitle

\tableofcontents

\section*{Introduction}

The main motivation for this paper and its sequel \cite{next} is to
understand representation theory of double loop groups and the
corresponding Lie algebras and the role played in it by their third
cohomology.

To explain this, we revisit representation theory of the ordinary loop
groups. Its algebraic version is the formal loop group $G\ppart$ of a
reductive complex algebraic group $G$. The corresponding Lie algebra is
the formal loop algebra $\g\ppart$, where $\g=\on{Lie} G$. A great
discovery made over thirty years ago was the realization that in order
to treat its representation theory ``in the right way'', we need to
pass from $\g\ppart$ to its universal {\em central extension}, the
affine Kac--Moody algebra $\ghat$ (and similarly for the loop
group). These central extensions (for different $\g$) may be obtained
in a unified way from the central extension of the ``master'' Lie
algebra $\gl_\infty$ of continuous endomorphisms of $\C\ppart$. It is
known \cite{FT1} that $H^\bullet(\gl_\infty,\C) = \C[c_2,c_4,...]$
with $\on{deg} c_{2n} = 2n$, and so $H^2(\gl_\infty,\C) = \C c_2$.
The corresponding two-cocycle defines the universal central extension
$\wh\gl_\infty$ of $\gl_\infty$. For each finite-dimensional
representation $V$ of $\g$ we have a natural embedding of $\g\ppart$
into $\gl_\infty$ (as endomorphisms of $V\ppart \simeq \C\ppart$),
and the pull-back of the universal central extension of $\gl_\infty$
gives rise to $\ghat$. In the same way one obtains the Virasoro
algebra and its semi-direct products with $\ghat$.

Representation theory of these groups and Lie algebras begins in
earnest when we give an example of a natural representation with a
non-zero action of the central element. The representation that is the
easiest to construct is the {\em fermionic Fock representation} of
$\wh\gl_\infty$ \cite{KP}. Its restriction to $\ghat$ gives rise to
important integrable representations \cite{IFrenkel,KP}, and we have
the beginnings of the general theory \cite{Kac}.

The fermionic Fock representation may also be constructed
geometrically, as the space of global sections of the {\em determinant
line bundle} on the infinite Grassmannian $\on{Gr}$ which parametrizes
``lattices'' in the infinite-dimensional topological vector space
$\C\ppart$. The Lie algebra $\gl_\infty$ and the corresponding group
$GL_\infty$ naturally act on $\on{Gr}$. The important point is that
this action does not lift to the determinant line bundle; it is only
the central extension that acts on this line bundle.

\medskip

Now let us consider the formal double loop group $G\ppart\ppars$ and
its Lie algebra $\g\ppart\ppars$.\footnote{These objects have
polynomial versions in which $\C\ppart\ppars$ is replaced by
$\C[t^{\pm 1},s^{\pm 1}]$, and most of our results have analogues for
these polynomial versions. However, we focus in this paper on the case
of the formal power series, because the corresponding theory seems
more natural.} Just like ordinary loop algebras, these Lie algebras
naturally embed into another ``master'' Lie algebra,
$\gl_{\infty,\infty}$, which consists of continuous endomorphisms of
$\C\ppart\ppars$. Hence in order to develop representation theory of
double loop algebras it is natural to start with $\gl_{\infty,\infty}$
and see what kinds of representations we can construct. By analogy
with loop algebras, we look at the cohomology of
$\gl_{\infty,\infty}$. It is now the exterior algebra
$\bigwedge(e_3,e_5,...)$ with $\on{deg} e_{2n+1} = 2n+1$
\cite{FT1}. In particular, the second cohomology that is responsible
for central extensions vanishes, and the first non-trivial cohomology
class $e_3$ occurs in degree three.

This cohomology class, in turn, gives rise to third cohomology classes
of all double loop algebras. One can show that this cohomology class
is non-zero for $\g={\mathfrak s}{\mathfrak l}_n, n>2$, as well as for
$\g={\mathfrak g}{\mathfrak l}_1$ (see \cite{Feigin}). We also obtain
third cohomology classes of the Lie algebra of derivations of
$\C\ppart\ppars$ and its semi-direct products with $\g\ppart\ppars$,
which also embed naturally into $\gl_{\infty,\infty}$. We wish to
interpret these cohomology classes from the representation theoretic
point of view.\footnote{Note that double loop algebras also have
non-trivial (actually, infinite-dimensional) second cohomology.
Representations of the corresponding universal central extension have
been studied in the literature; see, e.g., \cite{MRY,La,BBS,B}.}

To contrast representations of $\gl_\infty$ and $\gl_{\infty,\infty}$
in more concrete terms, let us recall that the naive action of
$\gl_\infty$ on the fermionic Fock representation creates infinite
expressions, which need to be regularized by imposing ``normal
ordering'' (see, e.g., \cite{FB}). This normal ordering
creates an ``anomaly'': commutation relations get distorted by
additional terms which are interpreted as a central extension of
$\gl_\infty$. If we try to imitate this procedure in the case of
$\gl_{\infty,\infty}$, the result is much worse: not only the
commutation relations get distorted, but the additional terms
themselves turn out to be infinite. The anomaly is thus much more
severe in the case of $\gl_{\infty,\infty}$ than in the case of
$\gl_\infty$. It cannot possibly be cured by normal ordering alone. In
fact, the cohomology class related to this anomaly has now migrated
from $H^2$ to $H^3$.

How can we possibly interpret this third cohomology class in
representation theory?  The idea is that these classes are naturally
realized when groups (or Lie algebras) act on {\em categories} rather
than vector spaces. (Informally, one can say that the structure of a
category is needed to absorb the $H^3$-anomaly discussed above.) To
see how this works, we first recall how the second cohomology class is
realized.

Suppose that $G$ is a group and $V$ is a complex vector space. To
define a representation of $G$ on $V$ (here, for simplicity, we
consider $G$ as an abstract group; but we extend this to algebraic
groups and Lie algebras below and in \cite{next}), we need to
assign to each $g \in G$ a linear operator $T_g$ on $V$, so that
$1 \in G$ goes to the identity, and for each pair $g,h \in G$ we
have the equality $T_{gh} = T_g T_h$. We generalize this by
relaxing the last condition and demanding only that
$$
T_{gh} = \al_{g,h} T_g T_h,
$$
where $\al_{g,h} \in \C^\times$. Thus, we arrive at the notion of
projective representation of $G$, or equivalently, a
representation of the central extension of $G$ corresponding to
$\al_{g,h}$ (one checks easily that it defines a two-cocycle of
$G$ with coefficients in $\C^\times$).

Now we generalize this as follows. We replace a complex vector
space $V$ by an abelian category ${\mc C}$ over $\C$. A
representation of $G$ on ${\mc C}$ is a rule that assigns to each
$g \in G$ a functor $F_g$ so that $1 \in G$ goes to the identity
functor. For each pair $g,h \in G$ we then have two functors,
$F_{gh}$ and $F_g \circ F_h$. Functors are objects of a category
(rather than a set), and therefore it is not a good idea to demand
that they be equal.  Rather, we should demand that they are
isomorphic. So we include the data of isomorphisms $i_{g,h}: F_g
\circ F_h \overset\sim\longrightarrow F_{gh}$ in our definition.
Suppose now that we have three elements $g,h,k \in G$. Then we
have two different isomorphisms between $F_g \circ F_h \circ F_k$
and $F_{ghk}$; namely, $i_{g,hk} \circ i_{h,k}$ and $i_{gh,k}
\circ i_{g,h}$. These are already elements of a set (that of
morphisms from $F_g \circ F_h \circ F_k$ to $F_{ghk}$). Demanding
that
$$
i_{g,hk} \circ i_{h,k} = i_{gh,k}\circ i_{g,h},
$$
we obtain an analogue of an ordinary representation of $G$ on a
vector space. Alternatively, we may demand that this equality is
only satisfied up to a non-zero scalar:
$$
i_{g,hk} \circ i_{h,k} = \al_{g,h,k} i_{gh,k} \circ i_{g,h},
\qquad \al_{g,hk} \in \C^\times.
$$

It is easy to check that $\al_{g,h,k}$ defines a three-cocycle of $G$
with coefficients in $\C^\times$. It is non-trivial if and only if our
representation of $G$ on ${\mc C}$ is not equivalent to an ordinary
representation (that is, one with $\al_{g,h,k} \equiv 1$), in the
obvious sense. Moreover, one can check that if two representations of
$G$ on ${\mc C}$ are isomorphic, in a natural sense, then
the corresponding two cocycles differ by a coboundary.

Thus, we obtain a natural realization of the third cohomology when we
consider representations of groups on categories. This notion may also
be generalized to Lie algebras, as we explain in \cite{next}. We call
representations of this type {\em gerbal representations}.

\medskip

The idea that a group or a Lie algebra should act on a category
rather than a vector space is not new. In recent years
representations of groups on categories have naturally arisen in
different contexts. For instance, in the work of D. Gaitsgory and
one of the authors \cite{FG} (see \cite{F:book} for an exposition)
categories with an action of the loop group of a reductive
algebraic group $G$ naturally arise in the framework of local
geometric Langlands correspondence (in this case the third
cohomology class is trivial). Another class of examples is
presumably provided by various categories of branes arising in
topological field theory. The group of symmetries of the theory
should act on such a category. Moreover, the data of
non-trivial third cohomology classes may be naturally included in
this context (see, e.g., \cite{Freed} and references therein). One
can probably use these data to construct gerbal representations of
groups on categories of branes realizing these cohomology classes,
though so far we have not seen examples of such representations
discussed in the literature.

\medskip

The goal of this paper is to construct explicitly non-trivial examples
of gerbal representations of the group $GL_{\infty,\infty}$. In
\cite{next} we will construct gerbal representations of the Lie
algebra $\gl_{\infty,\infty}$. One of the gerbal representations of
$GL_{\infty,\infty}$ that we construct may be thought of as an
analogue of the fermionic Fock representation of $\gl_\infty$. The
corresponding category may be realized as the category of Fock
representations of a Clifford algebra built from the vector space
$\C\ppart[[s]]$ plus its dual. We compute explicitly the third
cohomology class of $GL_{\infty,\infty}$ corresponding to this
representation and show that it is non-zero. It gives rise to a
non-trivial central extension of $GL_{\infty,\infty}$ by $B\C^\times$
(see Definition \ref{2-group of central extensions}) considered
previously by S. Arkhipov and K. Kremnizer in \cite{AK}, following
ideas of M. Kapranov. Thus, we obtain a genuine representation of this
extension on a category of representations of this Clifford algebra.

For any reductive group $G$ and a finite-dimensional representation
$V$ of $G$ we have a natural embedding of the double loop group
$G\ppart\ppars$ into $GL_{\infty,\infty}$. Hence our gerbal
representation of $GL_{\infty,\infty}$ gives rise to gerbal
representations of $G\ppart\ppars$. The corresponding cohomology class
is the restriction of the third cohomology class of
$GL_{\infty,\infty}$ to $G\ppart\ppars$. We expect that this
restriction is non-zero if and only if $H^6(BG,\Z)$ is non-zero.  This
is the case, for example, for the groups $GL_1, SL_n, n>2$, and
$GSp_{2n}, n>1$. (This is analogous to the fact that the central
extensions of $G\ppart$ are classified by $H^4(BG,\Z)$.)

\medskip

We expect that gerbal representations of $GL_{\infty,\infty}$ may also
be constructed geometrically, using a ``2-infinite Grassmannian''.
Set-theoretically, this is just the set of lattices in
$\C\ppart\ppars$. The natural transitive action of
$GL_{\infty,\infty}$ on this set lifts to the action of the above
$B\C^\times$-extension of $GL_{\infty,\infty}$ on a $\C^\times$-gerbe
over it (see also \cite{AK}). This is analogous to the action of
$\wh{GL}_\infty$ on the $\C^\times$-bundle over the infinite
Grassmannian. Recall that we obtain a representation of
$\wh{GL}_\infty$ by taking the vector space of global sections of the
corresponding determinant line bundle. If one could define the
2-infinite Grassmannian and the $\C^\times$-gerbe over it
algebro-geometrically, then the category of ``global sections'' of the
sheaf of abelian categories corresponding to the $\C^\times$-gerbe
(the way a line bundle corresponds to a $\C^\times$-bundle) would give
us a gerbal representation of $GL_{\infty,\infty}$. This is still an
open question, but we make some comments on how to answer it \S
\ref{2-Gr}.

\medskip

We hope that the gerbal representations of $GL_{\infty,\infty}$
that we construct in this paper are the tip of the iceberg of a
rich and interesting representation theory of this group and the
double loop groups.

As for possible applications of this theory, note that the
infinite Grassmannian and the fermionic Fock representation of
$\wh{GL}_\infty$ play an important role in the study of the KP
hierarchy and closely related integrable hierarchies, such as KdV
equations (see \cite{DJKM,SW}). The gerbal representations of
$GL_{\infty,\infty}$ which we construct here and their possible links
to the 2-infinite Grassmannian may give us some clues as to what
two-dimensional analogues of integrable hierarchies of soliton
equations might look like. In particular, one can ask what is the
two-dimensional analogue of the boson--fermion correspondence which is
important in the study of the KP hierarchy.

\medskip

The paper is organized as follows. In \S \ref{central extensions} we
review the ``1-dimensional story'', that is, representations of the
``master'' group $GL_\infty$, from different points of view.  We
explain the construction of its central extension and describe a
natural representation on a Fock module over an infinite-dimensional
Clifford algebra. The goal of this paper is to develop an analogous
representation theory for the ``2-dimensional'' group
$GL_{\infty,\infty}$. The key difference is that representations of
the latter are realized in categories rather than vector space. Hence
we need to develop the formalism of actions of groups on
categories. This is done in \S \ref{generalities}. In particular, we
show that to a gerbal representation of a group $G$ on an abelian
category $\calC$ corresponds a cohomology class in
$H^3(G,\calZ(\calC)^\times)$, where $\calZ(\calC)$ is the center of
$\calC$.

In \S \ref{Gerbal extensions} we introduce the group
$GL_{\infty,\infty}$ and related groups and Lie algebras. We construct
a natural gerbal representation of $GL_{\infty,\infty}$ on an abelian
category $\calC_{\bbL}^{\on{ss}}$ of modules over a Clifford algebra.
We show that it lifts to a genuine representation of the 2-group
$\bbG\bbL_{\infty,\infty}$ which is an extension of
$GL_{\infty,\infty}$ by determinantal gerbes.

Next, we wish to calculte the third cohomology class corresponding to
this gerbal representation of $GL_{\infty,\infty}$. We introduce the
relevant cohomology groups in \S \ref{coh}. In particular, we discuss
the universal $\bbZ$-central extension of $GL_{\infty,\infty}$, which
is interesting on its own right. Computing the cohomology class of the
gerbal representation $\calC_{\bbL}^{\on{ss}}$ directly seems like a
daunting, if not impossible, task. In \S \ref{coh cl} we devise a
different method. Namely, we develop the formalism of what we call
``gerbal pairs of groups'' which allows us to calculate in a regular
way the third cohomology classes of gerbal representations in some
situations. We then apply this formalism to a natural gerbal pair
associated to the group $GL_{\infty,\infty}$. This alows us to
calculate the desired cohomology class of the gerbal representation
$\calC_{\bbL}^{\on{ss}}$. At the end of \S \ref{coh cl} we discuss
representations of the group $\wh{GL}_{\frakf,\infty}$ on modules over
a completion of the Clifford algebra $\Cl_{\calO_\bbK}$, which may be
of independent interest.

\medskip

\noindent{\bf Acknowledgments.} We thank D. Ben-Zvi, B. Feigin, D.
Gaitsgory, V. Kac, D. Kazhdan and B. Tsygan for useful
discussions. This paper was finished while E.F. visited
Universit\'e Paris VI as Chaire d'Excellence of Fondation des
Sciences Math\'ematiques de Paris. He thanks the Foundation for
its support and the group ``Algebraic Analysis'' at Universit\'e
Paris VI, and especially P. Schapira, for hospitality.

\section{1-dimensional story: Projective
  representations}\label{central extensions}

In this section, we review the constructions of central extensions
of the ``master'' group $GL_\infty$ and the Lie algebra
$\gl_\infty$ and their natural representations from different
points of view. This objects are ``1-dimensional'' in the sense
that they are attached to the algebra $\C\ppart$ of formal Laurent
power series in one variable. This section serves as a motivation
of the main constructions of this paper concerning the
``2-dimensional'' case, when $\C\ppart$ is replaced by power
series in two variables. However, we will only consider the
``2-dimensional" theory for groups in this paper and leave the Lie
algebra case to \cite{next}.

\subsection{Homological constructions}

We first explain the construction for Lie algebras.

\subsubsection{Tate vector spaces} \label{Tate}

Let $K=\bbC\ppart$, regarded as a topological $\bbC$-vector space,
with the usual $t$-adic topology. This is an example of the
so-called Tate vector space.

We recall the following standard definitions (see, e.g.,
\cite{Kap}).

\begin{dfn} A topological vector space
is called {\em linearly compact} if it is the topological dual of
a discrete vector space. A topological space is called {\em
linearly locally compact}, or {\em Tate}, if it admits a basis of
neighborhoods of 0 consisting of linearly compact subspaces. A
{\em lattice} in a Tate vector space $V$ is a linearly compact
open subspace of $V$.
\end{dfn}

Any two lattices $L_1, L_2$ in a Tate vector space are {\em
commensurable} with each other; that is, the quotients $L_1/(L_1
\cap L_2)$ and $L_2/(L_1 \cap L_2)$ are finite-dimensional.

Tate vector spaces form a category, whose Hom's are the continuous
linear maps. This is an exact category in the sense of Quillen
(see \cite{Quillen} \S 2). A sequence
\[0 \to V' \to V \to V'' \to 0\]
is an exact sequence in this category if it is an exact sequence
of vector spaces, $V'\to V$ is a closed embedding, and $V\to V''$
is an open surjective map. In the standard terminology of exact
categories, $V'\to V$ is called admissible monomorphism and
denoted by $V'\rightarrowtail V$, and $V\to V''$ is called
admissible epimorphism and denoted by $V\twoheadrightarrow V''$.
Note that in this case the topology on $V''$ coincides with the
quotient topology.

\begin{rmk}\label{exact
str} We could identify the category of Tate vector spaces as a
subcategory of pro-vector spaces by assigning $V$ the projective
system $\{V/V_\alpha\}$, where $V_\alpha$ range over the set of
open subspaces of $V$. Then a sequence $V'\rightarrowtail
V\twoheadrightarrow V''$ is exact if and only if there is an index
set $I$ such that $V=\lim\limits_{\overleftarrow{\alpha\in
I}}V_\alpha$ (resp. $V'=\lim\limits_{\overleftarrow{\alpha\in
I}}V'_\alpha$. resp. $V''=\lim\limits_{\overleftarrow{\alpha\in
I}}V''_\alpha$), and for each $\alpha\in I$, the sequence
$V'_\alpha\to V_\alpha\to V''_\alpha$ is exact. Therefore, we see
that if $V$ has a countable basis of neighborhood 0, then any
short exact sequence $V'\rightarrowtail V\twoheadrightarrow V''$
splits.
\end{rmk}

We remark that the following discussion remains unchanged if we
replace $K$ by any other non-compact Tate vector space.

\subsubsection{The Lie algebra $\gl_\infty$}

Let $\End K$ be the algebra of continuous endomorphism of $K$ and
$\gl_\infty$ the associated Lie algebra. It is well-known that
$H^2(\gl_\infty)$ has a non-zero class $[c_2]$ and therefore there is
a non-trivial $\bbC$-central extension of $\gl_\infty$ corresponding
to this class. Let us review the construction of this central
extension following Kac--Peterson \cite{KP} and Arbarello--De
Concini--Kac \cite{ACK}.

Recall that a lattice in $K$ is a compact open subspace of $K$.
For example, $\calO_K=\bbC[[t]]\subset K$ is a lattice. Let
$\gl^+_\infty$ be the Lie algebra of continuous endomorphisms of
$\calO_K$, and $\gl_\frakf$ the two-sided ideal of discrete
endomorphisms, i.e., endomorphisms which have open kernels.
Observe that there is a canonical trace functional
$\tr:\gl_\frakf\to\bbC$.

We write $K=\calO_K\oplus\calO^-$, where $\calO^-$ is a discrete
vector space which could be chosen as $t^{-1}\bbC[t^{-1}]$, and
denote by $\pi:K\to\calO_K$ the projection onto the first factor.
The projection induces a map $\gl_\infty\to\gl^+_\infty$, which is
also denoted by $\pi$, as follows: for $g\in\gl_\infty$, define
$\pi(X)\in\gl_\infty^+$ by $\pi(X)v=\pi(X(v))$ for any
$v\in\calO_K$. Observe that $\pi:\gl_\infty\to\gl^+_\infty$ is not
a Lie algebra homomorphism.

We also introduce
\[\widetilde{\gl}_\infty=\{(A,X)\in\gl^+_\infty\times \gl_\infty \, |
\, A-\pi(X)\in \gl_\frakf\}.\] One then has the following exact
sequence of associative algebras (and therefore Lie algebras)
\begin{equation}\label{e.s.of
Lie algebras}
0\to\gl_\frakf\stackrel{i}{\rightarrow}\widetilde{\gl}_\infty
\stackrel{p}{\rightarrow}\gl_\infty\to 0,
\end{equation}
where $i(A)=(A,0)$ and $p(A,X)=X$. One also has a section of $p$
given by $X\mapsto(\pi(X),X)$.

Since $\tr:\gl_\frakf\to\bbC$ is a Lie algebra homomorphism, one
can push out the exact sequence above to get a central extension
of $\gl_\infty$,
\begin{equation}
0\to\bbC\to\wh{\gl}_\infty\to\gl_\infty\to 0.
\end{equation}
It is clear that this central extension does not depend on the
choice of a lattice $\calO_K$ and the splitting
$K=\calO_K\oplus\calO^-$. It is the universal central extension of
$\gl_\infty$. Observe that there is a spectral sequence of Lie
algebra cohomology associated to (\ref{e.s.of Lie algebras}), and
the above construction shows that the class $[c_2]\in
H^2(\gl_\infty)$ corresponding to this central extension is
obtained by $[\tr]\in H^1(\gl_\frakf)$ by transgression.

\subsubsection{The group $GL_\infty$}

We will denote by $GL_\infty$ the group of continuous
automorphisms of $K$. Then $\gl_\infty$ could be viewed as the Lie
algebra of $GL_\infty$. In fact, $GL_\infty$ is a group space
(i.e., a sheaf of groups over $(\mathbf{Aff}/\bbC)_{fppf}$) which
assigns to every commutative $\bbC$-algebra $R$ the group of
continuous automorphisms of $R\ppart$. We will return to this
algebro-geometrical structure of $GL_\infty$ in \S \ref{nonlinear
version}. For the moment, we just regard $GL_\infty$ as an
abstract group.

Like in the case of Lie algebras, $H^2(GL_\infty,\bbC^\times)$ has
a non-trivial class which is obtained by transgression from
$[\det]\in H^1(GL_\frakf,\bbC^\times)$ (what this means precisely
is explained below). However, unlike the case of Lie algebra, it
is difficult (and may be impossible) to write down the cocycle
explicitly.

We will imitate the case of Lie algebras to define this central
extension, as in \cite{PS}, Chapter 6. So we let $GL_\infty^+$ be
the group of continuous automorphisms of $\calO_K$ and $GL_\frakf$
be the normal subgroup consisting of automorphisms fixing some
open subspaces of $\calO_K$. Remark that we have a canonical
homomorphism $\det:GL_\frakf\to\bbC^\times$. Furthermore, let
$\widetilde{GL}_\infty$ be the group of invertible elements in
$\widetilde{\gl}_\infty$ (regarded as an associative algebra).
Unlike the case of Lie algebras, where we have a short exact
sequence, here we only obtain a left exact sequence of groups
\[1\to
GL_\frakf\stackrel{i}{\to}\widetilde{GL}_\infty
\stackrel{p}{\to}GL_\infty,\] where
$i:GL_\frakf\to\widetilde{GL}_\infty$ is given by $i(a)=(a,1)$.
(We recall that $p(a,g)=g$.)

There is a surjective group homomorphism $\deg:GL_\infty\to\bbZ$
defined as follows: for any $g\in GL_\infty$,
$$
\deg(g)=\dim(\frac{\calO_K}{\calO_K\cap
g\calO_K})-\dim(\frac{g\calO_K}{\calO_K\cap g\calO_K})
$$
(see Proposition \ref{degree map} for a more general discussion).

Denote $GL_\infty^0=\ker(\deg)$. Then we have
\begin{equation}\label{e.s.}
1\to GL_\frakf\to\widetilde{GL}_\infty\to GL_\infty^0\to 1
\end{equation}
(See Proposition \ref{e.s. of groups} for the general case.)
Pushing-out this sequence by $\det$, we obtain a central extension
of $GL_\infty^0$,
\begin{equation}\label{central extension of
GLinfty0} 1\to\bbC^\times\to\widehat{GL}_\infty^0\to
GL_\infty^0\to 1.
\end{equation}

As in the case of Lie algebras, this central extension does not
depend on the choice of the splitting. There is a spectral
sequence associated to (\ref{e.s.}). According to the
construction, the cohomology class in
$H^2(GL_\infty^0,\bbC^\times)$ corresponding to the central
extension (\ref{central extension of GLinfty0}) comes by
transgression from $[\det]\in H^1(GL_\frakf,\bbC^\times)$ (which
is non-trivial, see \S \ref{H2 and central extensions}). We denote
it by $[C^0_2]$ in what follows.

We would like to extend the above central extension to a central
extension of $GL_\infty$.

\begin{prop}\label{uniqueness of central extension}
There is a unique (up to isomorphism) central extension
\begin{equation}\label{central
extension of GLinfty} 1\to\bbC^\times\to\widehat{GL}_\infty\to
GL_\infty\to 1
\end{equation}
whose restriction to $GL_\infty^0$ is (\ref{central extension of
GLinfty0}).
\end{prop}

\begin{proof} Observe that $GL_\infty$ is the
semi-direct product of $GL_\infty^0$ and $\bbZ$. Choose any
element of degree one, e.g., $\sigma\in GL_\infty$ sending $t^i
\mapsto t^{i+1}$ if we identify $K$ with $\bbC\ppart$. Then
$\langle\sigma\rangle\subset GL_\infty$ splits the map $\deg$. The
Lyndon-Hochschild-Serre spectral sequence implies that
$H^2(GL_\infty,\bbC^\times)\cong
H^2(GL_\infty^0,\bbC^\times)^{\langle\sigma\rangle}$, where
$\sigma$ acts on $GL_\infty^0$ by conjugation. Indeed, observe
that $H^1(GL_\infty^0,\bbC^\times)=0$ since there is an injection
$H^1(GL_\infty^0,\bbC^\times)\to
H^1(\widetilde{GL}_\infty,\bbC^\times)$, and Proposition
\ref{group acyclic} claims that
$H^1(\widetilde{GL}_\infty,\bbC^\times)=0$. On the other hand,
$H^i(\bbZ,\bbC^\times)=0$ for $i\geq 2$.

We claim that the class $[C^0_2]$ is $\sigma$-invariant, hence in
$H^2(GL_\infty^0,\bbC^\times)^{\langle\sigma\rangle}$, and
therefore the proposition follows.

To prove this, we only need to give an automorphism
$\wh{\sigma}:\widehat{GL}_\infty^0\to\widehat{GL}_\infty^0$
lifting the automorphism $\sigma: GL_\infty^0\to GL_\infty^0, b\to
\sigma b\sigma^{-1}$. In fact, we will give a group homomorphism
$\tilde{\sigma}:\widetilde{GL}_\infty\to\widetilde{GL}_\infty$
which lifts $\sigma$.

Namely, we choose a splitting $\calO_K=\sigma\calO_K\oplus\bbC$.
For any $a\in GL^+_\infty$, let $a_\sigma=\sigma
a\sigma^{-1}\oplus\mathrm{id}:\sigma\calO_K\oplus\bbC\to\sigma\calO_K\oplus
\bbC$. Then set $\tilde{\sigma}(a,g)=(a_\sigma,\sigma
g\sigma^{-1})$.

It is clear that $\tilde{\sigma}$ is an injective homomorphism
$\widetilde{GL}_\infty\to\widetilde{GL}_\infty$ and gives an
injective homomorphism
$\wh{\sigma}:\widehat{GL}_\infty^0\to\widehat{GL}_\infty^0$.
Although $\tilde{\sigma}$ is not a group automorphism (since it is
not surjective), $\wh{\sigma}$ is surjective and gives a group
automorphism of $\widehat{GL}_\infty^0$ covering $\sigma$. This is
because for any $a\in GL_\infty^+$ one can find $e,e'\in
GL_\frakf, \det(e)=\det(e')=1$ such that $eae'=a'_\sigma$ for some
$a'\in GL_\infty^+$ (by means of elementary transformations).
Therefore, any element in $\widehat{GL}_\infty^0$ may be
represented by an element
$(a_\sigma,b)=\tilde{\sigma}(a,\sigma^{-1}b\sigma)\in
\widetilde{GL}_\infty$. This completes the proof that $[C^0_2] \in
H^2(GL_\infty^0,\bbC^\times)^{\langle\sigma\rangle}$.

The corresponding central extension (\ref{central extension of
GLinfty}) is just obtained by forming the semi-direct product
$\widehat{GL}_\infty=\widehat{GL}_\infty^0\rtimes\langle\sigma\rangle$,
where $\sigma$ acts on $\widehat{GL}_\infty^0$ via $\wh{\sigma}$
constructed above.
\end{proof}

\begin{rmk} One can show that there is
\emph{no} automorphism
$\tilde{\sigma}:\widetilde{GL}_\infty\to\widetilde{GL}_\infty$
covering $\sigma$. Therefore, there is \emph{no} extension of
$GL_\infty$ by $GL$ whose restriction to $GL_\infty^0$ gives
(\ref{e.s.}).
\end{rmk}

Observe that while the automorphism group of the sequence
(\ref{central extension of GLinfty0}) is trivial, the automorphism
group of (\ref{central extension of GLinfty}) is
$H^1(GL_\infty,\bbC^\times)=\bbC^\times$.

\subsection{Geometric
constructions: non-linear version} \label{nonlinear version}

In the previous section we have constructed central extensions of
$GL_\infty$ and $\gl_\infty$ corresponding to given cohomology
classes by homological methods. While the homological construction
tells us where the cohomology classes come from, it will be more
meaningful if we could see such central extensions come into life
more naturally. We will justify the ``naturality'' in this section
from the geometric point of view and in the next section from the
algebraic point of view.

\subsubsection{The infinite
Grassmannian}    \label{1-Gr}

Recall that $K$ is a Tate vector space. The infinite Grassmannian
(i.e. the Sato Grassmannian) is defined as the moduli space of
lattices in $K$. More precisely, for any commutative
$\bbC$-algebra $R$, we set
\[\Gr(R)=\left\{\begin{array}{l}L\subset
R\wh{\otimes}K \mbox{ a } R\mbox{-submodule},
t^N(R\wh{\otimes}\calO_K)\subset L\subset
t^{-N}(R\wh{\otimes}\calO_K), \\ \mbox{ for some } N \in \Z_+,
L/t^N(R\wh{\otimes}\calO_K) \mbox{ is a projective }
R\mbox{-module}\end{array}\right\}.\]

We show in \S \ref{det lb} below that the corresponding functor is
represented by an ind-scheme over $\bbC$. Denote by $L_0$ the
standard lattice $\calO_K \subset K$. The virtual dimension of $L$
is defined as
\[\dim L=\dim\frac{L}{L\cap
L_0}-\dim\frac{L_0}{L\cap L_0}.\] Then the connected components of
$\Gr$ are labeled by the virtual dimensions. Namely, we set
\[\Gr^n(\bbC)=\{L\in \Gr(\bbC), \dim L=n\}.\]
Then
\[\Gr=\bigsqcup_{n \in
\Z} \Gr^n. \]

The group $GL_\infty$ acts on $\Gr(\bbC)$ by sending $L$ to $gL$
for any $g\in GL_\infty$. This action is transitive. (More
generally, the group space $GL_\infty$ acts on $\Gr$, and $\Gr$ is
a homogeneous space of $GL_\infty$.) It follows from the
definition that $GL_\infty^0$ fixes each component, and $\sigma$
sends $\Gr^n$ to $\Gr^{n-1}$. Similarly, $\gl_\infty$ also acts on
$\Gr$, i.e., each element $X\in\gl_\infty$ gives rise a vector
field on $\Gr$.

A great discovery of the Sato school was that the action of an
infinite-dimensional abelian Lie subalgebra of $\gl_\infty$ gives
rise to the KP hierarchy \cite{DJKM,SW}, which explains the
importance of the infinite Grassmannian in soliton theory.

\subsubsection{The determinant line bundle} \label{det lb}

Let $R$ be a commutative $\bbC$-algebra. Given $L\in\Gr^n(R)$,
one could find $N \in \Z_+$ such that
$t^N(R\wh{\otimes}\calO_K)\subset L \subset
t^{-N}(R\wh{\otimes}\calO_K)$. Then $M =
L/t^N(R\wh{\otimes}\calO_K)$ defines a projective $R$-module of
rank $(N+n)$ in $$R^{2N}\cong
t^{-N}(R\wh{\otimes}\calO_K)/t^N(R\wh{\otimes}\calO_K).$$ This way,
we obtain that
\[\Gr^n=\lim\limits_{\stackrel{\longrightarrow}{N}}G(N+n,2N),\]
where $G(k,n)$ is the usual Grassmannian of $k$-planes in the
$n$-dimensional vector space. The closed embedding $G(N+n,2N) \to
G(N+n+1,2N+2)$ is obtained by sending $M \subset R^{2N}$ to $M
\oplus R \subset R^{2N+2}$. This gives $\Gr^n$ the structure of an
ind-scheme.

The Grassmannian $G(k,n)$ carries a determinant  line bundle,
which assigns to each $k$-plane the line of its top exterior
power.  This is the negative generator of the Picard group of
$G(k,n)$ (recall that $\pic(G(k,n)\cong\bbZ$). Under the embedding
$G(N+n,2N)\to G(N+n+1,2N+2)$, the determinant  line bundle on
$G(N+n+1,2N+2)$ restricts to the determinant line bundle on
$G(N+n,2N)$. Therefore, $\pic(\Gr^n)\cong\bbZ$ and there is a
determinant line bundle on $\Gr^n$, which is the negative
generator of its Picard group, and restricts to the determinant
line bundle on each $G(N+n,2N)$. We denote by $\calL$ the line
bundle on $\Gr$ that restricts to the determinant line bundle on
each connected component. We call $\calL$ the determinant line
bundle on $\Gr$.

More invariantly, if $L,L'\in\Gr(R)$ are two lattices, then there
is a positive integer $N$ such that $L,L'\supset
t^N(R\wh{\otimes}\calO_K)$ and both $L/(R\wh{\otimes}\calO_K),
L'/(R\wh{\otimes}\calO_K)$ are projective. Then one defines an
invertible $R$-module by
\begin{equation}\label{determinant}
\det(L|L'):=\bigwedge{}^{\mathrm{top}}
(L/t^N(R\wh{\otimes}\calO_K))\otimes(\bigwedge{}^{\mathrm{top}}
(L'/t^N(R\wh{\otimes}\calO_K))^{-1}.
\end{equation}

It is clear that this $R$-module is independent of the choice of
$N$ up to a canonical isomorphism. Therefore, we obtain a line
bundle over $\Gr\times\Gr$. If we identify
$\Gr\cong\Gr\times\{L_0\}\hookrightarrow\Gr\times\Gr$, the above
line bundle restricts to the determinant line bundle on $\Gr$.

\begin{rmk}\label{composition of determinantal lines} The
important feature of the determinant lines is that for
$L,L',L''\in\Gr(R)$, there is a canonical isomorphism
\begin{equation}\label{tensor of determinantal lines}
\gamma_{L,L',L''}:\det(L|L')\otimes\det(L'|L'')\cong\det(L|L'')
\end{equation}
such that for any $L,L',L'',L'''\in\Gr(R)$,
$\gamma_{L,L',L'''}\gamma_{L',L'',L'''}=\gamma_{L,L'',L'''}
\gamma_{L,L',L''}$.
\end{rmk}

In what follows, $\det(L|L')^\times$ will denote the set of
nowhere vanishing sections of $\det(L|L')$, regarded as a line
bundle on $\spec R$.

\subsubsection{Central
extensions}    \label{hat}

Now, we regard $GL_\infty$ as a group space. Let $\calL$ be the
determinant line bundle on $\Gr$. Since $\pic(\Gr^n)\cong\bbZ$,
for any $g\in GL_\infty(\bbC)$, we have $g^*\calL\cong\calL$. So
it is natural to ask whether the action of $GL_\infty(\bbC)$ on
$\Gr(\bbC)$ could be lifted to an action on $\calL$. That is,
whether we can choose isomorphisms $i_g:g^*\calL\cong\calL$,
such that $i_{gh}=i_h\circ h^*(i_g):(gh)^*\calL\cong\calL$.
Unfortunately (or fortunately!), this is not the case.

There is a canonical $\bbG_m$-central extension $\widehat{GL}'_\infty$
of $GL_\infty$ which tautologically acts on $\calL$ (see, e.g.,
\cite{Kap}). The $R$-points of $\widehat{GL}'_\infty$ are pairs
$(c,g)$ with $g\in GL_\infty(R)$ and $c\in\det(g L_0|L_0)^\times$. The
multiplication is $(c,g)(c',g')=(cg(c'),gg')$. Here
$g(c)\in\det(gg'L_0|gL_0)^\times$, and therefore by (\ref{tensor of
determinantal lines})
$$
cg(c')\in\det(gL_0|L_0)^\times\otimes(gg'L_0|gL_0)^\times\cong
\det(gg'L_0|L_0)^\times.
$$
Furthermore, Remark \ref{composition of determinantal lines}
guarantees that the group structure of $\wh{GL}'_\infty$ is
well-defined.

The natural action of $\widehat{GL}'_\infty$ on $\calL$ is defined
as follows. We observe that the diagonal action of $GL_\infty$ on
$\Gr\times\Gr$ lifts to an action on the line bundle over it
simply via the canonical identification $$ \det(L|L') =
\det(gL|gL'). $$ Observe that the fiber of $\calL$ over
$L\in\Gr(R)$ is $\calL_L=\det(L|L_0)$. Therefore, for given
$(c,g)\in \widehat{GL}'_\infty(R)$,
\[\calL_L=\det(L|L_0)=\det(gL|gL_0)\stackrel{\otimes
c}{\to}
\det(gL|gL_0)\otimes\det(gL_0|L_0)\cong\det(gL|L_0)=\calL_{gL}\]
defines an action of $\widehat{GL}'_\infty$ on $\calL$.

\begin{prop}\label{identification
of two central extensions} The group of $\bbC$-points of the
central extension $\widehat{GL}'_\infty$ constructed above is
isomorphic to the extension (\ref{central extension of GLinfty}).
\end{prop}

\begin{proof} By
Proposition \ref{uniqueness of central extension}, it is enough to
prove that the above central extension, when restricted to
$GL_\infty^0$, is isomorphic to (\ref{central extension of
GLinfty0}). Recall the definition of $\widetilde{GL}_\infty$. We
will show that there is a homomorphism
$\widetilde{GL}_\infty\to\widehat{GL}'_\infty$ sending
$(a,g)\in\widetilde{GL}_\infty$ to $(c,g)\in\widehat{GL}'_\infty$,
where $g\in GL_\infty^0$, $a\in GL_\infty^+$ and
$c\in\det(gL_0|L_0)^\times$, such that if $(a,g)=(a,1)\in
GL_\frakf$ and $\det(a)=1$, then $c=1$. It is clear that this
implies that the "neutral" component of
$\widehat{GL}'_\infty(\bbC)$ is obtained by pushing-out
$\widetilde{GL}_\infty$ via $\det$ and therefore is isomorphic to
(\ref{central extension of GLinfty0}).

Indeed, regard $a\in GL_\infty^+$ as a continuous automorphism of
$L_0=\calO_K$. We know that there is a sublattice $L\subset
L_0\cap g^{-1}L_0$ such that $a|_L=\pi(g)|_L$ since
$a-\pi(g)\in\gl_\frakf$.  Since $g(L)\subset L_0$,
$\pi(g)|_L=g|_L$. Therefore $a(L)=g(L)$ is a lattice in $L_0$.
Then we obtain the following isomorphism
\[\frac{gL_0}{g(L)}\stackrel{g^{-1}}{\longrightarrow}\frac{L_0}{L}
\stackrel{a}{\longrightarrow}\frac{L_0}{a(L)}=\frac{L_0}{g(L)}.\]
The top exterior power of the above isomorphism gives us the
desired element
\[c\in\Hom(\bigwedge{}^{\mathrm{top}}(g(L_0)/g(L)),
\bigwedge{}^{\mathrm{top}}(L_0/g(L)))\cong\det(gL_0|L_0)^\times,\]
and it is clear that if $g=1$ (then $a\in GL_\frakf$), then $c$
constructed above is just
$\det(a)\in\det(L_0|L_0)^\times=\bbC^\times$. This completes the
proof of the proposition.
\end{proof}

So in what follows we will not distinguish between
$\widehat{GL}_\infty$ and $\widehat{GL}'_\infty$.

\subsubsection{The
gerbe of determinantal theories}    \label{The
    determinantal gerbes}

We rephrase the above construction, using the determinantal gerbe
$\calD_K$ associated to the Tate vector space $K$. We first recall
some basic facts about the determinantal gerbe associated to a
Tate vector space here for reader's convenience. The following
definition is due to M. Kapranov \cite{Kap}.

\begin{dfn}
\label{def det th} Let $V$ be a Tate vector space. A {\em
determinantal theory} $\Delta$ on $V$ is a rule that assigns to
every lattice $L\subset V$ a line $\Delta(L)$ and to every
$L_1\subset L_2$ an isomorphism
\[\Delta_{L_1,L_2}:\Delta(L_1)\otimes\det(L_2/L_1)\to\Delta(L_2)\]
such that for any $L_1\subset L_2\subset L_3$ three lattices, the
obvious diagram
\[\begin{CD}
\Delta(L_1)\otimes\det(L_2/L_1)\otimes\det(L_3/L_2)@>>>
\Delta(L_1)\otimes\det(L_3/L_1)\\
@VVV@VVV\\
\Delta(L_2)\otimes\det(L_3/L_2)@>>>\Delta(L_3)
\end{CD}\]
is commutative. The category whose objects are determinantal
theories on $V$ and morphisms are isomorphisms of determinantal
theories (defined in the obvious way) is a $\bbC^\times$-gerbe,
called the {\em determinantal gerbe} of $V$ and denoted by
$\calD_V$.
\end{dfn}

There is another $\bbC^\times$-gerbe that is described in
\cite{ACK}. Namely, objects are lattices in $V$ and
$\Hom(L,L')=\det(L|L')^\times$. It is clear that the latter gerbe
is equivalent to the former under the functor that sends a lattice
$L\subset V$ to the determinantal theory $\Delta$ which assigns
$\Delta(L)=\bbC$ (this assignment uniquely determines $\Delta$).

The following lemma is the content of \cite{Kap}, Proposition
1.4.5, and \cite{AK}, Lemma 7. Let $A$ be an abelian group. Recall
that if $\calF$, $\calF'$ are two $A$-gerbes, the tensor product
$\calF\otimes\calF'$ is also an $A$-gerbe, defined as follows. The
objects in $\calF\otimes\calF'$ are $(x,x')$ with $x$ an object in
$\calF$ and $x'$ an object in $\calF'$. The morphisms between
$(x,x')$ and $(y,y')$ are
\[\Hom_{\calF\otimes\calF'}((x,x'),(y,y')):=\Hom_\calF(x,y)\otimes_A
\Hom_{\calF'}(x',y').\]

\begin{lem}\label{determinantal
gerbe} Let $V'\rightarrowtail V$ be an admissible monomorphism.
Then there is a canonical equivalence of $\bbC^\times$-gerbes
\[\calD_{V'}\otimes\calD_{V/V'}\cong\calD_{V}\]
and for $V_1\rightarrowtail V_2\rightarrowtail V_3$ we have a
natural transformation
\[\xymatrix{\calD_{V_1}\otimes\calD_{V_2/V_2}\otimes
\calD_{V_3/V_2}\ar[r]\ar[d]&\calD_{V_1}\otimes\calD_{V_3/V_1}
\ar[dl]\ar[d]\\
\calD_{V_2}\otimes\calD_{V_3/V_2}\ar[r]&\calD_{V_3}\\
}\] such that if given $V_1\rightarrowtail V_2\rightarrowtail
V_3\rightarrowtail V_4$, the cubical diagram of natural
transformations commutes.
\end{lem}
Observe that if $\varphi:V\to V'$ is an isomorphism, we obtain a
natural equivalence of categories
$\calD_\varphi:\calD_V\cong\calD_{V'}$.

It would be possible to define $\calD_K$ as a sheaf of groupoids over
$\mathbf{Aff}/\bbC$ with appropriate topology, if one could make sense
of a family of Tate vector spaces (or Tate modules).  This has been
done by V. Drinfeld in \cite{Dr}, and it is a non-trivial fact from
\emph{loc. cit.} that the determinantal gerbe of a Tate $R$-module is
a $\bbG_m$-gerbe over $\spec R$ in Nisnevich topology (see
\emph{loc. cit.} \S3.6 and \S5.2).

Now, the group $GL_\infty$ acts on the gerbe $\calD_K$. Moreover,
there is a $GL_\infty$-equivariant covering $\Gr\to\calD_K$. At
the level of $\bbC$-points, this map sends $L \in \Gr(\bbC)$ to
the determinantal theory $\Delta$ which assigns $\Delta(L)=\bbC$.
Recall that we denote by $\calL$ the determinant line bundle on
$\Gr\times\Gr$. We will denote by $\mathrm{Tot}(\calL)^\times$ its
total space with zero section deleted. Then
\[\Gr\times_{\calD_K}\Gr\cong\mathrm{Tot}(\calL)^\times\]

In fact, we can recover the central extension of $GL_\infty$ just
from $\calD_K$. Given a map $L:\spec\bbC\to\calD_K$, we obtain a
morphism $act_L:GL_\infty\to\calD_K$. By \cite{AK}, Theorem 1,
which goes back to Brylinski,
$$\widehat{GL}_{\infty,L}:=\spec\bbC\times_{\calD_K}GL_\infty\to
GL_\infty$$ is a $\bbG_m$-central extension of $GL_\infty$. In
particular, if $L=L_0$,
$\widehat{GL}_{\infty,L_0}=\widehat{GL}_\infty$ is the central
extension we constructed in previous subsection.

\subsection{Algebraic constructions:
linear version}

At this point, it would be desirable to produce a representation
of the central extensions of $GL_\infty$ and $\gl_\infty$
constructed above, on which the central elements act
non-trivially. This is the fermionic Fock representation introduced
in \cite{KP}.

\subsubsection{The space of global sections
of the determinant line bundle}

Recall that the determinant line bundle $\calL$ on $\Gr$ is
$\widehat{GL}_\infty$-equivariant. We define
\begin{equation}
\label{geom Fock} \bigwedge := \Gamma(\Gr,\calL^*)^*.
\end{equation}
This is a representation of $\widehat{GL}_\infty$, on which the
central subgroup $\bbC^\times$ acts by the identity. We now give a
more concrete description of $\bigwedge$.

\subsubsection{Clifford modules}
\label{Clifford modules}

Let $K=\bbC\ppart$. This is a topological vector space (with the
usual $t$-adic topology) whose topological dual is (defined to be)
$K^*=\bbC\ppart dt$. There is a natural symmetric bilinear form on
$K\oplus K^*$ induced by the residue pairing. Let
$\Cl_K=\Cl(K\oplus K^*)$ be the corresponding completed
topological Clifford algebra. If we set $\phi_n=t^n$ and
$\phi_n^*=t^n\frac{dt}{t}$ (note that in this notation, the dual
of $\phi_n$ is $\phi^*_{-n}$!), then the algebra is topologically
generated by $\phi_n,\phi_m^*$ subject to the following relations
\[[\phi_n,\phi_m]_+=0, \qquad \
[\phi_n^*,\phi_m^*]_+=0, \qquad
[\phi_n,\phi_m^*]_+=\delta_{n,-m}.\]

Since $\calO_K\oplus\calO_K dt$ is a Lagrangian subspace of
$K\oplus K^*$, $\bigwedge(\calO_K\oplus\calO_K dt)$ is a
subalgebra of $\Cl_K$. Let $\ind_{\bigwedge(\calO_K\oplus\calO_K
dt)}^{\Cl_K}(\bbC|0\rangle)$ be the $\Cl_K$-module generated be
the (vacuum) vector $|0\rangle_{\calO_K}$, on which
$\bigwedge(\calO_K\oplus\calO_K dt)$ acts by 0. This is a discrete
$\Cl_K$-module, called the fermionic Fock module. The following
lemma can be proved by reducing to the finite-dimensional case.

\begin{lem}$\ind_{\bigwedge(\calO_K\oplus\calO_K
dt)}^{\Cl_K}(\bbC|0\rangle_{\calO_K})\cong\Gamma(\Gr,\calL^*)^*=
\bigwedge$.
\end{lem}

Thus, $\ind_{\bigwedge(\calO\oplus\calO
dt)}^{\Cl_K}(\bbC|0\rangle)$ is a concrete realization of the
representation $\bigwedge$ of $\widehat{GL}_\infty$.

\subsubsection{Another description of the action
of $\widehat{GL}_\infty$ on $\bigwedge$}    \label{another descr}

We may also discover this representation of $\widehat{GL}_\infty$
from the representation-theoretic point of view.  The starting
point is, as in the finite-dimensional case, the following
statement:

\begin{lem}\label{equi}
The category of discrete $\Cl_K$-modules $\calC_K$ is equivalent
to the category of vector spaces. Any non-zero discrete
irreducible $\Cl_K$-module is isomorphic to $\bigwedge$.
\end{lem}

\begin{proof} Consider the functor $F$ from the
category $\calC_K$ to the category of vector spaces taking a
$\Cl_K$-module $M$ to the space of invariants of the subalgebra
$\bigwedge(\calO_K\oplus \calO_K dt) \subset \Cl_K$ in $M$. Let
$G$ be the left adjoint functor which sends a vector space $V$ to
$V \otimes \bigwedge$, where $\Cl_K$ acts on the second factor. It
is easy to see that the space of $\bigwedge(\calO_K\oplus \calO_K
dt)$-invariants in $\bigwedge$ is spanned by the vacuum vector
$|0\rangle_{\calO_K}$. For this we identify $\bigwedge$ with
$\bigwedge(t^{-1}\C[t^{-1}] \oplus t^{-1}\C[t^{-1}]dt)$ and
construct a basis of monomials in the generators $\phi_n, n<0$,
and $\phi^*_n, n \leq 0$. One readily checks that a non-constant
linear combination of such monomials cannot possibly be
annihilated by $\phi_m, m\geq 0$, and $\phi^*_m, m>0$. This shows
that $F \circ G$ is the identity functor. Let us show that $G
\circ F$ is also isomorphic to the identity functor.

Thus, we have a homomorphism $F(M) \otimes \bigwedge \to M$ and we
have to prove that it is an isomorphism. First of all, $\bigwedge$
is an irreducible $\Cl_K$-module. To see that, we use the above
explicit realization to give $\bigwedge$ a $\Z$-grading such that
$\deg \phi_n = \deg \phi^*_n = -n$. Any non-zero submodule is
graded. This implies that it must be generated by a vector that is
invariant under $\bigwedge(\calO_K\oplus \calO_K dt) \subset
\Cl_K$, hence by $|0\rangle_{\calO_K}$. This implies that the map
$F(M) \otimes \bigwedge \to M$ is injective. Next, let us show
that if $M \neq 0$, then $F(M) \neq 0$. To see that, observe that
for any vector $v$ in a discrete $\Cl_K$-module $M$ there exists
$N \in \Z_+$ such that $\bigwedge(t^{N+1}\calO_K\oplus t^N\calO_K
dt) \cdot v = 0$. Consider the vector space $V_N = t^{-N}
\calO_K/t^{N+1} \calO_K$ and its dual vector space $V_N^* \simeq$
$t^{-N} \calO_K \frac{dt}{t}/t^{N+1} \calO_K \frac{dt}{t}$. Let
$\Cl_{V_N}$ be the Clifford algebra associated to $V_N \oplus
V_N^*$. Then $\Cl_{V_N} \cdot v \subset \Cl_K \cdot v $. It
follows from the theory of modules over finite-dimensional
Clifford algebras that the former has a non-zero vector invariant
under $\bigwedge(\calO_K/t^{N+1} \calO_K \oplus \calO_K dt/t^N
\calO_K dt)$. Viewed as a vector of $M$, it is then
$\bigwedge(\calO_K\oplus \calO_K dt)$-invariant, and so $F(M) \neq
0$. Now suppose that $F(M) \otimes \bigwedge \to M$ is not
surjective, and $M' \neq 0$ is the cokernel. Then $M'$ contains a
non-zero $\bigwedge(\calO_K\oplus \calO_K dt)$-invariant vector.
Reducing to finite-dimensional Clifford algebras, as above, we
obtain that it can be lifted to a $\bigwedge(\calO_K\oplus \calO_K
dt)$-invariant vector in $M$ itself. This is a contradiction,
which completes the proof.
\end{proof}

Thus, the category of $\Cl_K$-modules is very simple. On the other
hand, we can construct many "natural" non-trivial discrete
irreducible $\Cl_K$-modules. Namely, for any $L\in\Gr(\bbC)$,
$L\oplus L^\perp$ is a Lagrangian in $K\oplus K^*$ and therefore
$$
M_L:=\ind_{\bigwedge(L\oplus L^\perp)}^{\Cl_K}(\bbC|0\rangle_L)
$$
is a discrete irreducible $\Cl_K$-module. Therefore, all of them
are isomorphic to $M_{L_0}=\bigwedge$. We have the following
simple observation.

\begin{lem}\label{Hom set} For any $L,L'\in\Gr(\bbC)$, there is
a canonical isomorphism
\[\Hom(M_L,M_{L'})\cong\det(L|L').\]
such that for $L,L',L''$, the following diagram commutes.
\[\begin{CD}\Hom(M_L,M_{L'})@.\otimes@.\Hom(M_{L'},M_{L''})@>>>
\Hom(M_L,M_{L''})\\@VVV@. @VVV@VVV
\\
\det(L|L')@.\otimes@.\det(L'|L'')@>\cong
>>\det(L|L'')
\end{CD}\]
\end{lem}

\begin{proof}
By Schur's lemma, $\Hom(M_L,M_{L'})$ is one-dimensional. The
statement of the lemma identifies this one-dimensional space
canonically.

Since $M_L$ is generated by a vector $|0\rangle_L$ annihilated by
the subalgebra $\bigwedge(L\oplus L^\perp)$, defining a
homomorphism $M_L \to M_{L'}$ is equivalent to choosing a
$\bigwedge(L\oplus L^\perp)$-invariant vector in $M_{L'}$. Thus,
$\Hom(M_L,M_{L'})$ is canonically identified with the space of
$\bigwedge(L\oplus L^\perp)$-invariants in $M_{L'}$. The latter
module is, in turn, generated by a vector $|0\rangle_{L'}$
annihilated by the subalgebra $\bigwedge(L'\oplus L'{}^\perp)$.

Now, there is a positive integer $N$ such that
$$
t^{N} \calO_K \subset L,L' \subset t^{-N} \calO_K, \qquad t^{N}
\calO_K \frac{dt}{t} \subset L^\perp,L'{}^\perp \subset t^{-N}
\calO_K \frac{dt}{t}.
$$
We have the Clifford algebra $\Cl_{V_N}$ introduced in the proof
of Lemma \ref{equi} and for any subspace $U \subset V_N$ its
irreducible module $M_U$ defined in the same way as above. It is
clear that the space of $\bigwedge(L\oplus L^\perp)$-invariants in
$M_{L'}$ is equal to the space of $\bigwedge((L/t^N \calO_K)
\oplus (L/t^N \calO_K)^\perp)$-invariants in $M_{L'/t^N\calO_K}$.
The latter space is canonically identified with $\det(L|L')$ (see
formula (\ref{determinant})).
\end{proof}

More explicitly, these homomorphisms may be constructed as
follows. Suppose for simplicity that $L' \subset L$. Let us choose
a basis $\{ v_i \}_{i \in I}$, of $L/L'$. Choose a lifting
$\wt{v}_i$ of $v_i$ to $L \subset K$. Then the vector
\begin{equation}    \label{image}
\wedge_{i \in I} \wt{v}_i \vac_{L'} \in M_{L'}
\end{equation}
is independent of the choice of the liftings and is annihilated by
$\bigwedge(L\oplus L^\perp)$. All homomorphisms $M_L \to M_{L'}$
are defined by sending $\vac_L \subset M_L$ to a multiple of this
vector, and we have a canonical identification of
$\Hom(M_L,M_{L'})$ with $\on{det}(L/L') = \det(L|L')$.

\medskip

Now we define an action of $\wh{GL}_\infty$ on $\bigwedge$. Since
$GL_\infty$ acts on $K$ via continuous automorphisms, it acts on
$\Cl_K$ via continuous automorphisms as well.  However, this
action does not necessarily lift to an action of any given
$\Cl_K$-module. Rather, what we have is an action of $GL_\infty$
on the the category of discrete $\Cl_K$-modules (see Example
\ref{example of genuine action} for more details). Explicitly, for
any $g\in GL_\infty$ and any $\Cl_K$-module $M$, we obtain a new
$\Cl_K$-module $M^g$, whose underlying vector space is the same as
$M$, but for any $a\in\Cl_K, m\in M^g$, the action is given by the
formula $a\cdot m=g^{-1}(a)m$. Clearly then,
$\bigwedge^g=M_{gL_0}$. This module is isomorphic to $\bigwedge$.
However, there is no canonical isomorphism, and choosing
particular isomorphisms for different $g$ we necessarily obtain a
projective action of $GL_\infty$.

More precisely, we obtain a representation of
$\widehat{GL}_\infty$ by the formula
\begin{equation}
\label{cg}
(c,g):\bigwedge\stackrel{\mathrm{id}}{\to}\bigwedge\nolimits^g
\stackrel{c}{\to}\bigwedge.
\end{equation}
Here, $\mathrm{id}:\bigwedge\to\bigwedge^g$ is the identity map as
vector spaces, and
$$
c \in \Hom(\bigwedge\nolimits^g,\bigwedge)^\times =
\Hom(M_{gL_0},M_{L_0})^\times = \det(gL_0|L_0)^\times.
$$
Recalling the definition of the central extension $\wh{GL}'_\infty
= \wh{GL}_\infty$ given in \S \ref{hat}, we see this
representation of $\widehat{GL}_\infty$ on $\bigwedge$ is the same
as the one we have constructed in the previous section.

\section{Generalities on the actions of
groups on categories}\label{generalities}

In the previous section we have given a brief account of the
1-dimensional story. We explained how to construct representations
of the central extensions of the 1-dimensional ``master'' Lie
algebra $\gl_\infty$ and the corresponding group $GL_\infty$, both
algebraically and geometrically. Algebraically, they are realized
by using Fock representations of a Clifford algebra.
Geometrically, this representation may be obtained by using the
determinant line bundle on the infinite Grassmannian.

Now we begin the study of the 2-dimensional story. In the next
section we will define the corresponding ``master'' Lie algebra
$\gl_{\infty,\infty}$ and the group $GL_{\infty,\infty}$. Our goal
is to construct an abelian category on which $\gl_{\infty,\infty}$
and $GL_{\infty,\infty}$ act gerbally, realizing their non-trivial
third cohomology classes. (We will calculate the corresponding
cohomology class of $GL_{\infty,\infty}$ in \S \ref{classes
corr}.)

In this section we develop the necessary general formalism of
group actions on categories. There is a similar theory for gerbal
actions of Lie algebras on abelian categories, but it involves
more sophisticated machinery.  Roughly speaking, we need to
develop the theory of gerbal actions of groups in ``families'',
that is, over arbitrary bases. We may then define a gerbal action
of a Lie algebra as the action of the corresponding formal group.
This material will be discussed in the follow-up paper \cite{next}
where we will also give examples of gerbal representations of the
Lie algebra $\gl_{\infty,\infty}$.

\subsection{2-groups}\label{2-groups}

\subsubsection{Definition
of 2-groups}

We recall the definition of 2-groups. A good introduction for this
subject is \cite{BL}.

\begin{dfn}    \label{2-grp}
A 2-{\em group} is a monoidal groupoid $\calG$ such that the set of
isomorphism classes of objects of $\calG$, denoted by $\pi_0(\calG)$,
is a group under the multiplication induced from the monoidal
structure. Let $I$ denote the unit object of $\calG$. We set
$\pi_1(\calG)=\End_\calG I$.
\end{dfn}

In the literature, these objects often appear under different
names. For example, they are called weak 2-groups in \cite{BL},
and are called gr-categories in \cite{S}.

It is clear that any group (in the usual sense) can be regarded as
a 2-group with the trivial $\pi_1$. All 2-groups form a (strict)
2-category, with objects being 2-groups, 1-morphism being the
homomorphisms between 2-groups (i.e. monoidal functors), and
2-morphism being the monoidal natural transformations of monoidal
functors.

We recall that if the monoidal structure of a 2-group is upgraded
to a tensor category structure (i.e., there exists a commutativity
constraint whose square is the identity), then this 2-group is
called a Picard groupoid. Therefore, Picard groupoids should be
regarded as commutative 2-groups.

It is easy to prove that $\pi_1(\calG)$ is always an abelian
group. The following simple observation will be important to us.

\begin{lem}\label{action of G on A}
Given a 2-group $\calG$, there is a natural action of
$\pi_0(\calG)$ on $\pi_1(\calG)$ by automorphisms, i.e., there is
a natural group homomorphism
$\pi_0(\calG)\to\mathrm{Aut}(\pi_1(\calG))$.
\end{lem}
\begin{proof}Denote
$G=\pi_0(\calG)$ and $A=\pi_1(\calG)$. Observe that for any object
$x$ of $\calG$, $\End_\calG(x)$ is isomorphic to $A$ in two ways.
Namely, $l_x:A\to\End_\calG(x)$ is obtained by the canonical
isomorphism $I\otimes x\cong x$, and $r_x:A\to\End_\calG(x)$ by
$x\otimes I\cong x$. Therefore, for any $x$, we define
$\rho_x:=l^{-1}_x\circ r_x:A\to A$. By definition, we know that
for any $a\in A$, $\rho_x(a)$ is the unique element in $A$ such
that the following diagram is commutative:

\[\begin{CD}
x\otimes
I@>\cong>>x@<\cong<< I\otimes x\\
@VV1\otimes
aV@.@VV\rho_x(a)\otimes 1V\\
x\otimes I@>\cong>>x@<\cong<< I\otimes x
\end{CD}\]
Its uniqueness implies the following two properties which complete
the proof of the lemma.

(i) If $x\cong x'$, then $\rho_x=\rho_{x'}$;

(ii) $\rho_I=\mathrm{Id}$ and $\rho_x\rho_{x'}=\rho_{x\otimes
x'}$.
\end{proof}

\begin{rmk}\label{functorial}
(i)The above construction is functorial. Let   $F:\calG\to\calH$
be a homomorphism of 2-groups. Then it induces group homomorphisms
$F_0:\pi_0(\calG)\to \pi_0(\calH)$ and
$F_1:\pi_1(\calG)\to\pi_1(\calH)$. We have
$F(\rho_x(a))=\rho_{F(x)}(F(a))$.

(ii) It is known (in particular, from the unpublished thesis of
Grothendieck's student Sinh \cite{S}, see also \cite{BL}) that
2-groups with $\pi_0=G$ and $\pi_1=A$ are classified by $H^3(G,A)$
(the so-called Postnikov invariant). The above lemma is the first
step toward the construction of a class in $H^3(G,A)$ associated
to $\calG$. Although we will not use this statement in this paper
(so we do not make it precise), our construction will be closely
related to this statement, but from a different point of view.

(iii) As usual, we could work in any topos $\calT$ instead of the
category of sets. Then a sheaf of 2-group $\calG$ will be a stack,
such that for any $U\in\calT$, $\calG(U)$ is a 2-group and the
pullback functor respects to the monoidal structure (i.e. for
$f:V\to U$, $g:W\to V$, $f^*,g^*$ are monoidal functors, and the
canonical isomorphism $g^*\circ f^*\cong (f\circ g)^*$ are
monoidal natural transforms). Denote by $I_U$ the unit object in
$\calG(U)$. Observe that $U\mapsto\mathrm{End}_{\calG(U)}(I_U)$ is
a sheaf of abelian groups over $\calT$, which is denoted by
$\pi_1(\calG)$. However, $U\mapsto\pi_0(\calG(U))$ is usually only
a presheaf. We will denote its sheafification by $\pi_0(\calG)$
(so in general $\pi_0(\calG)(U)\neq\pi_0(\calG(U))$). This is a
sheaf of groups, called the coarse moduli of $\calG$. Remark (i)
shows that there is an action of $\pi_0(\calG)$ on $\pi_1(\calG)$.
If one regards $\pi_0(\calG)$ as a 2-group, then the natural
projection $\pi: \calG\to\pi_0(\calG)$ is a 2-group homomorphism.
\end{rmk}

\subsection{Central extensions}

We will discuss a special type of 2-groups which can be regarded
as the central extensions of groups by Picard groupoids. We will
confine ourselves to the situations that are needed in the
following. For a much more general treatment of extensions of
groups by 2-groups, see \cite{Br}. We will work in the topos of
sets.

Let $\calP$ be a Picard groupoid. Recall that a $\calP$-torsor
$\calQ$ (over a point) is a module category over $\calP$, i.e.,
there is a bifunctor $\otimes:\calP\times\calQ\to\calQ$ satisfying
the associativity constraint, such that the functor
$I\otimes\cdot$ is isomorphic to the identity functor and for any
$x\in\calQ$, the functor $\cdot\otimes x:\calP\to\calQ$ is an
equivalence of categories.

Let $A$ be an abelian group. Then it makes perfect sense to tensor
two $A$-torsors over $A$, which is again an $A$-torsor. This
tensor product makes the category of $A$-torsors a Picard
groupoid, denoted by $BA$. It is clear that a $BA$-torsor in the
standard terminology is just an $A$-gerbe.

We will call $\calQ$ a $\calP$-bitorsor if it is equipped with two
commuting $\calP$-torsor structures. See \cite{Br}, Definition
3.1.8 for the more general definition of bi-torsors of a 2-group.
The meaning of "commuting $\calP$-torsor structures" is spelled
out in (3.1.8.2)--(3.1.8.4) of \emph{loc. cit.}.

Now let $\calG$ be a 2-group. Denote
$A=\pi_1(\calG)=\mathrm{End}_\calG(I)$. We let $\calG_e$ be the
subgroupoid of objects in $\calG$ that are isomorphic to $I$.
Therefore, $\calG_e$ is a connected groupoid, and the functor
$x\mapsto \Hom_\calG(I,x), \calG_e\to BA$ is an equivalence of
tensor categories. Now for any $s\in\pi_0(\calG)$, let $\calG_s$
be the subgroupoid consisting of objects in the isomorphism class
$s$. Then the monoidal functor gives $\calG_s$ the structure of a
bi-torsor under $\calG_e$. Namely,
$l:\calG_e\otimes\calG_s\to\calG_s$ and
$r:\calG_s\otimes\calG_e\to\calG_s$.

\begin{lem} The following statements
are equivalent:

(1) The action of $\pi_0(\calG)$ on $\pi_1(\calG)$ is trivial;

(2) The monoidal functor $\calG_e\to\calG$ is central in the sense
of \cite{Bez}, Definition 1;

(3) The two $A$-gerbe structures on $\calG_s$ are the same.
\end{lem}

\begin{dfn}\label{2-group of central extensions}
If a 2-group $\calG$ satisfies the above conditions, then we call
it the {\em central extension of $G=\pi_0(\calG)$ by $BA$}.
\end{dfn}

\subsection{Actions of groups
on categories}\label{action of group on categories, the naive
version} We now begin to discuss actions of groups on categories.
Our approach in this section will be naive in the sense that we
will work mostly with groups of $\C$-points of algebraic groups.
In \cite{next} we will develop the theory in a way that will allow
us to work in families and take full advantage of the
algebro-geometric structures. This will allow us to introduce the
notion of a Lie algebra action on a category which will be
important for our purposes.

\subsubsection{The center}
We will use the following notation. If $\lambda:F\to G$ is a
morphism between two functors $F$ and $G$ acting from a category
$\calC$ to a category $\calC'$, we denote $\lambda_X:F(X)\to G(X)$
the specialization of $\lambda$ to $X$.

Now let $\calC$ be an abelian category. The center $\calZ(\calC)$
of $\calC$ is by definition the ring $\End\mathbf{1}_{\calC}$,
where $\mathbf{1}_{\calC}$ is the identity functor of $\calC$.
Thus, an element $a\in \calZ(\calC)$ assigns to every $X\in\calC$
a morphism $a_X\in\End_\calC X$ such that for any $f:X\to Y$,
$a_Y\circ f=f\circ a_X$. It is easy to see that $\calZ(\calC)$ is
in fact a commutative ring. For instance, if $\calC=A\Mod$, the
category of left-modules over a ring $A$, then
$\calZ(\calC)=Z(A)$, the center of $A$. For any $X,Y\in\calC$,
$\mbox{Hom}_{\calC}(X,Y)$ is a $\calZ(\calC)$-module.

Let $\calZ(\calC)^\times$ be the group of invertible elements in
$\calZ(\calC)$. This is an abelian group which may be defined for
any category as the automorphism group of the identity functor.

\subsubsection{The 2-group $\bbG\bbL(\calC)$}
\label{bbGbbL} Let $\calC$ be a $\bbC$-linear abelian category. We
denote by $\bbG\bbL(\calC)$ the category of $\bbC$-linear
auto-equivalences of $\calC$. By definition, the objects of
$\bbG\bbL(\calC)$ are $\bbC$-linear additive functors
$F:\calC\to\calC$ which are equivalences of categories. The morphisms
$\Hom_{\bbG\bbL(\calC)}(F,G)$ are the natural transformations from $F$
to $G$ which are isomorphisms. It is clear from the definition that
$\bbG\bbL(\calC)$ is a strict monoidal category. Furthermore,
$\bbG\bbL(\calC)$ is a 2-group, with
$\pi_1(\bbG\bbL(\calC))=\calZ(\calC)^\times$. (We will prove in
\cite{next} that $\bbG\bbL(\calC)$ may be regarded as the groupoid of
$\bbC$-points of a stack.)

Likewise, for any category $\calC$, not necessarily abelian, we
can define the 2-group $\bbA\bbU\bbT(\calC)$ of auto-equivalences
of $\calC$, whose objects are all auto-equivalences of $\calC$ and
morphisms are isomorphisms between auto-equivalences.  However,
this 2-group does not possess rich structure.

\subsubsection{Genuine actions}
Let $\calG$ be a 2-group.

\begin{dfn}\label{genuine action}
A {\em (genuine) representation} of $\calG$ on an abelian category
$\calC$ is a homomorphism of 2-groups $F:\calG\to\bbG\bbL(\calC)$.
Likewise, an action of $\calG$ on a general category is a
homomorphism of 2-groups $\calG\to\bbA\bbU\bbT(\calC)$.
\end{dfn}

If $\calG=G$ is just an ordinary group, then the above definition
gives us the usual notion of action of a group on categories (see
the Introduction).

\begin{ex}\label{example
of genuine action} Here is a basic example. Assume that $G$ acts
on a $\bbC$-algebra $R$ by automorphisms. Then it acts on the
category of (left) $R$-modules in the following way. Let
$m:R\otimes M\to M$ be a left $R$-module. Then define a new
$R$-module structure on $M$ by the formula
\[(F_g(m)(r,x)=m(g^{-1}r,x) \mbox{ for }
r\in R, x\in M.\] If $f:(\rho,M)\to (\rho',M')$ is a morphism
between $R$-modules, then define $F_g(f)=f$ as linear map between
underlying $\bbC$-vector spaces. Obviously, $F_g:\calC\to\calC$ is
a functor. It is easy to check that $F$ defines an action of $G$
on $\calC$.
\end{ex}

The genuine actions of $\calG$ on $\calC$ form a category. Namely,
$\Hom_{2\mbox{-Grp}}(\calG,\bbA\bbU\bbT(\calC))$. (We recall that
all 2-groups form a strict 2-category $2\mbox{-Grp}$.)

\subsubsection{Gerbal
actions} \label{gerbal actions} A genuine representation of a
group on a category is a categorical analogue of a representation
of a group on a vector space. We also want some categorical
analogue of the projective representations, which will be called
gerbal representations. We will only consider gerbal
representations of groups in this paper. In \cite{next} we will
develop the theory further and introduce the notion of a gerbal
representation of a Lie algebra.

\begin{dfn}\label{gerbe
action1} A {\em gerbal representation} of a group $G$ on an abelian
category $\calC$ is a homomorphism of groups
$F:G\to\pi_0(\bbG\bbL(\calC))$. Likewise, a gerbal action of $G$ on a
general category $\calC$ is a group homomorphism
$G\to\pi_0(\bbA\bbU\bbT(\calC))$.

Equivalently, this is an assignment to each $g\in G$ of an
auto-equivalence $F_g:\calC\to\calC$ such that
$F_e\cong\mathbf{1}_\calC$ and $F_gF_{g'}\cong F_{gg'}$.

A {\em homomorphism} of gerbal representations $\calC$ and $\calC'$ of
$G$ is a functor $H: \calC \to \calC'$ such that there exist
isomorphisms $H \circ F_g \simeq F_g \circ H$ for all $g \in G$. If
$H$ is an equivalence of categories, we call these representations
{\em equivalent}.
\end{dfn}

\begin{rmk}
One can give another definition, in which one also specifies, as
part of the data, isomorphisms between $F_gF_{g'}$ and $F_{gg'}$
for all $g,g' \in G$. We have discussed this definition in the
Introduction. The two definitions are essentially equivalent. As
an analogy, consider the notion of projective representation of a
group $G$ on a $\bbC$-vector space. This may be defined as a
homomorphism $G \to \on{PGL}(V)$ or as a rule that assigns to each
$g \in G$ an automorphism $T_g$ of $V$ such that $T_g T_h =
\alpha_{g,h} T_{gh}$, for some $\alpha_{g,h} \in \bbC^\times$. The
above definition of gerbal representation is an analogue of the
former, whereas the definition used in the Introduction is an
analogue of the latter. We find the above definition more
convenient and economical, because it avoids the data of the
isomorphisms which are in some sense redundant, as we will see
below.
\end{rmk}

\medskip

Given a gerbal action of $G$ on $\calC$, one obtains, by Lemma
\ref{action of G on A}, an action of $G$ on
$\pi_1(\GL(\calC))=\calZ(\calC)^\times$. More explicitly, the
homomorphism $\rho:G\to\mathrm{Aut}(\calZ(\calC)^\times)$ is
defined as follows: for any $g\in G$ and
$a\in\calZ(\calC)^\times$, \[\rho_g(a)_{F_g(X)}=F_g(a_X)\] for any
object $X$ in $\calC$. (It is easy to see that this condition
determines $\rho$ uniquely.) Observe that if $\calC$ is abelian
and this is a gerbal representation, the above formula in fact
defines an action of $G$ on $\calZ(\calC)$.

We also observe that for any two objects $x,y$ in
$\bbG\bbL(\calC)$ (or in $\bbA\bbU\bbT(\calC)$), $\Hom(x,y)$ is a
bi-pseudotorsor\footnote{We recall that, for a group $G$, a
$G$-pseudotorsor is a set that is either empty or is a torsor
under $G$.} under $\calZ(\calC)^\times$. Indeed, $\Hom_\calG(x,y)$
is a pseudo $\End(x)$-torsor under the left action and a pseudo
$\End(y)$-torsor under the right action. A prior, we obtain four
pseudo $\calZ(\calC)^\times$-torsor structures on $\Hom(x,y)$
since there are two isomorphisms
$l_x,r_x:\calZ(\calC)^\times\to\End(x)$ and two isomorphisms
$l_y,r_y:\calZ(\calC)^\times\to\End(y)$. However, it is clear that
the two pseudo $\calZ(\calC)^\times$-torsor structures obtained by
$l_x$ and $l_y$ in fact coincide. So do the other two. Therefore,
there are only two pseudo $\calZ(\calC)^\times$-torsor structures
on $\Hom(x,y)$, which furthermore commute with each other. (These
are the two $\calZ(\calC)^\times$-gerbe structures on
$\bbG\bbL(\calC)$ (and on $\bbA\bbU\bbT(\calC)$).) In what
follows, we will use the pseudo $\calZ(\calC)^\times$-torsor
structure on $\Hom(x,y)$ coming from $l_x$.

\subsubsection{Third cohomology class}
Let us pick an isomorphism
$c(g_1,g_2):F_{g_1}F_{g_2}\stackrel{\cong}{\longrightarrow}
F_{g_1g_2}$ for all pairs $g_1,g_2 \in G$. Then for
$g_1,g_2,g_3\in G$, there are two isomorphisms between
$F_{g_1}F_{g_2}F_{g_3}$ and $F_{g_1g_2g_3}$; namely,
\[F_{g_1}F_{g_2}F_{g_3}\stackrel{\cong}{\longrightarrow}
F_{g_1g_2}F_{g_3}\stackrel{\cong}{\longrightarrow}F_{g_1g_2g_3}\]
and
\[F_{g_1}F_{g_2}F_{g_3}\stackrel{\cong}{\longrightarrow}
F_{g_1}F_{g_2g_3}\stackrel{\cong}{\longrightarrow}F_{g_1g_2g_3}\]
Let $a(g_1,g_2,g_3)\in \calZ(\calC)^\times$ be the unique element
sending the first isomorphism to the second one since
$\Hom(F_{g_1}F_{g_2}F_{g_3},F_{g_1g_2g_3})$ is a torsor under
$\calZ(\calC)^\times$. That is, the following diagram is
commutative:
\[\begin{CD}
F_{g_1}F_{g_2}F_{g_3}(X)@>\cong>>F_{g_1g_2}F_{g_3}(X)@> \cong
>>F_{g_1g_2g_3}(X)\\
\parallel  & & &
&
@VVa(g_1,g_2,g_3)_{F_{g_1g_2g_3}(X)}V\\
F_{g_1}F_{g_2}F_{g_3}(X)@>\cong>>F_{g_1}F_{g_2g_3}(X)@>\cong>>
F_{g_1g_2g_3}(X)
\end{CD}\]
In this way, one defines a map $a:G\times G\times
G\to\calZ(\calC)^\times$.

\begin{thm}\label{H^3}
(i) $a$ is a cocycle, i.e.
\[\rho_{g_1}a(g_2,g_3,g_4)a(g_1,g_2g_3,g_4)a(g_1,g_2,g_3)=
a(g_1g_2,g_3,g_4)a(g_1,g_2,g_3g_4).\]

(ii) Given different isomorphisms
$c'(g_1,g_2):F_{g_1}F_{g_2}\stackrel{\cong}{\longrightarrow}
F_{g_1g_2}$, the new cocycle differs from the original one by a
coboundary, and therefore, there is a well-defined cohomology
class $[a]\in H^3(G,\calZ(\calC)^\times)$ associated to a gerbal
action of the group $G$ on the category $\calC$.

(iii) If this class is trivial, then the action $F$ could be
upgraded to a genuine action, i.e., one could choose
$c(g_1,g_2):F_{g_1}F_{g_2}\stackrel{\cong}{\longrightarrow}F_{g_1g_2}$
in such a way that the following diagram is commutative:
\[\begin{CD}
F_{g_1}F_{g_2}F_{g_3}@>\cong
>>F_{g_1g_2}F_{g_3}\\
@VV\cong V@VV\cong
V\\
F_{g_1}F_{g_2g_3}@>\cong>>F_{g_1g_2g_3}
\end{CD}\]
In this case, the set of isomorphisms $c(g_1,g_2)$ such that $F$
is a genuine action is a torsor under
$H^2(G,\calZ(\calC)^\times)$. Furthermore, the automorphism of
such $F$ is $H^1(G,\calZ(\calC)^\times)$. \end{thm}

\begin{rmk}

(i) There is a tautological gerbal action of the group
$\pi_0(\bbA\bbU\bbT(\calC))$ on $\calC$, which gives a cohomology
class $u\in H^3(\pi_0(\bbA\bbU\bbT(\calC)),\calZ(\calC)^\times)$.
This is the Postnikov invariant associated to the 2-group
$\bbA\bbU\bbT(\calC)$. Then the cohomology class $[a]$ in the
theorem is just the pullback of $u$.

(ii) Part (iii) of the theorem could be interpreted as follows: if
the class $[a]$ vanishes, then the group homomorphism
$F:G\to\pi_0(\bbA\bbU\bbT(\calC))$ can be lifted as a 2-group
homomorphism $\tilde{F}:G\to\bbA\bbU\bbT(\calC)$. The possible
liftings form a subcategory of
$\Hom_{2\mbox{-Grp}}(G,\bbA\bbU\bbT(\calC))$. The isomorphism
classes form a torsor under $H^2(G,\calZ(\calC)^\times)$. Given a
lifting $\tilde{F}$, the automorphism group of $\tilde{F}$ (i.e.,
invertible 2-morphisms between $\tilde{F}$ and itself) is
$H^1(G,\calZ(\calC)^\times)$.

(iii) The same remarks apply to $\bbG\bbL(\calC)$.\end{rmk}

\begin{proof}(i) is proved
by diagram chasing. Observe that the following diagram is always
commutative:

\[\begin{CD}
F_{g_1}F_{g_2}F_{g_3}F_{g_4}@>F_{g_1}F_{g_2}(c(g_3,g_4))>\cong>
F_{g_1}F_{g_2}F_{g_3g_4}\\
@V\cong Vc(g_1,c_2)V@V\cong
Vc(g_1,g_2)V\\
F_{g_1g_2}F_{g_3}F_{g_4}@>F_{g_1g_2}(c(g_3,g_4))>\cong>F_{g_1g_2}
F_{g_3g_4}
\end{CD}\]
To prove (i), it is enough to show that both sides are the same
when evaluated at $F_{g_1g_2g_3g_4}(X)$, for any object $X$ in
$\calC$. That is, to show that the rightmost loop in the following
diagram commutes. This follows from the fact that all other loops
in the diagram are commutative.

\[\xymatrix{F_{g_1}F_{g_2
g_3
g_4}(X)\ar[rrr]& & & F_{g_1 g_2 g_3 g_4}(X) \\
F_{g_1}F_{g_2}F_{g_3 g_4}(X)\ar[u]\ar@{-->}[drr] & F_{g_1}F_{g_2
g_3 g_4}(X)\ar|{(\rho_{g_1}a)(g_2, g_3, g_4)}[ul]\ar[r] & F_{g_1
g_2 g_3 g_4}(X)\ar|{(\rho_{g_1}a)(g_2, g_3, g_4)}[ur]&
\\
F_{g_1}F_{g_2}F_{g_3}F_{g_4}(X)\ar[u]\ar[r]\ar[d]&
F_{g_1}F_{g_2g_3}F_{g_4}(X)\ar[u]\ar[d]& F_{g_1 g_2}F_{g_3 g_4}(X)
\ar@{-->}[r] & F_{g_1 g_2 g_3 g_4}(X)\ar|{a(g_1, g_2, g_3
g_4)}[uu]\\
F_{g_1 g_2}F_{g_3} F_{g_4}(X)\ar[d]\ar@{-->}[urr] & F_{g_1 g_2
g_3} F_{g_4}(X)\ar[r] & F_{g_1 g_2 g_3 g_4}(X) \ar
@/_3.5pc/|(.4){a(g_1, g_2 g_3,
g_4)}[uu]&  \\
F_{g_1 g_2 g_3} F_{g_4}(X)\ar|{a(g_1, g_2, g_3)}[ur]\ar[rrr]& & &
F_{g_1 g_2 g_3 g_4}(X)\ar|{a(g_1, g_2, g_3)}[ul] \ar|{a(g_1 g_2,
g_3,
g_4)}[uu]\\
}\]

\vspace{2mm}

For (ii), observe that we could write $c'(g_1,g_2)=c(g_1,g_2)\cdot
d(g_1,g_2)$ for a unique $d(g_1,g_2)\in Z(\calC)^\times$. Then it
is easy to see that $a'=a+\delta d$. (iii) is standard.
\end{proof}

Thus, we associate to a gerbal action of a group $G$ on a category
$\calC$ a third cohomology class of $G$ with coefficients in
$\calZ(\calC)^\times$, which is the obstruction to upgrade the
gerbal action to a genuine action.

We will apply the following simple observation in \S \ref{classes
corr}. Let $\calD\subset\calC$ be a subcategory of $\calC$. We
call $\calD$ is invariant under the gerbal action $F$ of $G$ if
for any $u:X\to Y\in\calD$, $F_g(u):F_g(X)\to F_g(Y)\in\calD$ for
any $g\in G$. Then we have a gerbal action of $G$ on $\calD$.
Observe that there is a natural $G$-module homomorphism
$r:\calZ(\calC)^\times\to\calZ(\calD)^\times$ by restriction.
We have the following

\begin{prop}\label{invariant}Let $a\in H^3(G,\calZ(\calC)^\times)$
be the obstruction to upgrade the gerbal action of $G$ on $\calC$
to a genuine action. Then the obstruction in
$H^3(G,\calZ(\calD)^\times)$ to upgrade the gerbal action of $G$
on $\calD$ is $r(a)$.
\end{prop}

\subsubsection{An example}\label{example
of gerbe action}

Here is an example of the gerbal action of a group on a category.

Let $R$ be an associative $\bbC$-algebra, and $\calC$ be the
category of left $R$-modules. We assume that the $G$ acts on $R$
via \emph{outer} automorphisms, that is we have a homomorphism $G
\to \on{Aut}(R)/\on{Inn}(R)$. Then one can define a gerbal action
of $G$ on $\calC$. Namely, we choose any lifting of this
homomorphism to a map $s:G\to\mbox{Aut}(R)$ (not necessarily a
group homomorphism!). Let $m:R\otimes M\to M$ be a left
$A$-module. Then define a new $A$-module $F_gm: R\otimes M\to M$
by the formula \[(F_g m)(a,x)=m(s(g)^{-1}a,x) \mbox{ for } a\in A,
x\in M.\] If $f:(\rho,M)\to (\rho',M')$ is a morphism between
$A$-modules, then define $F_g(f)=f$ as linear map between
underlying $\bbC$-vector spaces. Clearly, $F_g:\calC\to\calC$ is a
functor. Moreover, it is easy to see that $F$ is a gerbal action
of $G$ on $\calC$ and we obtain the following:

\begin{cor}
If $G$ acts on an associative $\bbC$-algebra by outer
automorphisms, then there is a canonically defined class $a\in
H^3(G,Z(A)^\times)$. This class vanishes if and only if $G$ acts
on $A$ by (genuine) automorphisms.
\end{cor}

\subsubsection{2-groups arising from
gerbal actions}\label{arising}

Recall that if a group acts a vector space projectively, we obtain
a representation of a central extension of the group. Likewise, if
a group $G$ acts on a category $\calC$ gerbally, we obtain a
genuine action of a certain 2-group on this category.

Assume that there is a gerbal action $F$ of $G$ on some category
$\calC$, which gives us a third cohomology class $[a]\in
H^3(G,\calZ(\calC)^\times)$. Let $A\subset\calZ(\calC)^\times$ be
an abelian subgroup and suppose that the cohomology class $[a]\in
H^3(G,\calZ(\calC)^\times)$ lies in the image of the map
$H^3(G,A)\to H^3(G,\calZ(\calC)^\times)$. Then we have, for any
pair $g,g'\in G$, an isomorphism $c(g,g'):F_gF_{g'}\cong F_{gg'}$
such that $a(g,g',g'')\in A$ for any $g,g',g''\in G$, where
$a(g,g',g'')$ is as in \S 2.3.5. We claim that there is a 2-group
$\calG$, with $\pi_0(\calG)=G$ and $\pi_1(\calG)=A$, and a genuine
action of $\calG$ on $\calC$, which lifts the gerbal action of $G$
on $\calC$ (in the obvious sense). Namely, let $\calG$ be the
category whose objects are $g\in G$, and morphisms are
$\on{Hom}_\calG(g,g') = \emptyset$, if $g \neq g'$, and
$\End_\calG(g)=l_{F_g}(A)$, where
\[l_{F_g}:\pi_1(\bbA\bbU\bbT(\calC))=\calZ(\calC)^\times\to
\End_{\bbA\bbU\bbT(\calC)}(F_g)\]
is an isomorphism. (Let us recall that in a 2-group, the
isomorphism $l_x:\pi_1=\End(I)\to\End(x)$ is induced from the
isomorphism $I\otimes x\cong x$.) The monoidal structure of
$\calG$ is given by $g\otimes g'=gg'$, with the associativity
constraint
\[(a(g,g',g''): (g\otimes g')\otimes g''\cong
g\otimes(g'\otimes g''))\in\End_\calG(gg'g'').\]
It is clear that this monoidal structure makes $\calG$ a 2-group.
The action of $\calG$ on $\calC$ is tautological. Namely, $g$ acts
on $\calC$ by $F_g$.

\section{2-dimensional story:
Gerbal representations}\label{Gerbal
    extensions}

Having developed the formalism of gerbal representations of
groups, we now wish to apply it to a particular group, which we
call $GL_{\infty,\infty}$. We will construct a gerbal
representation of this group on a category of modules over a
Clifford algebra. Presumably, it may also be realized using the
corresponding ``double infinite'' Grassmannian (by analogy with
the 1-dimensional story). We give some indications of how to do
this in \S \ref{2-Gr}, but this Grassmannian is a rather
complicated geometric object that requires further study.

Alternatively, the gerbal action of $GL_{\infty,\infty}$ on this
category may be viewed as a genuine representation of a
``2-group'', which is a $B\bbC^\times$-central extension of
$GL_{\infty,\infty}$. We start this section by introducing these
notions and develop the formalism necessary to describe the gerbal
actions of the group $GL_{\infty,\infty}$. We then use this
formalism to define gerbal representations of
$GL_{\infty,\infty}$.

\subsection{The Lie algebra
$\gl_{\infty,\infty}$ and the group $GL_{\infty,\infty}$}

\subsubsection{Definition}

We start by realizing $\gl_\infty$ as a functor from associative
rings to associative rings. Namely, for any ring $R$, regard
$R\ppart$ as a topological {\em right} $R$-module with the
$t$-adic topology. Then define $\gl_\infty(R)$ as the ring of
continuous endomorphisms of $R\ppart$ viewed as a {\em right}
$R$-module (so that $\gl_\infty(R)$ acts on $R\ppart$ from the
left). Now we define
$$
\gl_{\infty,\infty} := \gl_\infty(\gl_\infty).
$$
We also set
\[
GL_{\infty,\infty}:=\{g\in\gl_{\infty,\infty}, g \mbox{ is
    invertible}\}.
\]

Let us described the Lie algebra in more concrete terms. If we
give $K=\bbC\ppart$ the topological basis $\{t^i\}$, then elements
in $\gl_\infty$ could be regarded as $\infty\times\infty$-matrices
$A=(A_{ij})_{i,j\in\bbZ}$ which act on $\bbC\ppart$ by the formula
$$ A(t^j)=\sum\limits_{i\in\bbZ}A_{ij}t^i. $$
It is easy to see that
\[\gl_\infty=\left\{\begin{array}{l}(A_{ij})_{i,j\in\bbZ},
A_{ij}\in\bbC| \forall
m\in \bbZ, \exists \ n\in\bbZ, \\
\mbox{ such that whenever } i<m, j>n, A_{ij}=0
\end{array}\right\}.\]
Therefore,
\[\gl_{\infty,\infty}=\left\{\begin{array}{l}
(A_{ij})_{i,j\in\bbZ},A_{ij}\in\gl_\infty| \forall
m\in \bbZ, \exists \ n\in\bbZ, \\
\mbox{ such that whenever } i<m, j>n, A_{ij}=0
\end{array}\right\}.\]

From this presentation, it is clear that $\gl_{\infty,\infty}$
acts on
$$
\bbK := \bbC\ppart\ppars
$$
by the following formula. If we represent an element in
$\gl_{\infty,\infty}$ by $A=(A_{ij})_{i,j\in\bbZ}$ and
$A_{ij}=(A_{ij,mn})_{m,n\in\bbZ}$. Then
\[A(t^ns^j)=\sum_{m,n\in\bbZ}A_{ij,mn}t^ms^i.\]
Observe that $GL_{\infty,\infty}$ acts on $\bbK$ continuously by
the same formula as above.

The topology on $\bbK$ is given as follows: a basis of open
neighborhoods of $0\in\bbK$ consists of the subspaces
$$
s^m\bbC\ppart[[s]]+\sum\limits_{i\in\bbZ}t^{m_i}s^i\bbC[[t]],
$$
for some $m,m_i\in\bbZ$. The following lemma is proved in
\cite{Osi}, Proposition 3.

\begin{lem}The action of $\gl_{\infty,\infty}$ on $\bbK$
is continuous with respect to this topology.
\end{lem}

\begin{rmk}
In \cite{Osi}, Proposition 3, the author also shows that the
algebra of continuous endomorphisms of $\bbK$ is larger that
$\gl_{\infty,\infty}$.
\end{rmk}

\subsubsection{Lattices
in $\bbK=\bbC\ppart\ppars$}\label{Lattices}

Let us recall the following construction of completed tensor
products of topological vector spaces from \cite{BD}, \S 3.6.1.
All topological vector spaces in this subsection are assumed to be
linearly topologized, complete and separated with respect to the
topology.

Let $V_i, i=1,2,\ldots,n$, be topological vector spaces. Then one
defines the completed tensor product
$V_1\vec{\otimes}V_2\vec{\otimes}\cdots\vec{\otimes}V_n$ as the
completion of the plain tensor product $V_1\otimes
V_2\otimes\cdots\otimes V_n$ with respect to a topology in which a
vector subspace $U\subset V_1\otimes V_2\otimes\cdots\otimes V_n$
is open if and only if for any $1\leq i\leq n$ and $v_{i+1}\in
V_{i+1},\ldots,v_n\in V_n$, there exists an open $P\subset V_i$
such that $U\supset V_1\otimes\cdots V_{i-1}\otimes P\otimes
v_{i+1}\otimes\cdots v_n$. Observe that in general,
$V_1\vec{\otimes}V_2\ncong V_2\vec{\otimes}V_1$.

If $U$ be a topological vector space and $V\cong\bbC^n$ is a
finite-dimensional vector space, then $U\vec{\otimes}V\cong U^n$
as a topological vector space. Therefore, if $U$ is a Tate vector
space, so is $U\vec{\otimes}V$.  Furthermore, if $V_1\to V_2$ is a
closed embedding and $V_2/V_1$ is equipped with the quotient
topology, then $U\vec{\otimes}V_1\to U\vec{\otimes}V_2$ is a
closed embedding and $(U\vec{\otimes}V_2)/(U\vec{\otimes}V_1)$
with the quotient topology is canonically isomorphic to
$U\vec{\otimes}(V_2/V_1)$.

It is easy to check if $V_1=\bbC\ppart$ and $V_2=\bbC\ppars$, with
the usual adic topologies, then $V_1\vec{\otimes}V_2$ is
isomorphic to $\bbK$ as a topological vector space. In this paper
one could replace $\bbK$ by any other topological vector space of
the form $U\vec{\otimes}V$, where $U$ and $V$ are two non-compact
Tate vector spaces (see \S \ref{Tate}).

\begin{dfn}\label{lat} Let $U$ and $V$ be two Tate vector
spaces. We will call $\bbL\subset U\vec{\otimes}V$ a {\em lattice}
if (i) $\bbL$ is closed in $U\vec{\otimes}V$; and (ii) there exist
linearly compact open subspaces $P_1\subset P_2$ of $V$, such that
$U\vec{\otimes}P_1\subset \bbL\subset U\vec{\otimes}P_2$.
\end{dfn}

It is clear that the intersection of two lattices is again a
lattice.

For any pair of linearly compact open subspaces $P_1\subset
P_2\subset V$, $(U\vec{\otimes}{P_2})/(U\vec{\otimes}P_1)$ is a
Tate vector space. Since $\bbL$ is closed, if
$U\vec{\otimes}P_1\subset \bbL\subset U\vec{\otimes}P_2$, the
exact sequence
\[0\to\bbL/(U\vec{\otimes}P_1)\to
(U\vec{\otimes}{P_2})/(U\vec{\otimes}P_1)\to
(U\vec{\otimes}{P_2})/\bbL\to 0\] is an exact sequence in the
category of Tate vector spaces (see \S \ref{Tate}).

Unlike the case of $K$, two lattices $\bbL, \bbL'$ in
$U\vec{\otimes}V$ are not necessarily commensurable. But we have
the following weaker statement:

\begin{lem}
    \label{quot Tate}
For any two lattices $\bbL', \bbL$ in $U\vec{\otimes}V$ (where $U,
V$ are Tate vector spaces, as above) the quotients $\bbL/(\bbL
\cap \bbL')$ and $\bbL'/(\bbL \cap \bbL')$, endowed with the
quotient topology, are Tate vector spaces.
\end{lem}

\begin{proof} It is
sufficient to consider the case when $\bbL \subset
    \bbL'$. Let $P$ be a linearly
compact open subspace of $V$ such that
$U\vec{\otimes}P\supset\bbL'\supset\bbL$. Then $\bbL'/\bbL$ is a
closed subspace of $U\vec{\otimes}P/\bbL$. Since
$U\vec{\otimes}P/\bbL$ is a Tate vector space, so is $\bbL'/\bbL$.
\end{proof}

In the case when $U\vec{\otimes}V = \bbK$, we have the following
lattice in $\bbK$:
$$
\calO_\bbK:=\bbC\ppart[[s]]\cong K\vec{\otimes}\calO_K.
$$
Note that condition (ii) of Definition \ref{lat} may be written as
$$
s^N\calO_{\bbK}\subset \bbL \subset s^{-N}\calO_{\bbK}
$$
for sufficiently large integer $N$. Here is a more general example
of a lattice in $\bbK$:
$$
\bbL = s^N \bbC\ppart[[s]] \oplus \bigoplus_{i=-M}^{N-1} s^i
t^{m_i} \C[[t]], \qquad m_i \in \Z, M \in \Z_+.
$$

The following lemma shows that all other lattices in $\bbK$ may be
obtained from $\calO_\bbK$ by the $GL_{\infty,\infty}$-action.

\begin{lem}    \label{transitive}
(i) For any $g\in GL_{\infty,\infty}$, $g\calO_{\bbK}$ is a
lattice in $\bbK$.

(ii) $GL_{\infty,\infty}$ acts transitively on the set of lattices
in $\bbK$.
\end{lem}

\begin{proof}(i) It is clear that for any
$g\in GL_{\infty,\infty}$ there exists an integer $N$ such that
$s^N\calO_{\bbK}\subset g\calO_\bbK\subset s^{-N}\calO_{\bbK}$.
Since $\calO_\bbK$ is closed in $\bbK$ and the action of $g$ is
continuous, $g\calO_\bbK$ is also closed in $\bbK$.

(ii) Let $\bbL$ be a lattice such that
$s^N\calO_{\bbK}\subset\bbL\subset s^{-N}\calO_{\bbK}$. Then we
have the exact sequence
\[0\to\bbL/s^N\calO_\bbK\to
s^{-N}\calO_{\bbK}/s^N\calO_\bbK\to s^{-N}\calO_{\bbK}/\bbL\to 0\]
By Remark \ref{exact str}, we could assume that
$s^{-N}\calO_{\bbK}/s^N\calO_\bbK=\bbL/s^N\calO_\bbK\oplus L'$ for
some $L'\cong s^{-N}\calO_\bbK/\bbL$. For any $k\in\bbZ$, we
choose an isomorphism
\[s^{-N}\calO_\bbK/s^N\calO_\bbK\cong\sum\limits_{i=0}^{2N-1}
s^{(2k-1)N+i}\bbC\ppart\] and let the image of
$\bbL/s^N\calO_\bbK$ and $L'$ in
$\sum\limits_{i=0}^{2N-1}s^{(2k-1)N+i}\bbC\ppart$ be $L_k$ and
$L'_k$ respectively. We can always make
$L_0=\bbL\cap\sum\limits_{i=0}^{2N-1}s^{-N+i}\bbC\ppart$. Let us
define an automorphism $g:\bbK\to\bbK$ as follows. We write
\[\bbK=\sum\limits_{k<0}\sum\limits_{i=0}^{2N-1}s^{2kN+i}\bbC\ppart
\oplus\prod\limits_{k\geq
  0}\sum\limits_{i=0}^{2N-1}s^{2kN+i}\bbC\ppart\]
For each $k$, we define an isomorphism $$
g_k:\sum\limits_{i=0}^{2N-1}s^{2kN+i}\bbC\ppart\cong L_k+L'_{k+1}
$$
and then set $$ g=\sum\limits_{k<0}g_k+\prod\limits_{k\geq 0}g_k.
$$
It is clear that $g\in GL_{\infty,\infty}$ and $g\calO_\bbK=\bbL$.
\end{proof}

\subsection{A gerbal representation of $GL_{\infty,\infty}$ on a
  category of Clifford modules}

We wish to define a gerbal representation of $GL_{\infty,\infty}$ on a
certain abelian category realizing a non-trivial third cohomology
class. Recall that according to Definition \ref{gerbe action1}, a
gerbal representation of a group $G$ on an abelian category $\calC$ is
a homomorphism of groups $F:G\to\pi_0(\bbG\bbL(\calC))$, where
$\bbG\bbL(\calC)$ is the category of $\bbC$-linear auto-equivalences
of $\calC$ (see \S \ref{bbGbbL}). More explicitly, this means
assigning every $g\in G$ an auto-equivalence $F_g:\calC\to\calC$ such
that $F_e\cong\mathbf{1}_\calC$ and $F_gF_{g'}\cong F_{gg'}$ (see \S
\ref{gerbal actions}).

Theorem \ref{H^3} shows that a gerbal representation of $G$ on
$\calC$ gives rise to a third cohomology class of $G$ with
coefficients in $\calZ(\calC)^\times$, the (abelian) group of
invertible elements of the center $\calZ(\calC)$ of the category
$\calC$. Furthermore, as explained in \S \ref{arising}, we obtain
a 2-group $\calG$ equipped with a genuine action on $\calC$.

This is analogous to the notion of projective representation of a
group $G$ on a vector space $V$. Such a representation gives rise
to a second cohomology class of $G$ and to a canonical central
extension of $G$ which genuinely acts on $V$. In the 1-dimensional
story we naturally obtain a projective representation of
$GL_\infty$ on a module over a Clifford algebra, as explained in
\S \ref{central extensions}. Therefore it is natural to guess that
one can construct a gerbal action of $GL_{\infty,\infty}$ on a
category of modules over a Clifford algebra. What could this
category be? A naive guess is that it should be a category of
modules over the Clifford algebra $\Cl(\bbK\oplus\bbK^*)$.
However, this cannot be true since $GL_{\infty,\infty}$ acts by
automorphisms on $\Cl(\bbK\oplus\bbK^*)$ and therefore, it acts on
the category of modules {\em genuinely} (that is, the
corresponding third cohomology class is equal to $0$). It turns
out that the correct Clifford algebra is the Clifford algebra
$\Cl(\calO_\bbK\oplus\calO_\bbK^*)$ associated to the lattice
$\calO_\bbK$ of $\bbK$.  In \S \ref{classes corr} we will show
that the corresponding third cohomology class is the non-zero
class $[E_3]\in H^3(GL_{\infty,\infty},\bbC^\times)$ constructed
in \S \ref{group cohomology}.

\subsubsection{More on lattices in $\bbK=\bbC\ppart\ppars$}

Let $U$ and $V$ be Tate vector spaces. Observe that lattices of
$U\vec{\otimes}V$ in general are not open subspaces of
$U\vec{\otimes}V$. Let $\bbL$ be a lattice. We can endow
$(U\vec{\otimes}V)/\bbL$ with the quotient topology. Then
$(U\vec{\otimes}V)/\bbL$ is an ind-Tate vector space, i.e.
$(U\vec{\otimes}V)/\bbL=\lim\limits_{\overrightarrow{i\in I}}V_i$,
where $V_i$ are Tate vector spaces, $V_i\to V_j$ are closed
embeddings, and a subspace $W\subset (U\vec{\otimes}V)/\bbL$ is
open if and only if $W\cap V_i$ is open in $V_i$ for any $i\in I$.
For example, if $\bbL=\calO_{\bbK}\subset\bbK$, then
$$
\bbK/\calO_{\bbK}\cong
s^{-1}\bbC\ppart[s^{-1}]\cong\bbC\ppart^\bbN.
$$
A basis of open neighborhood of $0\in\bbK/\calO_{\bbK}$ could be
given by the sets $\sum\limits_{i\leq-1}s^{i}t^{m_i}\bbC[[t]]$ for
different collections $(m_i)$.

Let $\bbK^*:=K^*\vec{\otimes}K^*\cong\bbC\ppart\ppars dtds$. We
have a natural non-degenerate symmetric bilinear form on
$\bbK\oplus\bbK^*$ induced by the residue pairing. That is,
\[(f,\omega)=\res_{t=0}\res_{s=0}f\omega
\qquad f\in\bbK,\omega\in\bbK^*.\] It is clear that if
$\bbL\subset\bbK$ is a lattice, then $\bbL^\perp\subset\bbK^*$ is
also a lattice. We have a non-degenerate symmetric bilinear form
on $\bbL\oplus\bbK^*/\bbL^\perp$.

\begin{dfn}\label{2-lattice}
Let $\bbL$ be a lattice in $U\vec{\otimes} V$, where $U$ and $V$
are Tate vector spaces. We will call $L\subset\bbL$ a {\em
secondary lattice} of $\bbL$ if $L$ is closed in $\bbL$ and for
any linearly compact open subspace $P\subset V$ such that
$U\vec{\otimes}P\subset\bbL$, $L/(L\cap U\vec{\otimes}P))$ is
linearly compact open in $\bbL/(U\vec{\otimes}P)$.
\end{dfn}

If $\bbL = \C\ppart[[s]]$, a lattice in $\bbK = \C\ppart\ppars$,
then the following is an example of a secondary lattice:
$$
L = \prod_{n \geq 0} s^{m_i} \C[[t]], \qquad m_i \in \Z.
$$

Let $L$ be a secondary lattice of $\bbL$. Denote by
$L^\perp\subset(U \vec{\otimes} V)^*/\bbL^\perp$ its orthogonal
complement, modulo $\bbL^\perp$.  Observe that for any linearly
compact open subspace $P\subset V$ such that
$U\vec{\otimes}P\subset\bbL$, $(U\vec{\otimes}P)^\perp/\bbL^\perp$
is a Tate vector subspace of $(U \vec{\otimes} V)^*/\bbL^\perp$.
It is clear that $L^\perp\cap
((U\vec{\otimes}P)^\perp/\bbL^\perp)$ is a linearly compact open
subspace of $(U\vec{\otimes}P)^\perp/\bbL^\perp$. Therefore,
$L^\perp$ is an open subspace of $(U \vec{\otimes}
V)^*/\bbL^\perp$.

\begin{lem} Denote
by $(L^{\perp})^\perp\subset\bbL$ the orthogonal complement of
$L^\perp$. Then $(L^{\perp})^\perp=L$.
\end{lem}

\begin{proof} Observe that $L$ is a
secondary lattice in $\bbL$ if and only if for any linearly
compact open subspace $P\subset V$ such that
$U\vec{\otimes}P\subset\bbL$, $L/(L\cap (U\vec{\otimes}P))$ is
linearly compact open in $\bbL/(U\vec{\otimes}P)$, and
$L=\lim\limits_{\overleftarrow{P}}L/(L\cap (U\vec{\otimes}P))$.
Indeed, there is a natural injection $L\to
\lim\limits_{\overleftarrow{P}}L/(L\cap (U\vec{\otimes}P))$, which
is a surjection if and only if $L$ is closed.

Then
$(L^{\perp})^\perp=\lim\limits_{\overleftarrow{P}}(L^\perp\cap
(U\vec{\otimes}P)^\perp/\bbL^\perp)^\perp$, where $(L^\perp\cap
(U\vec{\otimes}P)^\perp/\bbL^\perp)^\perp$ is the orthogonal
complement of $L^\perp\cap (U\vec{\otimes}P)^\perp/\bbL^\perp$ in
$\bbL/(U\vec{\otimes}P)$, under the natural pairing between
$\bbL/(U\vec{\otimes}P)$ and $(U\vec{\otimes}P)^\perp/\bbL^\perp$.
However, it is clear that $(L^\perp\cap
(U\vec{\otimes}P)^\perp/\bbL^\perp)^\perp=L/(L\cap
U\vec{\otimes}P)$. The lemma follows.
\end{proof}

Therefore $L\oplus L^\perp$ is a maximal isotropic subspace of
$\bbL\oplus (U \vec{\otimes} V)^*/\bbL^\perp$. We will call such
subspaces {\em Lagrangian}.

\subsubsection{The category of Clifford modules}

 \label{cat 2-Cl}

{}From now on we will consider lattices $\bbL \subset \bbK =
\C\ppart\ppars$. Let
$$
\Cl_\bbL:=\Cl(\bbL\oplus(\bbK^*/\bbL^\perp))
$$
be the Clifford algebra associated to the space
$\bbL\oplus(\bbK^*/\bbL^\perp)$ equipped with a non-degenerate
symmetric bilinear form defined above. Let $L$ be a secondary
lattice of $\bbL$. Since $L\oplus L^\perp$ is a Lagrangian
subspace of $\bbL\oplus\bbK^*/\bbL^\perp$, the exterior algebra
$\bigwedge(L\oplus L^\perp)$ is a subalgebra of $\Cl_\bbL$. We
define a $\Cl_\bbL$-module
$$
M_L=\ind_{\bigwedge(L\oplus L^\perp)}^{\Cl_\bbL}(\bbC|0\rangle).
$$

Observe that $M_L$ is not a discrete $\Cl_\bbL$-module.
Nevertheless, we still have the following

\begin{lem} $M_L$ is
an irreducible $\Cl_\bbL$-module.
\end{lem}

\begin{proof} We can always find
a subspace $L'\subset \bbL$ such that
    $\bbL=L\oplus L'$ and
that $\bbK^*/\bbL^\perp=L^\perp\oplus L'^\perp$. Therefore,
$M_L\cong\bigwedge(L'\oplus L'^\perp)|0\rangle$, and any element
in $M_L$ could be expressed as a finite sum
$m=\sum_{i_1,\ldots,i_r}v_{i_1}\cdots v_{i_r}|0\rangle$ with
$v_{i_j}\in L'\oplus L'^\perp$. Let $W$ be the span of these
$v_{i_j}$. This is a finite-dimensional subspace in $L'\oplus
L'^\perp$. One can always find a finite-dimensional subspace
$W^*\subset L\oplus L^\perp$ such that the non-degenerate
symmetric bilinear form on $\bbL\oplus \bbK^*/\bbL$ restricts to a
non-degenerate symmetric bilinear form on $W\oplus W^*$. Let
$\Cl_W=\Cl(W\oplus W^*)$ be the Clifford algebra associated to
this bilinear form. This is a subalgebra of $\Cl_\bbL$. It is
clear that $m\in\Cl_W|0\rangle=\ind_{\bigwedge
W^*}^{\Cl_W}(\bbC|0\rangle)$.  From the theory of Clifford
algebras modeled on finite-dimensional vector spaces we know that
there exists some $a\in\Cl_W\subset\Cl_\bbL$ such that
$a(m)=|0\rangle$. This proves that $M_L$ is irreducible.
\end{proof}

\begin{lem}    \label{non-isom}
$\Hom_{\Cl_\bbL}(M_L,M_{L'}) \neq 0$
if and only if $L$ and $L'$ are commensurable with each other. In
that case, there exists a canonical isomorphism
$$\Hom_{\Cl_\bbL}(M_L,M_{L'})\cong\det(L|L')$$ such
that the same diagram as in Lemma \ref{Hom set} holds.
\end{lem}

\begin{proof}
Assume that $\phi\in\Hom_{\Cl_\bbL}(M_L,M_{L'})$ is non-zero. Then
$\phi(|0\rangle_L)$ is an element in $M_{L'}$ which is annihilated
by $L\oplus L^\perp$. As in the proof of the previous lemma, one
can find a finite-dimensional Clifford algebra
$\Cl_W\subset\Cl_\bbL$ such that
$\phi(|0\rangle_L)\in\Cl_W|0\rangle_{L'}=\ind_{\bigwedge
W^*}^{\Cl_W} (\bbC|0\rangle_{L'})$.  Then it is clear that
$\phi(|0\rangle_L)$ is annihilated by $W^\perp\cap(L'\oplus
L'^\perp)$. Therefore, $W^\perp\cap(L'\oplus L'^\perp)\subset
L\oplus L^\perp$. Since $(L'\oplus L'^\perp)/(W^\perp\cap(L'\oplus
L'^\perp))\cong W^*$ is finite-dimensional, we obtain that
$L'/(L'\cap L)$ is also finite-dimensional. Since $M_L$ and
$M_{L'}$ are irreducible, $\phi$ is an isomorphism. Applying the
same argument to $\phi^{-1}$, we obtain that $L/(L\cap L')$ is
finite-dimensional.

Before we prove the second statement of the lemma, we claim that
the $\bigwedge(L\oplus L^\perp)$-invariant subspace of $M_L$ is
the line $\bbC|0\rangle$. Indeed, assume that $m\in M_L$ is
invariant under $\bigwedge(L\oplus L^\perp)$. Then, as in the
proof of the previous lemma, we can assume that
$m\in\Cl_W|0\rangle=\ind_{\bigwedge W^*}^{\Cl_W}(\bbC|0\rangle)$.
Then $m$ is invariant under $\bigwedge W^*$. By the theory of
finite dimensional Clifford algebras, $m\in\bbC|0\rangle$.

Now assume that $L$ and $L'$ are commensurable. As in the proof of
Lemma \ref{Hom set}, $\Hom_{\Cl_\bbL}(M_L,M_{L'})$ is canonically
isomorphic to the $\bigwedge(L\oplus L^\perp)$-invariant subspace
in $M_{L'}$, which is either zero- or one-dimensional. Indeed, if
this subspace is not zero, then $M_{L'}\cong M_L$ and we have just
seen that the $\bigwedge(L\oplus L^\perp)$-invariant subspace in
$M_L$ is one-dimensional.

Let $V=(L\oplus L^\perp)\cap(L'\oplus L'^\perp)$. Then $V^\perp/V$
is finite-dimensional and it carries a non-degenerate symmetric
bilinear form. Denote $\Cl(V^\perp/V)$ to be the corresponding
Clifford algebra. The $V$-invariant subspace $M_{L'}^V$ in
$M_{L'}$ is naturally a module over $\Cl(V^\perp/V)$. Let
$\Cl(V^\perp/V)|0\rangle_{L'}$ be the submodule containing the
vacuum vector $|0\rangle_{L'}$. We know that $(L\oplus L^\perp)/V$
is a Lagrangian subspace of $V^\perp/V$. The theory of
finite-dimensional Clifford algebras implies that the
$\bigwedge((L\oplus L^\perp)/V)$-invariant subspace in
$\Cl(V^\perp/V)|0\rangle_{L'}$ is canonically isomorphic to
$\bigwedge^{\on{top}} L/(L\cap L')\otimes(\bigwedge^{\on{top}}
L'/(L\cap L'))^{-1}$. That is, the $\bigwedge(L\oplus
L^\perp)$-invariant subspace in $M_{L'}$ is canonically isomorphic
to $\det(L|L')$.
\end{proof}

\begin{dfn} $\calC_{\bbL}^{\on{ss}}$ is
the semi-simple abelian category, whose objects are
$\Cl_\bbL$-modules that are direct sums of $M_L$, with $L$ being
secondary lattices of $\bbL$, and morphisms are homomorphisms of
these $\Cl_\bbL$-modules.
\end{dfn}

Let $g\in GL_{\infty,\infty}$. Let $\bbL\subset\bbK$ be a lattice,
and $L\subset\bbL$ a secondary lattice. It is clear that
$gL\subset g\bbL$ is a secondary lattice. We define, for any $g\in
GL_{\infty,\infty}$, a functor
$$
T_g:\calC_\bbL^{\on{ss}}\to\calC_{g\bbL}^{\on{ss}},
$$
which sends $M_L$ to $M_{gL}$, and the map of Hom's is defined as
follows: for $L,L'\subset\bbL$ commensurable, we have
$$
\Hom_{\Cl_\bbL}(M_L,M_{L'})\cong\det(L|L')\cong\det(gL|gL')\cong
\Hom_{\Cl_{g\bbL}}(M_{gL},M_{gL'}).
$$
This isomorphism gives us a map $\Hom_{\Cl_\bbL}(M_L,M_{L'}) \to
\Hom_{\Cl_{g\bbL}}(M_{gL},M_{gL'})$.

\subsubsection{Gerbal representation of
$GL_{\infty,\infty}$} \label{gerbal1}

The first main theorem of this paper is the following:

\begin{thm}\label{main1} There is a natural gerbal representation
of $GL_{\infty,\infty}$ on the category $\calC_{\bbL}^{\on{ss}}$.
These representations are equivalent to each other for all lattices
$\bbL \subset \bbK$.
\end{thm}

\begin{proof} We will construct for
any lattices two $\bbL,\bbL'\subset\bbK$, an equivalence of categories
\[\Xi_{\bbL,\bbL'}:\calC_{\bbL}^{\on{ss}}\to\calC_{\bbL'}^{\on{ss}}\]
such that:

\medskip

(i)
$\Xi_{\bbL',\bbL''}\circ\Xi_{\bbL,\bbL'}\cong\Xi_{\bbL,\bbL''}$;

\medskip

(ii) For any $g\in GL_{\infty,\infty}$, $\Xi_{g\bbL,g\bbL'}\circ
T_g\cong T_g\circ\Xi_{\bbL,\bbL'}$.

\medskip

It is enough to construct such functors for pairs
$\bbL\supset\bbL'$, so that (i) and (ii) hold. Then we extend
the definition to any pair $\bbL,\bbL'$ by
$\Xi_{\bbL,\bbL'}=\Xi_{\bbL',\bbL\cap\bbL'}^{-1}\circ
\Xi_{\bbL,\bbL\cap\bbL'}$, where $\Xi_{\bbL',\bbL\cap\bbL'}^{-1}$
is any quasi-inverse functor of $\Xi_{\bbL',\bbL\cap\bbL'}$.

Therefore, we assume that $\bbL\supset\bbL'$. According to Lemma
\ref{quot Tate}, $\bbL/\bbL'$ is a Tate vector space. We therefore
have the gerbe $\calD_{\bbL/\bbL'}$ of determinantal theories on
$\bbL/\bbL'$ (see Definition \ref{def det th}).
Choose an object $\Delta_{\bbL,\bbL'}\in\calD_{\bbL/\bbL'}$, that
is, a determinantal theory. If $L\subset\bbL$ is a secondary
lattice, then $L\cap\bbL'$ is a secondary lattice of $\bbL'$ and
$L/(L\cap\bbL')$ is a linearly compact open subspace of
$\bbL/\bbL'$. Now $\Xi_{\bbL,\bbL'}$ is defined on objects by the
formula
\begin{equation}
\label{Xi} \Xi_{\bbL,\bbL'}(M_L)=M_{L\cap\bbL'}\otimes
\Delta_{\bbL,\bbL'}(L/(L\cap \bbL'))^{-1}.
\end{equation}
To define $\Xi_{\bbL,\bbL'}$ on Hom's, observe that if $L$ and
$L'$ are two commensurable secondary lattices of $\bbL$, then we
have a canonical isomorphism
\[\det(L|L')\cong\det(L\cap\bbL'|L'\cap\bbL')
\otimes\det(\frac{L}{L\cap\bbL'}|\frac{L'}{L'\cap\bbL'}).\] Since
\[\det(\frac{L}{L\cap\bbL'}|\frac{L'}{L'\cap\bbL'})\cong
\Hom(\Delta_{\bbL,\bbL'}(\frac{L'}{L'\cap\bbL'}),
\Delta_{\bbL,\bbL'}(\frac{L}{L\cap\bbL'})),\] we have a canonical
isomorphism
\begin{multline*}
\Hom_{\Cl_\bbL}(M_L,M_{L'})\cong
\\
\Hom_{\Cl_{\bbL'}}
(M_{L\cap\bbL'}\otimes\Delta_{\bbL,\bbL'}(L/(L\cap\bbL'))^{-1},
M_{L'\cap\bbL'}\otimes\Delta_{\bbL,\bbL'}(L'/(L'\cap\bbL'))^{-1}).
\end{multline*}
This gives us the sough-after map
$$
\Hom_{\Cl_\bbL}(M_L,M_{L'}) \to
\Hom_{\Cl_{\bbL'}}(\Xi_{\bbL,\bbL'}(M_L),\Xi_{\bbL,\bbL'}(M_{L'})).
$$
(This explains the necessity of the second factor in formula
\eqref{Xi}.)

Now assume that $\bbL\supset\bbL'\supset\bbL''$. Then
\[\begin{split}\Xi_{\bbL',\bbL''}(\Xi_{\bbL,\bbL'}(M_L))=
&M_{L\cap\bbL''}\otimes \Delta_{\bbL,\bbL'}(L/(L\cap \bbL'))^{-1}
\otimes\Delta_{\bbL',\bbL''}(L\cap\bbL'/(L\cap
\bbL''))^{-1}\\
=&M_{L\cap\bbL''}\otimes(\Delta_{\bbL,\bbL'}\otimes
\Delta_{\bbL',\bbL''})(L/(L\cap\bbL''))^{-1}\end{split},\] where
we apply the canonical equivalence
$\calD_{\bbL/\bbL'}\otimes\calD_{\bbL'/\bbL''}\cong
\calD_{\bbL/\bbL''}$ of Lemma \ref{determinantal gerbe}.

On the other hand
\[\Xi_{\bbL,\bbL''}(M_L)=M_{L\cap\bbL''}\otimes
\Delta_{\bbL,\bbL''}(L/(L\cap \bbL''))^{-1}.\] Therefore, any
choice of an isomorphism
$\Delta_{\bbL,\bbL'}\otimes\Delta_{\bbL',\bbL''}\cong
\Delta_{\bbL,\bbL''}$ will give us an isomorphism of functors
$\Xi_{\bbL',\bbL''}\circ\Xi_{\bbL,\bbL'}\cong\Xi_{\bbL,\bbL''}$.

Next, let $g\in GL_{\infty,\infty}$. We have
\[T_g(\Xi_{\bbL,\bbL'}(M_L))=M_{g(L\cap\bbL')}\otimes
\Delta_{\bbL,\bbL'}(L/(L\cap\bbL'))^{-1}=M_{g(L\cap\bbL')} \otimes
(g\Delta_{\bbL,\bbL'})(gL/gL\cap g\bbL'))^{-1},\] where
$g(L\cap\bbL')$ is, by definition, the determinantal theory on
$g\bbL/g\bbL'$ which assigns to $\wt{L}\subset g\bbL/g\bbL'$ the
vector space $\Delta_{\bbL/\bbL'}(g^{-1}\wt{L})$. On the other
hand, we have
\[\Xi_{g\bbL,g\bbL'}(T_g(M_L))=M_{gL\cap
g\bbL'}\otimes \Delta_{g\bbL,g\bbL'}(gL/(gL\cap g\bbL'))^{-1}.\]
Note that $g\Delta_{\bbL,\bbL'}$ and $\Delta_{g\bbL,g\bbL'}$ are
objects of the same groupoid $\calD_{g\bbL/g\bbL'}$. Therefore,
the set of isomorphisms
$g\Delta_{\bbL,\bbL'}\cong\Delta_{g\bbL,g\bbL'}$ is a
$\C^\times$-torsor (in particular, non-empty). A choice of such an
isomorphism will give rise to an isomorphism
$T_g\circ\Xi_{\bbL,\bbL'}\cong\Xi_{g\bbL,g\bbL'}\circ T_g$.

Now we are ready to define a gerbal representation of
$GL_{\infty,\infty}$ on $\calC_\bbL^{\on{ss}}$. For $g\in
GL_{\infty,\infty}$, we define the functor
$F_{g}:\calC_\bbL^{\on{ss}} \to \calC_\bbL^{\on{ss}}$ as the
composition
$$
\calC_\bbL^{\on{ss}} \stackrel{T_g}{\to}
\calC_{g\bbL}^{\on{ss}}\stackrel{\Xi_{g\bbL,\bbL}}{\to}
\calC_\bbL^{\on{ss}}.
$$
Properties (i) and (ii) of $\Xi_{\bbL,\bbL'}$ imply that $F_g\circ
F_{g'}\cong F_{gg'}$. Thus, we obtain a gerbal representation of
$GL_{\infty,\infty}$ on $\calC_\bbL^{\on{ss}}$.

The functor $\Xi_{\bbL,\bbL'}$ defines an equivalence between the two
gerbal representations, $\calC_{\bbL}^{\on{ss}}$ and
$\calC_{\bbL'}^{\on{ss}}$, that is, we have an isomorphism of functors
$\Xi_{\bbL,\bbL'} \circ F_g \simeq F_g \circ \Xi_{\bbL,\bbL'}$ for all
$g \in GL_{\infty,\infty}$ (see Definition \ref{gerbe action1}). This
follows from properties (i) and (ii) of $\Xi_{\bbL,\bbL'}$. Thus, the
gerbal representation $\calC_{\bbL}^{\on{ss}}$ is independent of the
choice of the lattice $\bbL \subset \bbK$. (This is analogous to the
fact that the Fock modules $M_L$ are isomorphic as representations of
$GL_\infty$ for all lattices $L \subset K$.)
\end{proof}

According to Theorem \ref{H^3}, the gerbal representation of
$GL_{\infty,\infty}$ on $\calC_\bbL^{\on{ss}}$ gives us a cohomology
class in $H^3(GL_{\infty,\infty},\calZ(\calC_\bbL^{\on{ss}})^\times)$.
Since $\calC_\bbL^{\on{ss}}$ is $\bbC$-linear, there is a natural
embedding $\bbC^\times\to\calZ(\calC_\bbL^{\on{ss}})^\times$. We will
see that this class is in fact in the image
$H^3(GL_{\infty,\infty},\bbC^\times)\to
H^3(GL_{\infty,\infty},\calZ(\calC_\bbL^{\on{ss}})^\times)$, and will
compute it in \S \ref{gerbal2}. (We note that this map of cohomology
groups is injective since the embedding
$\bbC^\times\to\calZ(\calC_\bbL^{\on{ss}})^\times$ admits a splitting
by
$\calZ(\calC_\bbL^{\on{ss}})^\times\to\End(M_L)^\times\cong\bbC^\times$.)

\begin{rmk} Here is an informal explanation of the above construction
and the meaning of formula \eqref{Xi}. Naively, we would like to
identify the Clifford modules $M_L$ and $M_{L'}$, where $L' =
L\cap\bbL'$. If $L/L'$ were a finite-dimensional subspace of
$\bbL/\bbL'$, then we would have an isomorphism
$$
M_{L} \simeq M_{L'} \otimes \on{det}(L/L')^{-1},
$$
sending the generating vector $\vac_L$ to the vector \eqref{image} in
$M_{L'}$, which depends on the choice of a non-zero vector in
$\on{det}(L/L')$, interpreted as the wedge product of basis vectors
in $L/L'$.

But in general $L/L'$ is infinite-dimensional, and so the Clifford
modules $M_L$ and $M_{L'}$ are not isomorphic, according to Lemma
\ref{non-isom}. But we can generalize formula \eqref{image} by taking
the {\em infinite wedge product} of a basis in $L/L'$. To do this, we
must pick a vector in the determinant line
$\Delta_{\bbL,\bbL'}(L/L')$. This leads us to formula \eqref{Xi},
which we use to define a gerbal action of $GL_{\infty,\infty}$ on
$\calC_\bbL^{\on{ss}}$. But the determinant of $L/L'$ is
non-canonical; in order to define it, we have to choose a
determinantal theory $\Delta_{\bbL,\bbL'}$ on $\bbL/\bbL'$.
\end{rmk}

\subsection{The 2-group determinantal extension of
  $GL_{\infty,\infty}$}

In the previous section we constructed a gerbal representation of
$GL_{\infty,\infty}$ on $\calC_\bbL^{\on{ss}}$. This is a
2-dimensional analogue of the projective representation of
$GL_{\infty}$ on the Fock representation $\bigwedge$ described in \S
\ref{another descr}. In that case the operator on $\bigwedge =
M_{L_0}$ corresponding to $g \in GL_\infty$ was defined as the
composition of the tautological identification of $M_{L_0}$ with
$M_{gL_0}$ and a non-trivial isomorphism $M_{gL_0} \simeq M_{L_0}$,
which depends on the choice of $c \in
\on{det}(gL_0|L_0)^\times$. Thus, the projective action of $GL_\infty$
gives rise to a genuine action of $\wh{GL}_\infty$, which is the
$\bbC^\times$-central extension of $GL_\infty$ by determinant
lines. (Note that we can replace $L_0$ here by an arbitrary lattice $L
\in K$.)

We have defined a gerbal action of $GL_{\infty,\infty}$ on the
category $\calC_\bbL^{\on{ss}}$ in a similar way: the functor $F_g, g
\in GL_{\infty,\infty}$, on the category $\calC_\bbL^{\on{ss}}$ is the
composition of the tautological equivalence $\calC_\bbL^{\on{ss}}
\simeq \calC_{g\bbL}^{\on{ss}}$ and a non-trivial equivalence
$\Xi_{g\bbL,\bbL}: \calC_{g\bbL}^{\on{ss}} \to \calC_\bbL^{\on{ss}}$,
which depends on the choice of a determinantal theory
$\Delta_{\bbL,\bbL'}$ on $\bbL/\bbL'$.  Hence, in order to promote
this gerbal representation to a genuine representation, we must
include this additional choice as part of our data. Thus, the gerbal
representation of $GL_{\infty,\infty}$ gives rise to a genuine
representation of a 2-group $\bbG\bbL_{\infty,\infty}$, which is a
$B\bbC^\times$-central extension of $GL_{\infty,\infty}$ by
determinantal gerbes which is constructed in this section. This
2-group is essentially equivalent to the 2-group construct
previously by Arkhipov and Kremnizer \cite{AK}.  We will show that
this extension is non-trivial, which is equivalent to the fact that
the third cohomology class of our gerbal representation of
$GL_{\infty,\infty}$ on $\calC_\bbL^{\on{ss}}$ is also non-trivial.

We recall from Definition \ref{2-grp} that a 2-group is a monoidal
groupoid $\calG$ such that the set of isomorphism classes
of $\calG$, denoted by $\pi_0(\calG)$, is a group under the
multiplication induced from the monoidal structure.

\subsubsection{Determinantal gerbes}\label{dg}

Recall from Lemma \ref{quot Tate} that if $\bbL\subset\bbL'$ are
two lattices in $\bbK$, then $\bbL'/\bbL$ with the quotient
topology is a Tate vector space. We continue to use the notation
$\calD_{\bbL'/\bbL}$ for the $\C^\times$-gerbe of determinantal
theories on $\bbL'/\bbL$ (see Definition \ref{def det th}).

We will fix, once and for all, for each pair $\bbL\subset\bbL'$, a
linearly compact open subspace $V_{\bbL,\bbL'}\subset\bbL'/\bbL$,
such that for any $\bbL\subset\bbL'\subset\bbL''$, we have the
exact sequence
\[0\to
V_{\bbL,\bbL'}\to V_{\bbL,\bbL''}\to V_{\bbL',\bbL''}\to 0 \ \
\qquad \qquad \qquad (*)\] This is always possible. For example,
one can choose
\[V_{\bbL,\bbL'}=\bbL'\cap\bbC[[t]]\ppars/\bbL\cap\bbC[[t]]\ppars.\]

Recall Lemma \ref{The   determinantal gerbes}. If
$\bbL\subset\bbL'\subset\bbL''$ are three lattices, then there is
a canonical equivalence
\[\calD_{\bbL'/\bbL}\otimes\calD_{\bbL''/\bbL'}\stackrel{\cong}{\to}
\calD_{\bbL''/\bbL}\] (where the tensor product of two gerbes is
defined in \S \ref{determinantal gerbe}) and for a four-step
filtration, there is a canonical isomorphism between two
equivalences, such that for a five-step filtration, the natural
diagram of these canonical isomorphisms is commutative. Now our
choice $V_{\bbL',\bbL''}$ gives a quasi-inverse
\begin{equation}
\label{quasi-inverse}
\calD_{\bbL''/\bbL}\stackrel{\cong}{\to}\calD_{\bbL'/\bbL}\otimes
\calD_{\bbL''/\bbL'}.
\end{equation}
Namely, let $\Delta_{V_{\bbL',\bbL''}}$ be the determinantal
theory in $\calD_{\bbL''/\bbL'}$ such that
$\Delta_{V_{\bbL',\bbL''}}(V_{\bbL',\bbL''})=\bbC$. Then for any
$\Delta\in\calD_{\bbL''/\bbL}$ we define
$\Delta'\in\calD_{\bbL'/\bbL}$ as follows: let
$U\subset\bbL'/\bbL$ be a linearly compact open subspace. Choose
any linearly compact open subspace $U'\subset\bbL''/\bbL$ that
contains $U$, and define
$\Delta'(U)=\Delta(U')\otimes\Delta_{V_{\bbL',\bbL''}}(U'/U)^{-1}$.
It is clear this is independent of the choice of $U'$ up to a
canonical isomorphism. Now the functor sending
$\Delta\mapsto\Delta'\otimes\Delta_{V_{\bbL',\bbL''}}$ is the
desired quasi-inverse. The fact that our choice of
$\{V_{\bbL,\bbL'}\}$ satisfies ($*$) makes properties similar to
those of Lemma \ref{determinantal gerbe} hold for these
quasi-inverses.

Now let $\bbL$ and $\bbL'$ be two lattices in $\bbK$. Then
$\bbL\cap\bbL'$ is also a lattice. We define a $\bbC^\times$-gerbe
\[\calD(\bbL|\bbL')
:=
\calD_{\bbL'/\bbL\cap\bbL'}\otimes\calD_{\bbL/\bbL\cap\bbL'}^{-1}.\]
Recall that for any gerbe ${\mc D}$ we denote by ${\mc D}^{-1}$
the dual gerbe (the objects are the same as in ${\mc D}$ and the
set of morphisms between two objects in ${\mc D}^{-1}$ is the
torsor dual to the set of morphisms between these objects in ${\mc
D}$).

The following lemma is a consequence of Lemma \ref{determinantal
gerbe} and the above discussion. In it we assume that we have made
a choice of a collection of subspaces
$V_{\bbL,\bbL'}\subset\bbL'/\bbL$ satisfying the above properties,
and we use the corresponding quasi-inverse equivalences
\eqref{quasi-inverse}.

\begin{lem}\label{determinantal gerbe II}
(i) Any $g\in GL_{\infty,\infty}$ gives rise to an equivalence of
$\C^\times$-gerbes $\calD(\bbL|\bbL')\cong\calD(g\bbL|g\bbL')$.

(ii) There exists a preferred equivalence of categories
\[\calD(\bbL|\bbL')\otimes\calD(\bbL'|\bbL'')\cong\calD(\bbL|\bbL'')\]
and natural transformations between the two equivalences
$$\calD(\bbL_1|\bbL_2)\otimes\calD(\bbL_2|\bbL_3)\otimes
\calD(\bbL_3|\bbL_4)\cong\calD(\bbL_1|\bbL_4)$$ such that the
diagram of natural transformations is commutative.
\end{lem}

\subsubsection{$B\bbC^\times$-central extension
of $GL_{\infty,\infty}$}    \label{BGm central extension}

Let $\calG$ be a 2-group. In Definition \ref{2-group of central
extensions} we introduced the notion of central extension of
$\pi_0(\calG)$ by $B\pi_1(\calG)$.

Now we define a $B\bbC^\times$-central extension of
$GL_{\infty,\infty}$. We fix a lattice $\bbL\subset\bbK$. Let
$\GL_{\infty,\infty}$ be the following category: objects are
$(g,\Delta), g\in GL_{\infty,\infty}, \Delta\in\calD(g\bbL|\bbL)$;
the set of morphisms
$\Hom_{\GL_{\infty,\infty}}((g,\Delta),(g',\Delta'))$ is empty if
$g\neq g'$ and is the space $\Hom_{\calD
(g\bbL|\bbL)}(\Delta,\Delta')$ if $g=g'$. We define the monoidal
structure
$$
m:\GL_{\infty,\infty}\times\GL_{\infty,\infty}\to\GL_{\infty,\infty}
$$
by $m((g,\Delta),(g',\Delta'))=(gg',\Delta\otimes g(\Delta'))$.
Here we regard $\Delta\otimes g(\Delta')$ as an object of the
category
$$
\calD(g\bbL|\bbL)\otimes\calD(gg'\bbL|g\bbL)\cong\calD(gg'\bbL|\bbL),
$$
using Lemma \ref{determinantal gerbe II}. This Lemma also
guarantees that $\GL_{\infty,\infty}$ is a 2-group which is a
monoidal category. It is clear that $\pi_1(\GL_{\infty,\infty}) =
\C^\times$ and the action of $\pi_0(\GL_{\infty,\infty}) =
GL_{\infty,\infty}$ on $\pi_1(\GL_{\infty,\infty})$ is trivial.
Thus, $\GL_{\infty,\infty}$ is a $B\bbC^\times$-central extension
of $GL_{\infty,\infty}$.

\subsubsection{The genuine action
of $\bbG\bbL_{\infty,\infty}$ on $\calC_\bbL^{\on{ss}}$} Now we
show that the gerbal representation of $GL_{\infty,\infty}$ gives
rise to a genuine representation of the 2-group
$\GL_{\infty,\infty}$.

Let $\bbL$ and $\bbL'$ be two lattices of $\bbK$. We claim that
there is a functor
\[\Xi:\calD(\bbL|\bbL')\to\bbF\bbU\bbN\bbC(\calC_{\bbL}^{\on{ss}},
\calC_{\bbL'}^{\on{ss}}),\] where
$\bbF\bbU\bbN\bbC(\calC_{\bbK/\bbL}^{\on{ss}},\calC_{\bbK/\bbL'}^{\on{ss}})$
denotes the category of additive functors between these two
abelian categories, such that the following diagram commutes up to
a canonical isomorphism.
\[\begin{CD}\calD(\bbL|\bbL')\otimes\calD(\bbL'|\bbL'')@>
\Xi\otimes\Xi>>\bbF\bbU\bbN\bbC(\calC_{\bbL}^{\on{ss}},
\calC_{\bbL'}^{\on{ss}})\otimes\bbF\bbU\bbN\bbC(\calC_{\bbL'}^{\on{ss}},
\calC_{\bbL''}^{\on{ss}})\\
@VVV@VVV
\\
\calD(\bbL|\bbL'')@>\Xi>>\bbF\bbU\bbN\bbC(\calC_{\bbL}^{\on{ss}},
\calC_{\bbL''}^{\on{ss}})\end{CD}\]

It is clear that we need only to consider the case
$\bbL\subset\bbL'$ and $\bbL\supset\bbL'$. We first consider the
case $\bbL\supset\bbL'$. Then
$\calD(\bbL|\bbL')=\calD_{\bbL/\bbL'}^{-1}$. An object $\Delta$ is
a determinantal theory of the Tate vector space $\bbL/\bbL'$. If
$L\subset\bbL$ is a secondary lattice, then $L\cap\bbL'$ is a
secondary lattice of $\bbL'$ and $L/(L\cap\bbL')$ is a linearly
compact open subspace of $\bbL/\bbL'$. Now $\Xi$ is defined on
objects by the formula
\[\Xi(\Delta)(M_L)=M_{L\cap\bbL'}\otimes\Delta(L/(L\cap \bbL'))^{-1},\]
We define $\Xi(\Delta)$ on morphisms so as to make
$\Xi(\Delta)\in\bbF\bbU\bbN\bbC(\calC_{\bbK/\bbL}^{\on{ss}},
\calC_{\bbK/\bbL'}^{\on{ss}})$ in the same way as in the proof of
Theorem \ref{main1}.  Therefore, $\Xi$ defines a functor from
$\calD(\bbL|\bbL')$ to
$\bbF\bbU\bbN\bbC(\calC_{\bbK/\bbL}^{\on{ss}},\calC_{\bbK/\bbL'}^{\on{ss}})$.

Next, we consider the case $\bbL\subset\bbL'$. Then
$\calD(\bbL|\bbL')=\calD_{\bbL'/\bbL}$. Let
$\Delta\in\calD_{\bbL'/\bbL}$. For any secondary lattice
$L\subset\bbL$ we define
\[\Xi(\Delta)(M_L)=M_{L'}\otimes\Delta(V_{\bbL,\bbL'}),\]
where $L'$ is any secondary lattice in $\bbL'$ fitting into the
following exact sequence
\[0\to
L\to L'\to V_{\bbL,\bbL'}\to 0\] Observe that $M_{L'}$ is
independent of the choice of $L'$ up to a canonical isomorphism.
So $\Xi$ is well-defined in this case.

Now we define the genuine action of the 2-group
$\GL_{\infty,\infty}$ on $\calC_\bbL^{\on{ss}}$ as follows: for
$(g,\Delta)\in\GL_{\infty,\infty}$,
$F_{(g,\Delta)}:\calC_\bbL^{\on{ss}} \to \calC_\bbL^{\on{ss}}$ is
defined as the composition
$$
\calC_\bbL^{\on{ss}} \stackrel{g}{\to}
\calC_{g\bbL}^{\on{ss}}\stackrel{\Xi(\Delta)}{\to}\calC_\bbL^{\on{ss}}.
$$

\subsection{2-infinite
Grassmannian}    \label{2-Gr}

This subsection is completely independent from the rest of the
paper. Here we discuss informally a possible geometric realization
of the gerbal representation constructed in the previous section.

First, let us recall the 1-dimensional story. Consider the set of
lattices in the Tate vector space $K = \C\ppart$. Let us pick a
lattice, for instance, $L_0 = \C[[t]]$. Then to any lattice $L \in
K$ we associate a line $\det(L|L_0)$ and a $\C^\times$-torsor
$\det(L|L_0)^\times$. The group $GL_\infty$ acts transitively on
the set of lattices, but this action does not lift to the union of
these $\C^\times$-torsors. In fact, what lifts is a non-trivial
central extension $\wh{GL}_\infty$ of $GL_\infty$. This is clear
from the definition of this central extension given in \S
\ref{hat} (as the group $\wh{GL}'_\infty$).

We can use these geometric objects to construct a projective
representation of $GL_\infty$ on a vector space. For this we
present the set of lattices as the set of $\C$-points of an
ind-scheme $\on{Gr}$ (see \S \ref{1-Gr}) and the set of torsors as
the set of $\C$-points of a principal $\bbG_m$-bundle ${\mc
L}^\times$ on $\on{Gr}$. Here ${\mc L}$ is the determinant line
bundle introduced in \S \ref{det lb}. Now we can take the space of
global sections of ${\mc L}^*$, and the corresponding dual space
is the Fock representation $\bigwedge$ of $\wh{GL}_\infty$ (see
formula \eqref{geom Fock}).

\medskip

We would like to imitate this in the 2-dimensional case. The
set-theoretic part of the story is straightforward (see also
\cite{AK}). We have the set of lattices in $\bbK$, and the group
$GL_{\infty,\infty}$ acts transitively on it. Let us pick a
lattice, for instance, $\bbL_0 = {\mc O}_{\bbK}$. Then to any
lattice $\bbL$ we associate the $\C^\times$-gerbe
$\calD(\bbL|\bbL_0)$. The action of $GL_{\infty,\infty}$ on the
set of lattices lifts to an action of the 2-group
$\bbG\bbL_{\infty,\infty}$ on these gerbes. If we could realize
these objects algebro-geometrically, then we would be able to
construct a representation of $\bbG\bbL_{\infty,\infty}$ (that is,
a gerbal representation of $GL_{\infty,\infty}$) on the
corresponding category of ``global sections''. At the moment, we
do not know how to do this, but here are some indications of how
this could be done.

\medskip

For a Tate vector space $V$ with a countable basis of
neighborhoods of $0$, we will denote by $\Gr(V)$ the moduli space
of closed Tate vector subspaces of $V$. The corresponding functor
from commutative $\bbC$-algebras to sets is defined as follows:
for any commutative $\bbC$-algebra $R$,
\begin{equation}\label{moduli of Tate vector subspaces}
\Gr(V)(R)=\{\mbox{Tate } R\mbox{-modules that are direct summands of
}R\wh{\otimes}V\}.
\end{equation} Therefore, a map $\spec
R\to\Gr(V)$ is equivalent to a Tate $R$-module (see \cite{Dr}, \S
3.2.1 for the definition of Tate $R$-modules), which is a direct
summand of $R\wh{\otimes}V$. One can show that $\Gr(V)$ is a
$\bbC$-space. Observe that for $V=K$ the infinite Grassmannian
$\Gr$ defined in \S \ref{1-Gr} is a subspace of $\Gr(K)$. While
the later consists of all closed Tate vector subspaces of $K$, the
former only consists of those which are linearly compact.

If $M$ is a Tate $R$-module, then there is a $\bbG_m$-gerbe over
$\spec R$ in Nisnevich topology, namely, the gerbe $\calD_M$ of
determinantal theories of $M$ (see \cite{Dr}, \S 3.6). Therefore,
there is a tautological $\bbG_m$-gerbe over $\Gr(V)$, which
assigns to any $u: \on{Spec} R\to\Gr(V)$ the gerbe of
determinantal theories of the Tate $R$-module corresponding to
$u$. Denote this $\bbG_m$-gerbe over $\Gr(V)$ by $\calD_{\Gr(V)}$.

Now denote by $2\Gr$ the sought-after moduli of lattices of
$\bbK=\bbC\ppart\ppars$. It is natural to define it as the direct
limit
\[2\Gr=\lim\limits_{\overrightarrow{N}}\Gr(s^{-N}
\calO_\bbK/s^N\calO_\bbK).\] The tautological gerbes over these
spaces should be compatible with the pull-backs under the
embeddings
\[\Gr(s^{-N}\calO_\bbK/s^N\calO_\bbK)\to
\Gr(s^{-N-1}\calO_\bbK/s^{N+1}\calO_\bbK),\] and therefore we
should obtain a $\bbG_m$-gerbe $\calD_{2\Gr}$ over $2\Gr$.

Alternatively, consider the Cartesian square $2\Gr\times 2\Gr$.
Then there should be a tautological $\bbG_m$-gerbe over it, whose
fiber over $(\bbL,\bbL')$ is the gerbe $\calD(\bbL|\bbL')$ as
defined in \S \ref{dg}. The restriction of this gerbe to
$2\Gr\times\{\calO_\bbK\}$ is the above gerbe $\calD_{2\Gr}$.

The group $GL_{\infty,\infty}$ acts transitively on the set of
$\C$-points of $2\Gr$ (see Lemma \ref{transitive}), and this
action lifts to a tautological action of $\GL_{\infty,\infty}$
(see \S \ref{BGm central extension}) on $\calD_{2\Gr}$.

Recall that for a $\bbG_m$-gerbe $\calG$ over a scheme $S$, there
is a notion of a $\calG$-twisted $\calO$-module on $S$ (see, e.g.,
\cite{Li}). We denote the category of $\calG$-twisted
$\calO$-modules on $S$ by $\QCoh_S(\calG)$. The $\bbG_m$-gerbe
$\calG$ also gives rise to a sheaf of abelian categories over $S$,
$\calC:=\calG\underset{B\bbG_m}\times \on{Vect}$, and
$\QCoh_S(\calG)$ may be viewed as the category of global sections
of $\calC$. Since there is a an action of
$\bbG\bbL_{\infty,\infty}$ on $\calD_{2\Gr}$, we should obtain a
(genuine) representation of $\bbG\bbL_{\infty,\infty}$ (and hence
a gerbal representation of $GL_{\infty,\infty}$) on
$\QCoh_{2\Gr}(\calD_{2\Gr})$. In addition, we expect that there is
a $GL_{\infty,\infty}$-equivariant embedding of categories
$\calC_{\calO_\bbK}^{\on{ss}}\to \QCoh_{2\Gr}(\calD_{2\Gr})$.

\medskip

This is still a very rough scenario in which many details need to be
worked out. But let us look at the following simplified version, which
is its 1-dimensional analogue. Let $V$ be again a Tate vector space,
with a countable basis of neighborhoods of $0$. Let $\Gr(V)$ be the
moduli space of closed Tate vector subspaces of $V$ as defined in
(\ref{moduli of Tate vector subspaces}) and $\calD_{\Gr(V)}$ the
tautological $\bbG_m$-gerbe over it. Let $GL(V)$ be the group of
continuous automorphisms of $V$. Note that $GL(V)=GL_\infty$ if
$V=K$. Now $GL(V)$ acts (genuinely) on $\Gr(V)$, as well as on
$\calD_{\Gr(V)}$. Therefore, there is a genuine representation of
$GL(V)$ on $\QCoh_{\Gr(V)}(\calD_{\Gr(V)})$. On the other hand, we
recall that $\calC_V$ is the category of discrete Clifford modules
over $\Cl_V$ (see \S \ref{another descr}), and there is a genuine
representation of $GL(V)$ on $\calC_V$ (since $GL(V)$ acts by
automorphisms of $\Cl_V$). We claim that there is a
$GL(V)$-equivariant embedding
$\calC_V\to\QCoh_{\Gr(V)}(\calD_{\Gr(V)})$.

To see that, we first need a strengthening of Lemma \ref{equi},
which can be proven using the results in \cite{BBE}, \S 2.14-2.15.
Let $R$ be a commutative ring and $M$ a Tate $R$-module.
Consider the Clifford algebra $\Cl_M=\Cl_R(M\oplus M^*)$ and denote
by $\calC_M$ the abelian category of discrete modules of $\Cl_M$.
Let $\calD_M$ be the $\bbG_m$-gerbe of determinantal theories of
$M$ in Nisnevich topology on $\on{Spec} R$. We have the following
\begin{prop} The category $\calC_M$ is naturally
equivalent to the category of $\calD_M$-twisted ${\mc O}$-modules
over $\spec R$.
\end{prop}

Now we will associate to every lattice $L\subset V$ of $V$ an
object $\calF_L\in \QCoh_{\Gr(V)}(\calD_{\Gr(V)})$. Let $u:\spec
R\to \Gr(V)$ be a morphism given by $M\subset R\wh{\otimes}V$. For
any $L$, the tensor product $R\wh{\otimes} L$ is a lattice in
$R\wh{\otimes} V$, and $L_M:=(R\wh{\otimes} L)\cap M$ is a lattice
in $M$. Let $M_{L_M}$ be the vacuum module over $\Cl_M$ induced
from the trivial representation of $\bigwedge(L_M\oplus
L_M^\perp)$. By the above proposition, we obtain an object in
$\QCoh_{\spec R}(\calD_M)$, which we denote by $u^*\calF_L$. Then
the collection $\{u^*\calF_L\}_{u:\spec R\to\Gr(V)}$ define the
desired object $\calF_L\in \QCoh_{\Gr(V)}(\Gr(V))$.

Finally, there is a $GL(V)$-equivariant functor $\calC_V\to
\QCoh_{\Gr(V)}(\calD_{\Gr(V)})$ that will send $M_L$ to $\calF_L$,
where $M_L\in\calC_V$ is the vacuum module of $\Cl_V$ associated
with $L$.

Thus, we see that a simplified version of our proposal does work.

\section{Cohomology of $GL_{\infty,\infty}$ and related
groups} \label{coh}

In the previous section we have constructed a gerbal
representation of the group $GL_{\infty,\infty}$ on a certain
category of modules over a Clifford algebra. It is important to
identify the third cohomology class arising in this gerbal
representation. In this section, we first discuss the relevant
cohomology groups. In the next section we will show that there
exists of a particular non-zero cohomology class $[E_3]\in
H^3(GL_{\infty,\infty},\bbC^\times)$ and this class corresponds to
the gerbal representation of $GL_{\infty,\infty}$ on
$\calC_{\calO_\bbK}^{\on{ss}}$ constructed in Theorem \ref{main1}.

Analogous results on the cohomologies of the Lie algebra
$\gl_{\infty,\infty}$ and related Lie algebras will be discussed
in \cite{next}.

\subsection{Lie algebras and Lie groups
of matrices}

\subsubsection{Lie algebras of matrices} Let $R$ be an
associative (not necessarily unital) $\bbC$-algebra. By
definition, $R$ is a $\bbC$-vector space, with a ring structure
such that the multiplication $R\times R\to R$ is $\bbC$-bilinear.

We will denote by $R\ppart$ the $\bbC$-algebra of Laurent series
with coefficients in $R$. We will just regard it as a topological
{\em right} $R$-module, endowed with the $t$-adic topology. Let
$R[[t]]$ denote the $\bbC$-algebra of power series with
coefficients in $R$. Then it is an open submodule of $R\ppart$.
Observe that we have the following polarization:
\[R\ppart=R[[t]]\oplus t^{-1}R[t^{-1}].\]

Recall that we define $\gl_\infty(-)$ as a functor from the
category of associative (not necessarily unital) $\bbC$-algebras
to itself by assigning to any $\bbC$-algebra $R$ the algebra of
continuous {\em right} $R$-module endomorphisms of $R\ppart$ (so
that $\gl_\infty(R)$ acts on $R\ppart$ from the left). Likewise,
we define $\gl_\infty^+(R)$ as the algebra of continuous {\em
right} $R$-module endomorphisms of $R[[t]]$, and define
$\gl_\frakf(R)$ to be the two-sided ideal of $\gl_\infty^+(R)$
consisting of discrete $R$-module endomorphisms of $R[[t]]$, i.e.,
endomorphisms $f:R[[t]]\to R[[t]]$ such that $\exists N$,
$f|_{t^NR[[t]]}=0$.

Finally, there is a natural map $\pi:\gl_\infty(R)\to
\gl^+_\infty(R)$ (not an algebra homomorphism) induced by the
natural projection $\pi:R\ppart\to R[[t]]$ with respect to the
decomposition $R\ppart=R[[t]]\oplus t^{-1}R[t^{-1}]$. Namely,
$\pi(A)\in\gl^+_\infty(R)$ is defined by the formula
\[\pi(A)x=\pi(Ax) \ \ \ \ \ \ \ \mbox{ for } x\in
R[[t]]\]

Now we generalize the construction in \S \ref{central extensions}
as follows. Define
\[\widetilde{\gl}_\infty(R)=\{(A,X)\in
\gl^+_\infty(R)\times \gl_\infty(R), A-\pi(X)\in \gl_\frakf(R)\}\]
One has the following exact sequence of associative algebras (and
therefore Lie algebras)

\begin{equation}\label{basicexactsequence}
0\to\gl(R)\stackrel{i}{\rightarrow}\widetilde{\gl}_\infty(R)\stackrel{p}{\rightarrow}\gl_\infty(R)\to
0\end{equation} where $i(A)=(A,0)$ and $p(A,X)=X$. One also has
the section of $p$ given by $X\mapsto(\pi(X),X)$.

Therefore, $H_*(\gl(R)),H_*(\gl^+_\infty(R)),H_*(\gl_\infty(R))$
have the structures of graded Hopf algebras.

\subsubsection{Groups of matrices}\label{groups
of matrices}

Likewise, we will define $GL_\infty(R)$ to be the group of
continuous {\em right} $R$-module automorphisms of $R\ppart$,
$GL_\infty^+(R)$ the group of continuous {\em right} $R$-module
automorphisms of $R[[t]]$, and $GL_\frakf$ the normal subgroup of
$GL_\infty^+(R)$ consisting of those $g$ such that $\exists N$,
$g|_{t^NR[[t]]}=\mathrm{id}$. Furthermore, let
$\widetilde{GL}_\infty(R)$ be the group of invertible elements in
$\widetilde{\gl}_\infty(R)$, defined in the previous subsection.
Unlike the case of Lie algebras where we have the short exact
sequence (\ref{basicexactsequence}), for groups we only have the
following sequence
\[1\to
GL_\frakf(R)\stackrel{i}{\to}\widetilde{GL}_\infty(R)\stackrel{p}{\to}
GL_\infty(R),\] where $i:GL_\frakf(R)\to\widetilde{GL}_\infty(R)$
is given by $i(a)=(a,1)$. (We recall that $p(a,g)=g$.) We will
have a detailed discussion of this sequence in \S \ref{Two exact
sequences of groups}.

If $R$ has the unit, and if we use the standard topological basis
$\{t^i\}$ of $R\ppart$, these groups can be written in the
following concrete terms:
\[GL_\frakf(R)=\{(a_{ij})_{i,j\geq 0}, a
\mbox{ invertible }, a_{ij}\in R, a_{ij}=\delta_{ij} \mbox{ for }
j\gg 0\},\]
\[GL_\infty^+(R)=\left\{\begin{array}{l}(a_{ij})_{i,j\geq 0}, a
\mbox{
invertible }, a_{ij}\in R| \forall m\geq 0, \exists n\geq 0, \\
\mbox{ such that whenever } i<m, j>n, a_{ij}=0
\end{array}\right\},\]
\[GL_\infty(R)=\left\{\begin{array}{l}(a_{ij})_{i,j\in\bbZ},
a \mbox{ invertible }, a_{ij}\in R| \forall m\in \bbZ, \exists
n\in\bbZ, \\ \mbox{ such that whenever } i<m, j>n, a_{ij}=0
\end{array}\right\}.\]

The {\em left} actions of these group on $R\ppart$ are given by
$at^j=\sum a_{ij}t^i$.

\begin{rmk}
Even if $R$ does not have the unit, we can still present these
groups in matrix forms. For example, $GL_\frakf(R)$ consists of
those $A=(A_{ij})_{i,j\geq 0}$, $A_{ij}=0$ for $j\gg 0$, such that
there exists some $B$, $A+B+AB=0$, where $AB$ is the usual matrix
multiplication.
\end{rmk}

Observe that $GL(R),GL^+_\infty(R), GL_\infty(R)$ have inner sums,
that is, group homomorphisms
\[\begin{array}{l}\oplus: GL(R)\times
GL(R)\to
GL(R)\\
\oplus:GL^+_\infty(R)\times GL^+_\infty(R)\to
GL^+_\infty(R)\\
\oplus:GL_\infty(R)\times GL_\infty(R)\to
GL_\infty(R)\end{array}\] defined as follows: for $a=(a_{ij}),
b=(b_{ij})$, $a\oplus b=c$, where
\[c_{ij}=\left\{\begin{array}{ll}a_{\frac{i}{2},\frac{j}{2}}&\mbox{ if }
i,j
\mbox{ even },\\
b_{\frac{i-1}{2},\frac{j-1}{2}}&\mbox{ if }i,j \mbox{ odd
},\\

                                0 &\mbox{ otherwise
}
\end{array}\right.\]

Let us remark that $\gl_\frakf(R), \gl_\infty^+(R)$ and
$\gl_\infty(R)$ also have inner sums, given by the same formula as
in the previous subsection.

\subsubsection{Notation}

As in \S \ref{central extensions}, for brevity,
$\gl_\frakf(\bbC)$, $\gl^+_\infty(\bbC)$, $\gl_\infty(\bbC)$ and
$\widetilde{\gl}_\infty(\bbC)$ will be denoted by $\gl_\frakf$,
$\gl^+_\infty$, $\gl_\infty$, and $\widetilde{\gl}_\infty$,
respectively. Likewise, $GL_\frakf(\bbC)$, $GL_\infty^+(\bbC)$,
$GL_\infty(\bbC)$ and $\widetilde{GL}_\infty(\bbC)$ will be
denoted by $GL_\frakf$, $GL_\infty^+$, $GL_\infty$ and
$\widetilde{GL}_\infty$), respectively. As in \S \ref{Gerbal
extensions}, we denote $GL_\infty(\gl_\infty)$ by
$GL_{\infty,\infty}$. Furthermore, we will denote
$GL_\frakf(\gl_\infty)$ by $GL_{\frakf,\infty}$.

\subsection{Computation of group
cohomology}\label{group cohomology}

We do not have a complete description of the cohomology for the
groups introduced above (unlike the Lie algebra case, which will
be treated in \cite{next}). However, we still obtain some
interesting cohomology classes that will be sufficient for our
purposes.

\subsubsection{Starting point}\label{starting
point} We do not know the full cohomology groups
$H^\bullet(GL_\frakf,\bbC^\times)$. However, the determinant
$\det:GL_\frakf\to\bbC^\times$ determines a class $[\det]\in
H^1(GL_\frakf,\bbC^\times)$, which is the starting point of the
following construction.

Next, we turn to $GL_\infty$ and $GL_{\infty,\infty}$. Recall the
definition of $\widetilde{GL}_\infty(R)$ from \S \ref{groups of
matrices}. We have

\begin{prop}\label{group acyclic}
For any $\bbC$-algebra $R$,
$H^\bullet(GL_\infty^+(R),\bbC^\times)=H^\bullet(\widetilde{GL}_\infty(R),
\bbC^\times)\cong\bbC^\times$.
\end{prop}
\begin{proof}
Let us in fact prove that
$H_*(GL_\infty^+(R),k)=H_*(\widetilde{GL}_\infty(R),k)\cong k$ for
any field $k$. The proposition then follows from the universal
coefficient theorem.

We first prove that $H_*(GL_\infty^+(R),k)\cong k$. This is in
fact proved in \cite{W}. For the sake of completeness, we
reproduce the proof here.

Recall that $H_*(GL^+_\infty(R),k)$ has the structure of a Hopf
algebra since we have the inner sum $"\oplus"$. We will show that
there is a group homomorphism $\tau:GL^+_\infty(R)\to
GL^+_\infty(R)$ such that for any $a\in GL^+_\infty(R)$,
$a\oplus\tau(a)=\tau(a)$. Using this fact, we prove
$H_*(GL_\infty^+(R),k)\cong k$ by induction. Assume that
$H_{n-1}(GL^+_\infty(R),k)=0$. Then for $x\in
H_n(\gl^+_\infty(R),k)$, $\Delta(x)=x\otimes 1+1\otimes x+ \sum
u_i\otimes v_i$, where $u_i,v_i\in H_k(GL^+_\infty(R)),
k=1,2,\ldots,n-1$. By induction, $u_i,v_i$ vanish and therefore,
$\Delta(x)=x\otimes 1+1\otimes x$. Since
\[\tau=\oplus\circ(\mbox{id}\oplus\tau)\circ\Delta:GL^+_\infty(R)\to
GL^+_\infty(R)\] one obtain that $\tau_*(x)=x+\tau_*(x)$ and
therefore $x=0$.

Now we construct the morphism $\tau$: for $a=(a_{ij})\in
GL^+_\infty(R)$, we set
\[\tau(a)_{ij}=\left\{\begin{array}{ll}a_{i/2,j/2}  &\mbox{ if } i,j
\mbox{
even};\\
a_{m,n} &\mbox{ if } i=2^km+2^{k-1}-1, j=2^kn+2^{k-1}-1;\\
    0
&\mbox{ otherwise}
    \end{array}\right.\]
It is easy to check that $\tau$ has the required properties.

Next, we show that $H_*(\widetilde{GL}_\infty(R),k)\cong k$. We
have the surjective group homomorphism
\[1\to
J(R)\to\widetilde{GL}_\infty(R)\to GL_\infty^+(R)\to 1\] sending
$(a,g)\to a$. The kernel is
\[J(R)=\{g=(g_{ij})_{i,j\in\bbZ}\in GL_\infty(R),
g_{ij}=\delta_{ij} \mbox{ for } j\gg 0\}\] Therefore, it is enough
to prove that $H_*(J(R),k)\cong k$. We could present $J(R)$ as
\[J(R)=\lim\limits_{\overrightarrow{n}}J_n(R)\]
where
\[J_n(R)=\{g=(g_{ij})_{i,j\in\bbZ}\in
GL_\infty(R), g_{ij}=\delta_{ij} \mbox{ for } j\gg n\}\] Observe
that all $J_n(R)$ is indeed isomorphic to $J_0(R)$, and therefore
it is enough to show that $H_*(J_0(R),k)\cong k$. This in fact
follows from the same argument as for $GL_\infty^+(R)$.
\end{proof}

\subsubsection{Two exact sequences
of groups}\label{Two exact sequences of groups}

Recall that from \S \ref{groups of matrices}, we have a left exact
sequence of groups
\[1\to
GL_\frakf(R)\stackrel{i}{\to}\widetilde{GL}_\infty(R)\stackrel{p}{\to}
GL_\infty(R).\] In general, the map $p:\widetilde{GL}_\infty(R)\to
GL_\infty(R)$is not surjective. We have the following:

\begin{prop}\label{degree map}
There is a well-defined surjective group homomorphism
\[\deg:GL_\infty(R)\to K_0(R^{\on{op}})\]
where $K_0(R^{\on{op}})$ is the Grothendieck group of the category
of finitely generated projective {\em right} $R$-modules.
\end{prop}
\begin{proof}
The existence of $\deg:GL_\infty(R)\to K_0(R^{\on{op}})$ and its
surjectivity can be proved in a way similar to \cite{FW}
Proposition 1.3. We need the following

\begin{lem}
For any $g\in GL_\infty(R)$, and any integer $M$, there exists an
integer $N$ such that $t^MR[[t]]\cdot g\supset t^NR[[t]]$ and
$\frac{t^MR[[t]]\cdot g}{t^NR[[t]]}$ is a finitely generated
projective $R$-module. (Then any $N'\geq N$ also satisfies this
property.)
\end{lem}

\begin{proof}
Since $g\in GL_\infty(R)$ is continuous, for any $m$ there exists
$n$ such that $t^mR[[t]]\supset gt^nR[[t]]$. Therefore, for a
given $M$, we have
\[gt^MR[[t]]\supset t^{n_1}R[[t]]\supset
gt^{n_2}R[[t]]\supset t^{n_3}R[[t]]\] We claim that $N=n_3$
satisfies the desired property. Indeed,
$\frac{gt^MR[[t]]}{t^{n_1}R[[t]]}$ has projective dimension one
(i.e., $\mathrm{Ext}^i(\frac{gt^MR[[t]]}{t^{n_1}R[[t]]},-)=0$ for
$i\geq 2$) since $\frac{gt^MR[[t]]}{t^{n_1}R[[t]]}$ has a free
resolution of length two
\[0\to
t^{n_1}R[[t]]\to gt^MR[[t]]\to \frac{gt^MR[[t]]}{t^{n_1}R[[t]]}\to
0\] Then $\frac{t^{n_1}R[[t]]}{gt^{n_2}R[[t]]}$ is projective
since there is the following short exact sequence
\[0\to\frac{t^{n_1}R[[t]]}{gt^{n_2}R[[t]]}\to
\frac{gt^MR[[t]]}{gt^{n_2}R[[t]]}\to\frac{gt^MR[[t]]}{t^{n_1}R[[t]]}\to
0\] and the middle term is free. Then
$\frac{gt^{n_2}R[[t]]}{t^{n_3}R[[t]]}$ is projective and finitely
generated because of the following short exact sequence
\[0\to
\frac{gt^{n_2}R[[t]]}{t^{n_3}R[[t]]}\to\frac{t^{n_1}R[[t]]}{t^{n_3}R[[t]]}\to\frac{t^{n_1}R[[t]]}{gt^{n_2}R[[t]]}\to0\]
and the middle term is a free {\em right} $R$-module of finite
rank.

Therefore
\[\frac{gt^MR[[t]]}{t^{n_3}R[[t]]}\cong\frac{gt^MR[[t]]}{gt^{n_2}R[[t]]}\oplus\frac{gt^{n_2}R[[t]]}{t^{n_3}R[[t]]}\]
is a finitely generated projective {\em right} $R$-module.
\end{proof}

Now we define $\deg:GL_\infty(R)\to K_0(R^{\on{op}})$ as
$[\frac{R[[t]]}{t^NR[[t]]}]-[\frac{gR[[t]]}{t^NR[[t]]}]$, where
$N$ is chosen so that $t^NR[[t]]\subset gR[[t]]$ and
$\frac{gR[[t]]}{t^NR[[t]]}$ is a finitely generated projective
{\em right} $R$-module. Apparently, this element in
$K_0(R^{\on{op}})$ does not depend on the choice of $N$, and
therefore gives a well-defined map $\deg$. It is also easy to
check that this is indeed a group homomorphism.

Next, we prove that $\deg$ is surjective. First, observe that any
element in $K_0(R^{\on{op}})$ can be represented as $[R^n]-[P]$
where $R^n$ is the free {\em right} $R$-module of rank $n$, and
$P$ is a finitely generated projective {\em right} $R$-module. Now
let $Q$ be a {\em right} $R$-module such that $P\oplus Q\cong R^m$
for some $m$. Now we define $g\in GL_\infty(R)$ so that
$\deg(g)=[R^n]-[P]$. For $s\geq r$, denote by $R_r^s$ the free
{\em right} $R$-module of rank $s-r$ $Rt^r\oplus
Rt^{r+1}\oplus\cdots\oplus Rt^{s-1}$. Write
$R_{n+km}^{n+(k+1)m}=P_k\oplus Q_k$, where $P_k\cong P$ and
$Q_k\cong Q$. Then we define an isomorphism $g:R\ppart\to R\ppart$
by sending $R_{km}^{(k+1)m}$ isomorphically to $P_{k-1}\oplus
Q_k$, so that $A[[t]]\cong\prod_{k\geq 0}R_{km}^{(k+1)m}$ is
mapped isomorphically to $P_{-1}\oplus\prod_{k\geq 0}(Q_k\oplus
P_k)\cong P\oplus t^nR[[t]]$. It is clear that such an element $g$
is indeed in $GL_\infty(R)$, and
$[\frac{R[[t]]}{t^nR[[t]]}]-[\frac{gR[[t]]}{t^nR[[t]]}]=[R^n]-[P]$.
\end{proof}

\begin{prop}\label{e.s.
of groups} Let $\deg:GL_\infty(R)\to K_0(R^{\on{op}})$ be the
homomorphism constructed in Proposition \ref{degree map}. Then we
have the following exact sequence of groups:
\[1\to
GL_\frakf(R)\stackrel{i}{\to}\widetilde{GL}_\infty(R)\stackrel{p}{\to}
GL_\infty(R)\stackrel{\deg}{\to}K_0(R^{\on{op}})\to 0.\]
\end{prop}
\begin{proof}
It is enough to check the exactness at $GL_\infty(R)$. So we have
to prove that for $g\in GL_\infty(R)$, we have $\deg(g)=0$ if and
only if $\pi g=a-f$ for some $a\in GL_\infty^+(R)$ and
$f\in\gl_\frakf(R)$.

First, assume that $\pi g=a-f$ with $a\in GL_\infty^+(R)$ and
$f\in\gl_\frakf(R)$. Then there exists some $N$, large enough so
that $f|_{t^NR[[t]]}=0$. For $N$ large enough, $gt^NR[[t]]\subset
R[[t]]$ and therefore
$gt^NR[[t]]=\pi(gt^NR[[t]])=(a-f)t^NR[[t]]=at^NR[[t]]$. We have
\begin{eqnarray*}
\deg(g^{-1})=&[\frac{R[[t]]}{t^NR[[t]]}]-[\frac{g^{-1}R[[t]]}{t^NR[[t]]}]\\
=&[\frac{R[[t]]}{t^NR[[t]]}]-[\frac{R[[t]]}{gt^NR[[t]]}]\\
=&[\frac{R[[t]]}{t^NR[[t]]}]-[\frac{R[[t]]}{at^NR[[t]]}]\\
=&[\frac{R[[t]]}{t^NR[[t]]}]-[\frac{aR[[t]]}{at^NR[[t]]}]=0
\end{eqnarray*}

Conversely, assume that $\deg(g^{-1})=\deg(g)=0$, and therefore
$[\frac{R[[t]]}{t^NR[[t]]}]=[\frac{g^{-1}R[[t]]}{t^NR[[t]]}]$ for
$N$ large enough. It is well-known that if $[P]=[Q]$ in
$K_0(R^{\on{op}})$, then there exists some $n$ such that $P\oplus
R^n\cong Q\oplus R^n$ as {\em right} $R$-modules. Therefore, we
could choose $N$ large enough so that
\[\frac{R[[t]]}{t^NR[[t]]}\cong\frac{g^{-1}R[[t]]}{t^NR[[t]]}\cong\frac{R[[t]]}{gt^NR[[t]]}.\]
We pick such an isomorphism. Observe that since
$\frac{R[[t]]}{gt^NR[[t]]}$ is projective, there exists an
$R$-submodule $M\subset R[[t]]$, isomorphic to
$\frac{R[[t]]}{gt^NR[[t]]}$, such that $R[[t]]=gt^NR[[t]]\oplus
M$. Then we have an isomorphism
$f:\frac{R[[t]]}{t^NR[[t]]}\cong\frac{R[[t]]}{gt^NR[[t]]}\cong M$.

Now we define an isomorphism $a\in R[[t]]\to R[[t]]$ as follows:
\[R[[t]]\cong
t^NR[[t]]\oplus R[[t]]/t^NR[[t]]\stackrel{g\oplus f}{\to}
gt^NR[[t]]\oplus M=R[[t]].\] It is clear from the definition that
$a\in GL_\infty^+(R)$ and $\pi g-a|_{t^NR[[t]]}=0$, i.e. $\pi
g-a\in\gl_\frakf(R)$.
\end{proof}

Denote $GL_\infty(R)^0:=\ker(GL_\infty(R)\to K_0(R^{\on{op}}))$.
Therefore, we have a short exact sequence

\begin{equation}\label{basicexactsequenceforgroups}
0\to GL(R)\to \widetilde{GL}_\infty(R)\to GL_\infty(R)^0\to 1
\end{equation}
$GL_\infty(\bbC)^0$ will be denoted by $GL_\infty^0$ for
simplicity.

We can construct another exact sequence of groups. Apply the
functor $GL_\frakf$ to the short exact sequence of algebras
(\ref{basicexactsequence}), we obtain
\[1\to
GL_\frakf(\gl_\frakf(R))\to
GL_\frakf(\widetilde{\gl}_\infty(R))\to
GL_\frakf(\gl_\infty(R)).\] We have

\begin{prop}
The above exact sequence extends to
\[1\to
GL_\frakf(\gl_\frakf(R))\to
GL_\frakf(\widetilde{\gl}_\infty(R))\to
GL_\frakf(\gl_\infty(R))\to K_0(R^{\on{op}})\to 0.\]
\end{prop}
\begin{proof}
For any $n$, let $GL_{\frakf,n}(R)$ be the subgroup of
$GL_\frakf(R)$ consisting of those $a$ for which
$a|_{t^nR[[t]]}=\mathrm{id}$. Then it maps surjective to
$GL_n(R)$, the group of invertible $n\times n$ matrices with
coefficients in $R$, by $p_n$, where
$$
p_n(a=(a_{ij})_{i,j\geq 0})=(a_{ij})_{0\leq i,j\leq n-1}.
$$
Denote by $i_{n,n+1}$ the inclusion of $GL_{\frakf,n}(R)$ to
$GL_{\frakf,n+1}(R)$.

We have isomorphisms $\varphi_n: GL_n(\gl_\frakf(R))\cong
GL_\frakf(R),
GL_n(\widetilde{\gl}_\infty(R))\cong\widetilde{GL}_\infty(R)$, and
$GL_n(\gl_\infty(R))\cong GL_\infty(R)$ making the following
diagram commutative

\[\begin{CD}
1@>>>GL_{\frakf,n}(\gl_\frakf(R))@>>>
GL_{\frakf,n}(\widetilde{\gl}_\infty(R))@>>>
GL_{\frakf,n}(\gl_\infty(R))\\
@.@Vp_nVV@Vp_nVV@Vp_nVV\\
1@>>>GL_n(\gl_\frakf(R))@>>>GL_n(\widetilde{\gl}_\infty(R))@>>>
GL_n(\gl_\infty(R))\\
@.@V\varphi_nVV@V\varphi_nVV@V\varphi_nVV\\
1@>>>GL_\frakf(R)@>>>\widetilde{GL}(R)@>>>GL_\infty(R)@>\deg>>
K_0(R^{\on{op}})@>>>0
\end{CD}\]

Using this diagram, it is easy to see that the sequence

\[GL_{\frakf,n}(\widetilde{\gl}_\infty(R))\to
GL_{\frakf,n}(\gl_\infty(R))\stackrel{\varphi_n\circ
p_n}{\to}K_0(R^{\on{op}})\] is exact at
$GL_{\frakf,n}(\gl_\infty(R))$. Furthermore, it is easy to check
that, although $\varphi_n\circ p_n\neq\varphi_{n+1}\circ
p_{n+1}\circ i_{n,n+1}$, we have $\deg\circ\varphi_n\circ
p_n=\deg\circ(\varphi_{n+1}\circ p_{n+1}\circ i_{n,n+1})$. This
proves the proposition.
\end{proof}

\begin{rmk}
Observe that we have an isomorphism
$$
GL_\frakf(R)\cong
GL(\gl_\frakf(R))=\lim\limits_{\stackrel{\longrightarrow}{n}}
GL_n(\gl_\frakf(R))
$$
(not compatible with any $\varphi_n$ above).
\end{rmk}

Let $GL_{\frakf}(\gl_\infty(R))^0$ be the kernel of
$GL_\frakf(\gl_\infty(R))\to K_0(R^{\on{op}})$. Thus, we have
another exact sequence,

\begin{equation}\label{2ndbasicexactsequenceforgroups}1\to
GL_\frakf(\gl_\frakf(R))\to
GL_\frakf(\widetilde{\gl}_\infty(R))\to
GL_\frakf(\gl_\infty(R))^0\to 1.
\end{equation}

\subsubsection{$H^2$ and central extensions}\label{H2 and
central extensions}

Let $R=\bbC$. Then we have the exact sequence (\ref{e.s.}),
\[1\to
GL_\frakf\to\widetilde{GL}_\infty\to GL_\infty^0\to 1.\] Pushing
out this sequence by $\det$, we obtain a central extension of
$GL_\infty^0$.
\[1\to\bbC^\times\to\widehat{GL}_\infty^0\to GL_\infty^0\to
1,\] which we constructed in \S \ref{central extensions}. The
cohomology class $[C^0_2]\in H^2(GL_\infty^0,\bbC^\times)$ is
obtained by the transgression of $[\det]\in
H^1(GL_\frakf,\bbC^\times)$. We claim that it is non-trivial. This
is because in the Lyndon-Hochschild-Serre spectral sequence
associated to (\ref{e.s.}), $[C^0_2]\in E_2^{2,0}$ is obtained by
transgression of $[\det]\in E_2^{0,1}$. If $[C^0_2]=0$, then
$[\det]$ would survive to the $E_\infty$ term, so that
$H^1(\widetilde{GL}_\infty,\bbC^\times)\neq 0$, which contradicts
Proposition \ref{group acyclic}.

It was proved in Proposition \ref{uniqueness of central extension}
that such central extension can be extended to a unique central
extension $\widehat{GL}_\infty$ of $GL_\infty$.

We have a similar story for $GL_{\frakf,\infty}$. We only sketch
it, since it is almost a word-for-word repetition of the story of
$GL_\infty$. Let $R=\bbC$ in
(\ref{2ndbasicexactsequenceforgroups}), we obtain\[1\to
GL_\frakf(\gl_\frakf)\to GL_\frakf(\widetilde{\gl}_\infty)\to
GL_{\frakf,\infty}^0\to 1.\] We claim that there is a well-defined
group homomorphism $\det:GL_\frakf(\gl_\frakf)\to\bbC^\times$.
Indeed, for any $n$, we have the homomorphisms
\[GL_{\frakf,n}(\gl_\frakf)\to
GL_n(\gl_\frakf)\cong GL_\frakf\stackrel{\det}{\to} \bbC^\times,\]
which are compatible with the embeddings $i_{n,n+1}$. This gives
$\det:GL_\frakf(\gl_\frakf)\to\bbC^\times$. Now pushing-out the
above sequence by $\det$, we obtain

\begin{equation}\label{extension of
GLfrakfinfty0} 1\to\bbC^\times
\to\widehat{GL}_{\frakf,\infty}^0\to GL_{\frakf,\infty}^0\to 1.
\end{equation}

One can similarly prove the following:

\begin{prop}
(i) The extension (\ref{extension of GLfrakfinfty0})
is non-trivial. The cohomology class $[D^0_2]\in
H^2(GL^0_{\frakf,\infty},\bbC^\times)$ is the transgression of
$[\det]\in H^1(GL_{\frakf}(\gl_\frakf),\bbC^\times)$.

(ii) There is a unique (up to an isomorphism) central extension
\begin{equation}\label{extension of
GLfrakfinfty} 1\to\bbC^\times \to\widehat{GL}_{\frakf,\infty}\to
GL_{\frakf,\infty}\to 1
\end{equation}
whose restriction to $GL_{\frakf,\infty}^0$ is (\ref{extension of
GLfrakfinfty0}). The cohomology class corresponding to this
central extension is denoted by $[D_2]$.
\end{prop}

\begin{proof}
The only fact we use is that
$H^1(GL_\frakf(\widetilde{\gl}_\infty),\bbC^\times)=0$. This is
because $\widetilde{GL}_\infty\subset
GL_\frakf(\widetilde{\gl}_\infty)$ and
$H^1(\widetilde{GL}_\infty,\bbC^\times)=0$. Then (i) follows from
the same argument as the proof of the non-triviality of $[C^0_2]$,
and (ii) follows from the proof of Proposition \ref{uniqueness of
central extension}.
\end{proof}

As before, the automorphism of (\ref{extension of GLfrakfinfty})
is $H^1(GL_{\frakf,\infty},\bbC^\times)=\bbC^\times$.

\subsubsection{$\wh{GL}_{\frakf,\infty}$
as the central extension of $GL_{\frakf,\infty}$ by determinant
lines}\label{dl} We can also regard $\wh{GL}_{\frakf,\infty}$ as
the central extension of $GL_{\frakf,\infty}$ by determinant
lines. Recall that $\calO_\bbK=\bbC\ppart[[s]]$. This is a lattice
of $\bbK$ (see Definition \ref{lat}). Let $L, L'$ be two secondary
lattices of $\calO_\bbK$ (see Definition \ref{2-lattice}) We assume
that there exists some $n$ such that $L\cap s^n\calO_\bbK=L'\cap
s^n\calO_\bbK$. Observe that although $L$ and $L'$ may not be
commensurable with each other (an example being given below), we
can still define the determinant line $\det(L|L')$. This is
because for any $m\geq n$, there is a canonical isomorphism
\begin{equation}\label{independent of n}
\det(\frac{L}{L\cap s^m\calO_\bbK}|\frac{L'}{L'\cap
s^m\calO_\bbK})\stackrel{\cong}{\to}\det(\frac{L}{L\cap
s^n\calO_\bbK}|\frac{L'}{L'\cap s^n\calO_\bbK}).
\end{equation}

Then we can define
\begin{equation}\label{determinantal
line}\det(L|L')
:=\lim\limits_{\overleftarrow{m}}\det(\frac{L}{L\cap
s^m\calO_\bbK}|\frac{L'}{L'\cap s^m\calO_\bbK}).
\end{equation}
It is clear that if $L$ and $L'$ are commensurable, the new
definition coincides with the old one. This motivates us to define

\begin{dfn}\label{pseudo commensurability}Let $\bbL$ be a
lattice in $\bbK$ and $L,L'\subset\bbL$ are two secondary
lattices. We will call $L,L'$ {\em pseudo commensurable} with each
other if there exists some lattice $\bbL'\subset\bbL$ such that
$L\cap\bbL'=L'\cap\bbL'$. In this case, we define $\det(L|L')$ by
formula (\ref{determinantal line}).
\end{dfn}

Let us emphasize that the important properties of $\det(L|L')$ are
those stated in Remark \ref{composition of determinantal lines}.
That is, for $L,L',L''$ pseudo commensurable with each other,
there is a canonical isomorphism
\[\gamma_{L,L',L''}:\det(L|L')\otimes\det(L'|L'')\cong\det(L|L'')\]
such that for any $L,L',L'',L'''$,
$\gamma_{L,L',L'''}\gamma_{L',L'',L'''}=\gamma_{L,L'',L'''}
\gamma_{L,L',L''}$.

Now let $GL_\infty^+(\gl_\infty)$ acts on $\calO_\bbK$ via the
following formula. If we represent an element in
$GL_\infty^+(\gl_\infty)$ as $a=(a_{ij,mn})_{i,j\geq
0,m,n\in\bbZ}$, then
\begin{equation}\label{action
of GL+inftyglinfty} a(t^ns^j)=\sum_{m\in\bbZ,n\geq
0}a_{ij,mn}(t^ms^i).
\end{equation}
The following lemmas are easy to check.

\begin{lem}\label{sec lat}
For any $a\in GL_\infty^+(\gl_\infty)$, and a secondary lattice
$L\subset\calO_\bbK$, $aL$ is a secondary lattice. In fact,
$GL_\infty^+(\gl_\infty)$ acts transitively on the set of secondary
lattices in $\calO_\bbK$. Furthermore, if $a\in
GL_{\frakf,\infty}\subset GL_\infty^+(\gl_\infty)$, then $aL$ is
pseudo commensurable with $L$, i.e., $\exists n$ such that $L\cap
s^n\calO_\bbK=aL\cap s^n\calO_\bbK$.
\end{lem}

\begin{lem}If
$L,L'$ are two secondary lattices in $\calO_\bbK$, pseudo
commensurable with each other, then for $a\in
GL_\infty^+(\gl_\infty)$, $aL,aL'$ are pseudo commensurable, and
$a$ induces an isomorphism
$\det(L|L')\stackrel{\cong}{\to}\det(aL|aL')$.
\end{lem}

\begin{ex}\label{ex}
Let $L_0=\bbC[[t]][[s]]$. Observe that if we let $a\in
GL_{\frakf,\infty}$ which is defined by
\[at^i=\sum_{j\geq 0}s^jt^{i-j}, \ \ \ \ \ \ \ \ a(s^it^j)=s^it^j,
\]
then $aL_0\cap L_0=s\bbC[[t]][[s]]$. Therefore, $aL_0$ is not
commensurable with $L_0$.
\end{ex}

\bigskip

Recall the definition of $GL_{\frakf,n}(\gl_\infty)$, which will
be denoted by $GL_{n,\infty}$ in what follows for simplicity. This
is a subgroup of $GL_{\frakf,\infty}$ consisting of
$a:\gl_\infty[[s]]\to\gl_\infty[[s]]$ such that
$a|_{s^n\gl_\infty[[s]]}=\mathrm{id}$. We denote the restriction
of the central extensions $\wh{GL}_{\frakf,\infty}$ of
$GL_{\frakf,\infty}$ to $GL_{n,\infty}$ by $\wh{GL}_{n,\infty}$.
Let $L$ be a secondary lattice of $\calO_\bbK$. As shown in
Proposition \ref{identification of two central extensions}, the
central extension $\wh{GL}_{n,\infty}$ can be interpreted as the
group consisting of $(a,e)$, where $a\in GL_{n,\infty}$ and
$$e\in\det(\frac{aL}{aL\cap
s^n\calO_\bbK}|\frac{L}{L\cap
s^n\calO_\bbK})^\times=\det(\frac{aL}{L\cap
s^n\calO_\bbK}|\frac{L}{L\cap s^n\calO_\bbK})^\times.$$ We thus
obtain

\begin{prop}\label{another
presentation}The central extensions $\wh{GL}_{\frakf,\infty}$ of
$GL_{\frakf,\infty}$ can be interpreted as the group consisting of
$(a,e)$, where $a\in GL_{\frakf,\infty}$ and
$e\in\det(aL|L)^\times$ for any secondary lattice
$L\subset\calO_\bbK$.
\end{prop}

Observe that $GL_\infty^+(\gl_\infty)$ acts on
$GL_{\frakf,\infty}$ by conjugation $\rho$, i.e.,
$\rho_g(h)=ghg^{-1}$ for $g\in GL^+_{\infty}(\gl_\infty), h\in
GL_{\frakf,\infty}$.

\begin{cor}\label{D2
invariant under g} Let $g\in GL_\infty^+(\gl_\infty)$. Then the
central extension $\bbC^\times\to\wh{GL}_{\frakf,\infty}^g\to
GL_{\frakf,\infty}$ obtained by pullback of
$\bbC^\times\to\wh{GL}_{\frakf,\infty}\to GL_{\frakf,\infty}$
along $\rho_g:GL_{\frakf,\infty}\to GL_{\frakf,\infty}$ is
isomorphic to $\wh{GL}_{\frakf,\infty}$.
\end{cor}

\begin{proof} If we
present $\wh{GL}_{\frakf,\infty}$ as the group consisting of
elements $(a,e)$, where $a\in GL_{\frakf,\infty}$ and
$e\in\det(aL|L)^\times$ for any secondary lattice
$L\subset\calO_\bbK$, then $\wh{GL}_{\frakf,\infty}^g$ will have
similar presentation where $L$ is replaced by $g^{-1}L$.
\end{proof}

\subsubsection{An
extension of $GL_{\infty,\infty}$ by
   $GL_{\frakf,\infty}$} We first claim the
following:

\begin{lem}
$K_0(\gl^{\on{op}}_\infty)=0$ so that
$GL_{\infty,\infty}^0=GL_{\infty,\infty}$.
\end{lem}

\begin{proof} We define
a functor from the category of finitely generated projective
$\gl_\infty$-modules to the category of Tate vector spaces by
$P\mapsto P\otimes_{\gl_\infty}K$. (We recall hat $K=\bbC\ppart$
and $\gl_\infty=\mathrm{End}K$.) Fixing a finite set $\{p_i\}$ of
generators of $P$ as a {\em right} $\gl_\infty$-module, a basis of
open neighborhoods of $0\in P\otimes_{\gl_\infty}K$ is given by
$\sum_ip_i\otimes U_i$, where $U_i$ vary in the set of open
subspaces of $K$. It is easy to see that the topology on
$P\otimes_{\gl_\infty}K$ is independent of the choice of
$\{p_i\}$. Observe that $\Hom(K,P\otimes_{\gl_\infty}K)$ is a {\em
right} $\Hom(K,K)=\gl_\infty$-module. We first show that
$\Hom(K,P\otimes_{\gl_\infty}K)\cong P$.

Indeed, using an isomorphism $K\oplus K\cong K$ of Tate vector
spaces, we obtain $\gl_\infty\oplus\gl_\infty=\Hom(K,K\oplus
K)\cong\Hom(K,K)=\gl_\infty$. Therefore, any finitely generated
projective {\em right} $\gl_\infty$ is a direct summand of
$\gl_\infty$, in particular a cyclic {\em right}
$\gl_\infty$-module. Let $p\in\gl_\infty$ be a generator of $P$
(so $P=p\gl_\infty\subset\gl_\infty$). We thus obtain that
$\Hom(K,P\otimes_{\gl_\infty}K)$ is a direct summand of
$\Hom(K,\gl_\infty\otimes_{\gl_\infty}K)\cong\gl_\infty$. Observe
that under the last isomorphism, the map
$\varphi\in\Hom(K,\gl_\infty\otimes_{\gl_\infty}K)$ defined by
$\varphi(v)=p\otimes v$ goes to $p$. This proves that
$\Hom(K,P\otimes_{\gl_\infty}K)\cong P$.

Now, using the fact that for any Tate subspace $W\subset K$, we
have $W\oplus K\cong K$ as Tate vector spaces, we obtain that
\[P\oplus \gl_\infty\cong\Hom(K,P\otimes_{\gl_\infty}K\oplus
K) \cong\Hom(K,K)=\gl_\infty.\] This implies that in
$K_0(\gl^{\on{op}}_\infty)$, $[P]=0$.
\end{proof}

\begin{rmk}(i) For any $R$, we thus obtain a functor from
the category $\calP_R$ of finitely generated projective
$\gl_\infty(R)$-modules to the category $\calT_R$ of Tate $R$-modules
(see \cite{Dr}, Definition 3.2.1, for the definition of Tate
$R$-modules). At the level of Grothendieck groups, this gives an
isomorphism $K_0(\gl_\infty(R))=K_0(\calP_R)\cong K_0(\calT_R)\cong
K_{-1}(R)$. (See \cite{Dr}, Theorem 3.6.)

(ii) A variation of the main result in \cite{W} generalizes the
result in Remark (i). Namely, the functor $\gl_\infty$ shifts the
algebraic $K$-theory of a ring by degree one, i.e.,
$K_*(\gl_\infty(R))=K_{*-1}(R)$.

(iii) There is also an additive analogue of Remark (ii), see
\cite{FT3}, Proposition 4.1.5.
\end{rmk}

Now let $R=\gl_\infty$ in (\ref{basicexactsequenceforgroups}). By
the above lemma, we have an exact sequence of groups

\begin{equation}\label{extension of GLinftyinfty}
1\to GL_{\frakf,\infty}\to\widetilde{GL}_\infty(\gl_\infty)\to
GL_{\infty,\infty}\to 1.
\end{equation}

Since $H^\bullet(\widetilde{GL}_\infty(\gl_\infty))\cong\bbZ$ and
$H^1(GL_{\frakf,\infty},\bbZ)\cong\bbZ$, by the spectral sequence
associated to (\ref{extension of GLinftyinfty}), we obtain that
$H^2(GL_{\infty,\infty},\bbZ)\cong\bbZ$. The $\bbZ$-central
extension

\begin{equation}\label{Z-central extension
of GLinftyinfty}1\to\bbZ\to\wh{GL}_{\infty,\infty}\to
GL_{\infty,\infty}\to 1
\end{equation}
of $GL_{\infty,\infty}$ corresponding to the generator of
$H^2(GL_{\infty,\infty},\bbZ)$ can be obtained as the push-out of
(\ref{extension of GLfrakfinfty}) by
$\deg:GL_{\frakf,\infty}\to\bbZ$. This is the universal
$\bbZ$-central extension of $GL_{\infty,\infty}$.  Observe that
$\wh{GL}_{\infty,\infty}$ fits into the following exact sequence
of groups

\begin{equation}\label{extension
of whGLfrakfinfty}1\to
GL_{\frakf,\infty}^0\to\widetilde{GL}_\infty(\gl_\infty)\to
\wh{GL}_{\infty,\infty}\to1
\end{equation}

\begin{rmk}We
remark that the group $\wh{GL}_{\infty,\infty}$ plays a central
role in Osipov's work of reciprocity laws on algebraic surfaces
(see \cite{Osi1}).
\end{rmk}

In the next section, we will obtain a nontrivial cohomology class
$[E_3]\in H^3(GL_{\infty,\infty})$ (see Theorem \ref{H^3 of
GLinftyinfty}) from the sequence (\ref{extension of
GLinftyinfty}).

\section{Cohomology classes of gerbal representations of
  $GL_{\infty,\infty}$}    \label{coh cl}

In this section we compute the third cohomology class realized in
the gerbal representation of $GL_{\infty,\infty}$ on the category
$\calC_{\calO_\bbK}^{\on{ss}}$ constructed in Theorem \ref{main1}.
The main result is Theorem \ref{main2}. In order to prove this
result we realize this class as the cohomology class arising from
another gerbal representation of $GL_{\infty,\infty}$. We then use
some general results presented below on what we call ``gerbal
pairs of groups''.

\subsection{Gerbal pairs of groups}

In this section we develop a formalism which allows us to
calculate the cohomology classes corresponding to gerbal
representations in a certain situation. This is based on what we
call \emph{gerbal pairs of groups}. In this section, we will first
recall one construction of the Lyndon-Hochschild-Serre spectral
sequences. Then we will develop the general formalism of gerbal
pairs of groups. Finally, we will show that there is an action of
$\widetilde{GL}_\infty(\gl_\infty)$ on $\wh{GL}_{\frakf,\infty}$,
such that the groups
$(\widetilde{GL}_\infty(\gl_\infty),\wh{GL}_{\frakf,\infty})$
equipped with the short exact sequences (\ref{extension of
GLfrakfinfty}) and (\ref{extension of GLinftyinfty}) form a gerbal
pair in the sense of Definition \ref{gerbal pair}.

\subsubsection{The spectral sequence of a group
extension}\label{LHS s.s.}

Recall that if \[1\to H\to G\to K\to 1\] is an extension of
groups, then for a $G$-module $M$ the Lyndon-Hochschild-Serre
spectral sequence has the following second term:
\[E_2^{p,q}=H^p(K,H^q(H,M))\Rightarrow
H^{p+q}(G,M).\] We will recall one construction of this spectral
sequence given in Chapter II of \cite{HS}. It uses an appropriate
filtration on the complex of normalized cochains $C^\bullet(G,M)$.
Recall that for $M$ a $G$-module, the cohomology $H^\bullet(G,M)$
can be calculated by the complex of "normalized" cochains
$C^\bullet(G,M)$, where $C^n(G,M)$ is the group of maps $f:G^n\to
M$ such that $f(g_1,\ldots,g_n)=0$ whenever one of the $g_i$ is
the identity. The coboundary map $\delta_G:C^n(G,M)\to
C^{n+1}(G,M)$ is given by
\[\begin{array}{lll}(\delta_Gf)(g_1,\ldots,g_{n+1})
&=&g_1f(g_2,\ldots,g_{n+1})\\
&&+\sum(-1)^if(g_1,\ldots,g_ig_{i+1},\ldots,g_{n+1})\\
&&+(-1)^{n+1}f(g_1,\ldots,g_n)\end{array}.\]

Now let $H$ be a normal subgroup of $G$ with $K=G/H$ the quotient.
Then we have a filtration on $F^\bullet C^\bullet(G,M)$ defined by
the formula
\[F^pC^n(G,M)=\left\{\begin{array}{ll}\mathrm{Maps}(G^{n-p}\times
K^p,M)\cap C^n(G,M) &p\leq n\\ 0& p>n\end{array}\right.\]

The filtration gives rise to a spectral sequence which converges
to $H^\bullet(G,M)$. Let us recall the $E_1$ and $E_2$ terms.
There is a natural map $$\res:F^pC^n(G,M)\to C^p(K,C^{n-p}(H,M))$$
obtained by restriction. A key observation of Hochschild and Serre
is that this map induces an isomorphism $(E_1^{*,q},d_1)\cong
(C^\bullet(K,H^q(H,M)),\delta_K)$. This gives the desired spectral
sequence with $E_2^{p,q}\cong H^p(K,H^q(H,M))$.

In particular, we have the following commutative diagram:
\[\begin{CD}E_1^{0,2}=\dfrac{\{a\in
C^2,\delta_Ga\in
F^1C^3\}}{\delta_G(C^1)+F^1C^2}@>\cong>\res> H^2(H,M)\\
\bigcup @.\bigcup\\
E_2^{0,2}=\dfrac{\{a\in C^2,\delta_Ga\in
F^2C^3\}}{\delta_G(C^1)+\{a\in F^1C^2,\delta_Ga\in
F^2C^3\}}@>\cong>\res> H^2(H,M)^K\\ \bigcup @.
\bigcup\\
E_3^{0,2}=\dfrac{\{a\in C^2,\delta_Ga\in
F^3C^3\}}{\delta_G(C^1)+\{a\in F^1C^2,\delta_Ga\in
F^3C^3\}}@>\cong>\res> \small{\ker(H^2(H,M)^K\to H^2(K,H^1(H,M)))}
\end{CD}\]

\subsubsection{The definition and a criterion for
gerbal pairs of groups}

Let
\begin{equation}\label{central extension of
H} 1\to\bbC^\times\to\wh{H}\stackrel{\pi}{\to} H\to 1
\end{equation}
be a central extension of group $H$ by $\bbC^\times$, which
corresponds the cohomology class $a\in H^2(H,\bbC^\times)$. Then
the conjugation of $\wh{H}$ on itself descends to an action of $H$
on $\wh{H}$, which is denoted by $c$. That is, for any $h\in H,
\wh{h}\in\wh{H}$, choose any lifting $\tilde{h}\in\wh{H}$ of $h$,
then
\[c_h(\wh{h})=\tilde{h}\wh{h}\tilde{h}^{-1}\]

Now assume that we have an extension of groups
\begin{equation}\label{extension of K by H}
1 \to H\stackrel{i}{\to} G\to K \to 1
\end{equation}

Since $H$ is a normal subgroup of $G$, $G$ acts on $H$ by
conjugation, which is denoted by $\rho$, $\rho_g(h)=ghg^{-1}$.

\begin{dfn}
\label{gerbal pair} Suppose that we can lift $\rho$ to an action
$\tilde{\rho}$ of $G$ on $\wh{H}$, i.e.
$\pi(\tilde{\rho}_g(\wh{h}))=\rho_g(\pi(\wh{h}))$, so that
$\tilde{\rho}_h=c_h$ for $h\in H$ and
$\tilde{\rho}_g|_{\bbC^\times}=\mathrm{id}$. Then we will call the
pair $(G,\wh{H})$ of groups equipped with the short exact
sequences (\ref{central extension of H}) and (\ref{extension of K
by H}), and the action $\tilde{\rho}$, a {\em gerbal pair}.
\end{dfn}

Note that if $(G,\wh{H})$ is a gerbal pair of groups, then
$\wh{H}\to G$ form what is called a ``crossed module of groups''
(see \cite{Br}).

The following proposition gives the necessary and sufficient
conditions for $(G,\wh{H})$ to be a gerbal pair.

\begin{prop}\label{liftings}
The lifting of Definition \ref{gerbal pair} exists if and only if
the class $a\in H^2(H,\bbC^\times)$ is transgressive, that is,
$a\in H^2(H,\bbC^\times)^K$ and $d_2(a)=0$, where
$d_2:H^2(H,\bbC^\times)\to H^2(K,H^1(H,\bbC^\times))$ is the
differential in the $E_2$ term of the Lyndon-Hochschild-Serre
spectral sequence. If such a lifting exists, then all liftings form
a torsor under $Z^1(K,H^1(H,\bbC^\times))$.
\end{prop}

\begin{proof}
For any $g\in G$, it acts on $H$ by conjugation and therefore,
acts on the cohomology $H^\bullet(H,\bbC^\times)$. Denote this
({\em right}) action also by $\rho_g$. More explicitly, the class
$\rho_g(a)$ is obtained by pullback $\bbC^\times\to\wh{H}\to H$ by
$\rho_g:H\to H$. Denote the new central extension by
$\bbC^\times\to\wh{H}_g\to H$. Then $\rho_g$ could lift to
$\widetilde{\rho_g}$, which is the identity map on the central
$\bbC^\times\to\wh{H}$ if and only if there is an isomorphism
$\wh{H}\to\wh{H}_g$, which induces the identity maps on the
central $\bbC^\times$ and the quotient $H$, i.e. $\rho_g(a)=a$.
Furthermore, all such liftings are bijective to the isomorphisms
from $\wh{H}\to\wh{H}_g$, and therefore is a torsor under
$H^1(H,\bbC^\times)$. Therefore, a lifting $\tilde{\rho}$ exists
only if $a$ is a class invariant under $G$. It is a well-known
fact that $H$ acts on its cohomology trivially. Therefore, such a
lifting exists only if $a\in H^2(H,\bbC^\times)^K$.

We assume that $a\in H^2(H,\bbC^\times)^K$. We will identify a
lifting $\widetilde{\rho_g}:\wh{H}\to\wh{H}$ with an isomorphism
$\widetilde{\rho_g}:\wh{H}\to\wh{H}_g$. For every $g\in G$, we
choose a lifting of $\widetilde{\rho_g}$ of $\rho_g$ such that
\[\widetilde{\rho_1}=1, \ \
\ \ \ \ \ \ \ \widetilde{\rho_{gh}}=\widetilde{\rho_g}c_h \ \mbox{
for } g\in G, h\in H \ \ \ \ \ \ \\ \ \ \ \ \ \ \ (\dag)\] For any
two $g,g'\in G$, both $\widetilde{\rho_{gg'}}$ and
$\widetilde{\rho_g}\widetilde{\rho_{g'}}$ are liftings of
$\rho_{gg'}$. Therefore, there is a unique $\bar{b}(g,g')\in
H^1(H,\bbC^\times)$ sending the second lifting to the first
lifting. That is, the unique automorphism
$\bar{b}(g,g'):\wh{H}\to\wh{H}$ making the following diagram
commute.
\[\begin{CD}\wh{H}@>\bar{b}(g,g')>>\wh{H}\\
@V\widetilde{\rho_{gg'}}VV@V\widetilde{\rho_{g'}}VV\\
\wh{H}_{gg'}@<\widetilde{\rho_g}<<\wh{H}_{g'}\end{CD}\]

This way one defines $\bar{b}\in C^2(G,H^1(H,\bbC^\times))$. We
will check that

(i) $\bar{b}(hg,g')=\bar{b}(g,g'h)=\bar{b}(g,g')$. Therefore,
$\bar{b}$ is in fact in $C^2(K,H^1(H,\bbC^\times))$.

(ii) Define $b(g,g')=\bar{b}(g'^{-1},g^{-1})$. Then $b$ is a
cocycle. Choosing different liftings $\widetilde{\rho_g}$
satisfying ($\dag$), $b$ differs by a coboundary. Therefore, $[b]$
is a well defined cohomology class in $H^2(K,H^1(H,\bbC^\times))$.

(iii) If $[b]=0$ in $H^2(K,H^1(H,\bbC^\times))$, we can choose a
compatible family of liftings $\widetilde{\rho_g}$ such that
$\tilde{\rho}_g=\widetilde{\rho_g}$ is a lifting of $\rho$.

(iv) All the liftings form a torsor under
$Z^1(K,H^1(H,\bbC^\times))$.

(v) $[b]=d_2(a)$.

We begin with the proof of (i). $\bar{b}(hg,g')=\bar{b}(g,g')$
comes from the following commutative diagram:
\[\begin{CD}\wh{H}@>\bar{b}(g,g')>>\wh{H}@>\widetilde{\rho_{g'}}>>
\wh{H}_{g'}@>\widetilde{\rho_g}>>\wh{H}_{gg'}\\
@|@|@|@Vc_hVV\\
\wh{H}@>\bar{b}(hg,g')>>\wh{H}@>\widetilde{\rho_{g'}}>>\wh{H}_{g'}@>
\widetilde{\rho_{hg}}>>\wh{H}_{hgg'}\end{CD}\] Similarly, we
obtain the other identity.

\medskip

Next, we prove (ii). Observe that we have the following
commutative diagram
\[\begin{CD}\wh{H}@>\bar{b}(gg',g'')>>\wh{H}
@>\widetilde{\rho_{g''}}>>\wh{H}_{g''}
@>\widetilde{\rho_{gg'}}>>\wh{H}_{gg'g''}\\
@V\bar{b}(g,g'g'')VV@V\rho_{g''}\bar{b}(g,g')
VV@V\bar{b}(g,g')VV@A\widetilde{\rho_g}AA\\
\wh{H}@>\bar{b}(g',g'')>>\wh{H}@>\widetilde{\rho_{g''}}>>
\wh{H}_{g''}@>\widetilde{\rho_{g'}}>>\wh{H}_{g'g''}\end{CD}\]
Therefore,
$\bar{b}(g',g'')\bar{b}(gg',g'')^{-1}\bar{b}(g,g'g'')(\rho_{g''}
\bar{b}(g,g'))^{-1}=1$. It is readily to check that
$b(g,g')=\bar{b}(g'^{-1},g^{-1})$ is a cocycle. If we choose
different liftings $\widetilde{\rho_g}'$ satisfying ($\dag$), then
$\widetilde{\rho_g}'=\widetilde{\rho_g}\bar{u}(g)$ for a unique
$\bar{u}:C^1(G,H^1(H,\bbC^\times))$. Furthermore, the condition
($\dag$) gives $\bar{u}(gh)=c^{-1}_h\bar{u}(g)c_h=\bar{u}(g)$,
i.e., $\bar{u}\in C^1(K,H^1(H,\bbC^\times))$. Easy calculation
shows that $\bar{b}'(g,g')=\bar{b}(g,g')\bar{u}(g')^{-1}(\rho_{g'}
\bar{u}(g))^{-1} \bar{u}(gg')$. Then $b=b'\delta u$ where
$u(g)=\bar{u}(g^{-1})$.

\medskip

Next, we prove (iii). Assuming that $[b]=0$ in
$H^2(K,H^1(H,\bbC^\times))$. Then there exists an $u\in
C^1(G,H^1(H,\bbC^\times))$, such that $u(gh)=u(g)$ for $h\in H$
and $b(g,g')=\rho_{g^{-1}}(u(g'))u(g)u(gg')^{-1}$. Define a new
family of liftings by
$\widetilde{\rho_g}'=\widetilde{\rho_g}u(g^{-1})$. By the
calculation made in (ii),
$\widetilde{\rho_{gh}}'=\widetilde{\rho_g}'c_h$ and
$\widetilde{\rho_{gg'}}'=\widetilde{\rho_g}'\widetilde{\rho_{g'}}'$.

\medskip

Next, we prove (iv). Assuming that there exists at least one
lifting. Let $\tilde{\rho},\tilde{\rho}'$ be two liftings. Then we
can write $\tilde{\rho}'_g=\tilde{\rho}_g\bar{u}(g)$ with
$\bar{u}(g)\in H^1(H,\bbC^\times), \bar{u}(gh)=\bar{g}$. It is
clear that $\bar{u}(gg')=\rho_{g'}(\bar{u}(g))\bar{u}(g')$.
Therefore, if we define $u(g)=\bar{u}(g^{-1})$, then $u\in
Z^1(G,H^1(H,\bbC^\times))$. Using the fact that $u(hg)=u(g)$,
$u\in Z^1(K,H^1(H,\bbC^\times))$.

\medskip

It remains to prove (v). We choose for any $g\in G$, a lifting
$\widetilde{\rho_g}:\wh{H}\to\wh{H}$ satisfying $(\dag)$. We
choose a section $s:K\to G$. Then any element $g\in G$ could be
uniquely written as $\fraks(k)h$ for $h\in H$ and $k\in K$. For
$k,k'\in K$, denote
\[t(k,k')=s(kk')^{-1}s(k)s(k')\in H\]
We also choose a section $H\to\wh{H}$. That is, for every $h\in
H$, we choose a particular lifting in $\wh{H}$, which is denoted
by $\tilde{h}$ or $(h,1)$. Now we define $\tilde{a}\in
C^2(G,\bbC^\times)$ by
\begin{equation}\label{tilde(a)}
\tilde{a}(s(k)h,s(k')h')=\frac{\widetilde{t(k,k')}
\widetilde{\rho_{s(k')^{-1}}}(\tilde{h})\tilde{h'}}
{(t(k,k')\rho_{s(k')^{-1}}(h)h',1)}
\end{equation}
It is clear that the restriction of $\tilde{a}$ to $H$ is a
cocycle representing the cohomology class $a$.

We claim $\delta_G(\tilde{a})\in F^2C^3(G,\bbC^\times)$. First,
according to the proof of \cite{HS}, Theorem I (or direct
calculation), $\delta_G(\tilde{a})\in F^1C^3(G,\bbC^\times)$.
Therefore, it is enough to show that
\[\frac{\tilde{a}(s(k_2)h_2,s(k_3))\tilde{a}
(s(k_1)h_1,s(k_2)h_2s(k_3))}{\tilde{a}(s(k_1)h_1s(k_2)h_2,
s(k_3))\tilde{a}(s(k_1)h_1,s(k_2)h_2)}=
\frac{\tilde{a}(s(k_2),s(k_3))\tilde{a}(s(k_1)h_1,s(k_2)s(k_3))}
{\tilde{a}(s(k_1)h_1s(k_2),s(k_3))\tilde{a}(s(k_1)h_1,s(k_2))}\]
Direct calculation shows the above identity is equivalent to
\[\widetilde{\rho_{s(k_3)^{-1}}}\left(\frac{(h_2,1)(t(k_1,k_2)
\rho_{s(k_2)^{-1}}(h_1),1)}{(t(k_1,k_2)\rho_{s(k_2)^{-1}}
(h_1)h_2,1)}\right)
=\frac{(h_2,1)(t(k_1,k_2)\rho_{s(k_2)^{-1}}(h_1),1)}
{(t(k_1,k_2)\rho_{s(k_2)^{-1}}(h_1)h_2,1)}\] But this follows from
$\frac{(h_2,1)(t(k_1,k_2)\rho_{s(k_2)^{-1}}(h_1),1)}{(t(k_1,k_2)
\rho_{s(k_2)^{-1}}(h_1)h_2,1)}$ is in the central
$\bbC^\times\subset\wh{H}$ and $\widetilde{\rho_{s(k_3)^{-1}}}$
leaves the central $\bbC^\times$ invariant.

Now, by the construction of the spectral sequence, a cocycle
representing $d_2(a)\in H^2(K,H^1(H,\bbC^\times))$ can be chosen
as $\delta_G(\tilde{a})(h,s(k),s(k'))$ where $h\in H, k,k'\in K$.
Let us calculate this cocycle. We have
\[\begin{split}\delta_G(\tilde{a})(h,s(k),s(k'))
&=\frac{\tilde{a}(s(k),s(k'))\tilde{a}(h,s(k)s(k'))}{\tilde{a}
(hs(k),s(k'))\tilde{a}(h,s(k))}\\
&=\frac{\widetilde{\rho_{s(kk')^{-1}}}
(\tilde{h})\widetilde{t(k,k')}}{\widetilde{t(k,k')}
\widetilde{\rho_{s(k')^{-1}}}((\rho_{s(k)^{-1}}(h),1))}
\frac{(\rho_{s(k)^{-1}}(h),1)}{\widetilde{\rho_{s(k)^{-1}}}
(\tilde{h})}\\
&=\frac{\widetilde{\rho_{s(kk')^{-1}}}(\tilde{h})\widetilde{t(k,k')}}
{\widetilde{t(k,k')}\widetilde{\rho_{s(k')^{-1}}}
((\rho_{s(k)^{-1}}(h),1))}\widetilde{\rho_{s(k')^{-1}}}
\left(\frac{(\rho_{s(k)^{-1}}(h),1)}{\widetilde{\rho_{s(k)^{-1}}}
(\tilde{h})}\right)\\
&=\frac{\widetilde{\rho_{s(kk')^{-1}}}(\tilde{h})
\widetilde{t(k,k')}}{\widetilde{t(k,k')}\widetilde{\rho_{s(k')^{-1}}}
(\widetilde{\rho_{s(k)^{-1}}}(\tilde{h}))}\\
&=\widetilde{\rho_{s(k')^{-1}}}(\widetilde{\rho_{s(k)^{-1}}}
(b(k,k')(h)))=b(k,k')(h)
\end{split}\]
Therefore, $[b]=d_2(a)$. This completes the proof.
\end{proof}

\subsubsection{Gerbal pairs
and gerbal representations}\label{second lemma}

We still assume that we have the short exact sequences of groups
(\ref{central extension of H}) and (\ref{extension of K by H}).
For a character $\lambda:\bbC^\times\to\bbC^\times$, denote by
$\mathrm{Rep}_\lambda(\wh{H})$ the category of complex
representations of $\wh{H}$, on which the central $\bbC^\times$
acts by $\lambda$.

We assume that the lifting $\tilde{\rho}$ of Definition
\ref{gerbal pair} exists. Then, by Proposition \ref{liftings},
$a\in H^2(H,\bbC^\times)^K$ and $d_2(a)=0$. Let us choose such a
lifting, so that $(G,\wh{H})$ is a gerbal pair of groups in the
sense of Definition \ref{gerbal pair}. Then we have an action of
$G$ on $\wh{H}$ by automorphisms, which leaves the central
$\bbC^\times\subset\wh{H}$ invariant. Therefore we obtain a
representation of $G$ on $\rep_\lambda(\wh{H})$ (see Example
\ref{example of genuine action}). We claim that we obtain a gerbal
representation of $K$ on $\rep_\lambda(\wh{H})$, in the sense of
Definition \ref{gerbe action1}.

Indeed, for any $g\in G$, we have an auto-equivalence
$F_g:\rep_\lambda(\wh{H})\to\rep_\lambda(\wh{H})$. For $h\in H$, since
the action $\tilde{\rho}_h$ on $\wh{H}$ is just the conjugation $c_h$,
we have $F_h\cong\mathrm{id}$. Now we choose a set-theoretic section
$s:K\to G$, and for any $k\in K$, define
$F_k=F_{s(k)}:\rep_\lambda(\wh{H})\to\rep_\lambda(\wh{H})$. Since for
$k,k'\in K$, there exists some $h\in H$ such that $s(kk')=s(k)s(k')h$,
we obtain that $$F_{kk'}=F_{s(kk')} = F_{s(k)s(k')h} \cong
F_{s(k)}F_{s(k')}F_h\cong F_kF_{k'}.$$

By Theorem \ref{H^3}, there is a cohomology class
$e_{\tilde{\rho}}\in
H^3(K,\calZ(\mathrm{Rep}_\lambda(\wh{H}))^\times)$. Observe that
since $\rep_\lambda(\wh{H})$ is $\bbC$-linear,
$\bbC^\times\subset\calZ(\mathrm{Rep}_\lambda(\wh{H}))^\times$
naturally. We therefore obtain a map
$\bbC^\times\stackrel{\lambda}{\to}\bbC^\times\subset
\calZ(\rep_\lambda(\wh{H}))^\times$.

\begin{prop}\label{gerbal
action} (i) $e_{\tilde{\rho}}\in
\mathrm{Im}(H^3(K,\bbC^\times)\stackrel{\lambda_*}{\to}
H^3(K,\calZ(\mathrm{Rep}_\lambda(\wh{H}))^\times)$.

(ii) The projection of $e_{\tilde{\rho}}$ along
\[H^3(K,\bbC^\times)\cong
E_2^{3,0}\to H^3(K,\bbC^\times)/d_2(H^1(K,H^1(H,\bbC^\times)))\cong
E_3^{3,0}\] is independent of the choice of the lifting $\tilde{\rho}$
and is equal to $d_3(a)$, where $d_3$ is the differential $E_3^{0,2}
\to E_3^{3,0}$ in the $E_3$ term of the Lyndon-Hochschild-Serre
spectral sequence (this is well-defined since $d_2(a)=0$).
\end{prop}

\begin{proof}
Without loss of generality, we may assume that $\lambda$
   is the identity.

Fix a lifting $\tilde{\rho}$. We choose a set-theoretic section
$s:K\to G$. For $k,k'\in K$, denote
\[t(k,k'):=s(kk')^{-1}s(k)s(k')\in H.\]

Recall that the functor
$F_k:\mathrm{Rep}_\lambda(\wh{H})\to\mathrm{Rep}_\lambda(\wh{H})$
is defined as follows. Let $(M,m)$ be a representation of
$\wh{H}$. Then $F_k(M,m)$ has $M$ as the underlying vector space
with multiplication given by
\[F_km(\wh{h},x)=m(\tilde{\rho}_{s(k)^{-1}}\wh{h},x).\]
It is easy to show that the natural transform
$c(k,k'):F_kF_{k'}M\cong F_{kk'}M$ will satisfy the following
formula
\[c(k,k')m((c_{t(k,k')^{-1}}(\tilde{\rho}_{s(kk')^{-1}}(\wh{h}))),x)
=m(\tilde{\rho}_{s(kk')^{-1}}(\wh{h}),c(k,k')x)\] for any $x\in M,
\wh{h}\in\wh{H}$. Therefore, we can choose
\[c(k,k')=m(\widetilde{t(k,k')},\cdot),\]
where $\widetilde{t(k,k')}$ is any lifting of $t(k,k')$ to
$\wh{H}$. The natural transformation
\[F_{k_1}F_{k_2}F_{k_3}M\to
F_{k_1k_2}F_{k_3}M\to F_{k_1k_2k_3}M\] is given by
\[x\mapsto
m(\widetilde{t(k_1k_2,k_3)}(\tilde{\rho}_{s(k_3)^{-1}}
(\widetilde{t(k_1,k_2)})),x).\] On the other hand, the natural
transformation
\[F_{k_1}F_{k_2}F_{k_3}M\to
F_{k_1}F_{k_2k_3}M\to F_{k_1k_2k_3}M\] is given by
\[x\mapsto
m(\widetilde{t(k_1,k_2k_3)}\widetilde{t(k_2,k_3)},x).\] Observe
that
\[i\circ\pi(\widetilde{t(k_1,k_2k_3)}\widetilde{t(k_2,k_3)}
(\widetilde{t(k_1k_2,k_3)}\tilde{\rho}_{s(k_3)^{-1}}
(\widetilde{t(k_1,k_2)}))^{-1})=1.\] Therefore
\[\widetilde{t(k_1,k_2k_3)}\widetilde{t(k_2,k_3)}
(\widetilde{t(k_1k_2,k_3)}\tilde{\rho}_{s(k_3)^{-1}}
(\widetilde{t(k_1,k_2)}))^{-1}\] is in the central $\bbC^\times$
of $\wh{H}$.

We find that
\[e_{\tilde{\rho}}(k_1,k_2,k_3)=\widetilde{t(k_1,k_2k_3)}
\widetilde{t(k_2,k_3)}(\widetilde{t(k_1k_2,k_3)}
\tilde{\rho}_{s(k_3)^{-1}}(\widetilde{t(k_1,k_2)}))^{-1}\] is a
3-cocycle. Different choices of liftings $\widetilde{t(k,k')}$
make $e_{\tilde{\rho}}$ differ by a coboundary. Therefore,
$e_{\tilde{\rho}}$ is a well-defined cohomology class in
$H^3(K,\bbC^\times)$. Furthermore, if $\tilde{\rho}'$ is another
lifting of $\rho$, then we know from the proof of Lemma
\ref{liftings} that $\tilde{\rho}'_g=\tilde{\rho}_gu(g^{-1})$,
where $u\in Z^1(K,H^1(H,\bbC^\times))$. Therefore, we have
\[\frac{e_{\tilde{\rho}'}(k_1,k_2,k_3)}{e_{\tilde{\rho}}
(k_1,k_2,k_3)}=\frac{1}{u(s(k_3))(t(k_2,k_3))}.\] Applying the
same method as in the proof of Lemma \ref{liftings}, one can show
the right hand side is an expression for $d_2(u)$. Therefore, the
image of $e_{\tilde{\rho}}$ in
$$H^3(K,\bbC^\times)/d_2(H^1(K,H^1(H,\bbC^\times)))$$ is
independent of the lifting $\tilde{\rho}$.

It remains to prove that this image is exactly the class $d_3(a)$.
As in the proof of Lemma \ref{liftings}, we choose a section
$H\to\wh{H}$, and the corresponding lifting of $h$ is denoted by
$\tilde{h}$ or $(h,1)$. Then we define
\[\tilde{a}(s(k)h,s(k')h')=\frac{\widetilde{t(k,k')}
\tilde{\rho}_{s(k')^{-1}}(\tilde{h})\tilde{h'}}{(t(k,k')
\rho_{s(k')^{-1}}(h)h',1)}\] We have already shown that
$\delta_G(a)\in F^2C^3(G,\bbC^\times)$. Using the fact that
$\tilde{\rho}_{gg'}=\tilde{\rho_g}\tilde{\rho}_{g'}$, one can even
show that $\delta_G(a)\in F^3C^3(G,\bbC^\times)$. Furthermore, a
direct calculation shows that
\[\delta_G(\tilde{a})(s(k_1),s(k_2),s(k_3))=e_{\tilde{\rho}}
(k_1,k_2,k_3).\] This completes the proof.
\end{proof}

Observe that to prove the above proposition, we do not really use
the fact that each object in $\rep_\lambda(\wh{H})$ is realized as
a representation of $\wh{H}$. All we need is that
$\rep_\lambda(\wh{H})$ is a $\bbC$-linear abelian category, with a
genuine action of $G$ satisfying certain properties. This allows
us to generalize the proposition in the following way.

We still assume that $(G,\wh{H})$ is a gerbal pair.
Therefore,there is an action $\tilde{\rho}$ of $G$ on $\wh{H}$,
which lifts the action $\rho$ of $G$ on $H$. Recall that
$\pi:\wh{H}\to H$ so that for any $h\in H$, $\pi^{-1}(h)$ is a
$\bbC^\times$-torsor. Let $\calC$ be a $\bbC$-linear abelian
category. Recall the definition of the 2-group $\GL(\calC)$ (see \S
\ref{bbGbbL}). Assume that there is a genuine representation of
$G$ on $\calC$, i.e. a 2-group homomorphism $F:G\to\GL(\calC)$.
For simplicity, we will assume that $F_gF_{g'}=F_{gg'}$. (This is
not essential, but simplifies the discussion.) Since $\calC$ is
$\bbC$-linear, there is a natural $\bbC^\times$-action on
$\Hom_{\GL(\calC)}(F_g,\mathbf{1}_{\calC})$, for any $g\in G$.
Assume that for any $h\in H$, there is a $\bbC^\times$-equivariant
embedding
\[\alpha_h:\pi^{-1}(h)\to\Hom_{\GL(\calC)}(F_h,\mathbf{1}_\calC)\]
such that:

\noindent (i) for any $g\in G$, the following diagram is
commutative:
\[\begin{CD}
\pi^{-1}(h)@>\alpha_h>>\Hom_{\GL(\calC)}(F_h,\mathbf{1}_\calC)\\
@V\tilde{\rho}_gVV@VVF_g\otimes-\otimes
F_{g^{-1}}V\\
\pi^{-1}(\rho_g(h))@>\alpha_{\rho_g(h)}>>\Hom_{\GL(\calC)}(F_{\rho_g(h)},
\mathbf{1}_\calC)
\end{CD}\]

\noindent (ii) for any $h,h'\in H$, the following diagram is
commutative:
\[\begin{CD}\pi^{-1}(h)@.\otimes@.\pi^{-1}(h')@>>>\pi^{-1}(h'h)\\
@V\alpha_hVV@.@VV(-\otimes
F_h)\circ\alpha_{h'}V@VV\alpha_{h'h}V\\
\Hom(F_h,\mathbf{1}_\calC)@.\otimes@.\Hom(F_{h'}F_h,F_h)@>>>\Hom(F_{h'h},
\mathbf{1}_\calC)\end{CD}\]

\begin{prop}\label{variation}
Assumptions are as above. Then there is a gerbal representation of
$K$ on $\calC$. The corresponding cohomology class $e\in
\mathrm{Im}(H^3(K,\bbC^\times)\to H^3(K,\calZ(\calC)^\times)$. The
projection of $e$ along
\[H^3(K,\bbC^\times)\cong
E_2^{3,0}\to
H^3(K,\bbC^\times)/d_2(H^1(K,H^1(H,\bbC^\times)))\cong E_3^{3,0}\]
is equal to $d_3(a)$.
\end{prop}

The proof is identical to the proof of Proposition \ref{gerbal
action}.

\subsection{Gerbal pairs associated
with $GL_{\infty,\infty}$}

Having developed the general formalism of gerbal pairs, we now
apply it to a special case. Namely, we show that there is an
action of $\widetilde{GL}_\infty(\gl_\infty)$ on
$\wh{GL}_{\frakf,\infty}$, so that the groups
$(\widetilde{GL}_\infty(\gl_\infty),\wh{GL}_{\frakf,\infty})$,
equipped with the short exact sequences (\ref{extension of
GLfrakfinfty}) and (\ref{extension of GLinftyinfty}), form a
gerbal pair in the sense of Definition \ref{gerbal pair}.
Similarly,
$(\widetilde{GL}_\infty(\gl_\infty),\wh{GL}^0_{\frakf,\infty})$
equipped with (\ref{extension of GLfrakfinfty0}) and
(\ref{extension of whGLfrakfinfty}) also form a gerbal pair.
Applying Proposition \ref{liftings}, we will obtain a particular
class $[E_3]\in H^3(GL_{\infty,\infty},\bbC^\times)$, and
similarly $[\wh{E}_3]\in
H^3(\wh{GL}_{\infty,\infty},\bbC^\times)$.

Recall that (\ref{extension of GLfrakfinfty}) gives us a map
$\widehat{GL}_{\frakf,\infty}\to GL_{\frakf,\infty}$. Combined
with (\ref{extension of GLfrakfinfty}), we obtain a group
homomorphism
$\delta:\widehat{GL}_{\frakf,\infty}\to\widetilde{GL}_\infty(\gl_\infty)$
as the composition of $\widehat{GL}_{\frakf,\infty}\to
GL_{\frakf,\infty}\to\widetilde{GL}_\infty(\gl_\infty)$.

\begin{thm}\label{crossed
module corresponding to E3} There is a unique action of
$\widetilde{GL}_\infty(\gl_\infty)$ on $\wh{GL}_{\frakf,\infty}$,
making the pair
$(\widetilde{GL}_\infty(\gl_\infty),\wh{GL}_{\frakf,\infty})$,
equipped with the short exact sequences (\ref{extension of
GLfrakfinfty}) and (\ref{extension of GLinftyinfty}), into a
gerbal pair in the sense of Definition \ref{gerbal pair}.
\end{thm}

\begin{rmk} This action will
also make
$(\widetilde{GL}_\infty(\gl_\infty),\wh{GL}^0_{\frakf,\infty})$ a
gerbal pair.
\end{rmk}

\begin{proof} Once such an action exists, the uniqueness is
clear. This is because according to Proposition \ref{liftings},
all such actions form a torsor under
$Z^1(GL_{\infty,\infty},H^1(GL_{\frakf,\infty},\bbC^\times))
=H^1(GL_{\infty,\infty},\bbC^\times)=0$. Therefore, we will be
focusing on the existence of such an action.

Observe that $GL_{\frakf,\infty}$ is a normal subgroup of
$GL_\infty^+(\gl_\infty)$, and the action of
$\widetilde{GL}_\infty(\gl_\infty)$ on $GL_{\frakf,\infty}$
factors through the adjoint action of $GL_\infty^+(\gl_\infty)$ on
$GL_{\frakf,\infty}$ via the natural projection
$\widetilde{GL}_\infty(\gl_\infty)\to GL_\infty^+(\gl_\infty)$.
Therefore, it is enough to prove that there is an action of
$GL_\infty^+(\gl_\infty)$ on $\wh{GL}_{\frakf,\infty}$ such that
$(GL_\infty^+(\gl_\infty),\wh{GL}_{\frakf,\infty})$ is a gerbal
pair.

Denote the conjugation of $GL_\infty^+(\gl_\infty)$ on
$GL_{\frakf,\infty}$ by $\rho$, i.e., $\rho_g(h)=ghg^{-1}$ for
$g\in GL^+_{\infty}(\gl_\infty), h\in GL_{\frakf,\infty}$. Denote
the conjugation of $GL_{\frakf,\infty}$ on
$\wh{GL}_{\frakf,\infty}$ by $c$, i.e.,
$c_h(a)=\tilde{h}a\tilde{h}^{-1}$ for $h\in GL_{\frakf,\infty},
a\in\wh{GL}_{\frakf,\infty}$, and
$\tilde{h}\in\wh{GL}_{\frakf,\infty}$ is any lifting of $h$. We
want to show that $\rho$ lifts to an action $\tilde{\rho}$ of
$GL_\infty^+(\gl_\infty)$ on $\widehat{GL}_{\frakf,\infty}$, such
that: (i) $\tilde{\rho}|_{GL_{\frakf,\infty}}=c$; (ii) for any
$g\in GL_\infty^+(\gl_\infty)$, $\tilde{\rho}_g$ is the identity
map on the central $\bbC^\times\subset GL_{\frakf,\infty}$.

First, by Corollary \ref{D2 invariant under g}, for any $g\in
GL_\infty^+(\gl_\infty)$, $\rho_g:GL_{\frakf,\infty}\to
GL_{\frakf,\infty}$ lifts to a group automorphism
$\widetilde{\rho_g}:\wh{GL}_{\frakf,\infty}\to\wh{GL}_{\frakf,\infty}$,
which leaves the central $\bbC^\times$ invariant. We denote the
restriction of $\widetilde{\rho_g}$ to $\wh{GL}^0_{\frakf,\infty}$
by $\widetilde{\rho_g}^0$. Since
$H^1(GL^0_{\frakf,\infty},\bbC^\times)=0$, such
$\widetilde{\rho_g}^0$ is unique. Therefore, we obtain an action
$\tilde{\rho}^0$ of $GL_\infty^+(\gl_\infty)$ on
$GL^0_{\frakf,\infty}$ that lifts $\rho$. It is clear that this
action makes $\wh{GL}^0_{\frakf,\infty}\to
GL_\infty^+(\gl_\infty)$ a gerbal pair.

Next, we show we can extend $\tilde{\rho}^0$ to an action
$\tilde{\rho}$ of $GL_\infty^+(\gl_\infty)$ on
$\wh{GL}_{\frakf,\infty}$ such that $\wh{GL}_{\frakf,\infty}\to
GL_\infty^+(\gl_\infty)$ is a gerbal pair. Recall that by choosing
an element $\sigma\in GL_{\frakf,\infty}$ such that
$\deg(\sigma)=1$, we obtain a splitting
$GL_{\frakf,\infty}=GL^0_{\frakf,\infty}\rtimes\bbZ$. For example,
we will choose $\sigma$ such that $\sigma(t^i)=t^{i+1}$ and
$\sigma(s^it^j)=s^it^j$ for $i\geq 1$. Choosing a lifting
$\wh{\sigma}$ of $\sigma$ in $\wh{GL}_{\frakf,\infty}$, we obtain
a splitting
$\wh{GL}_{\frakf,\infty}=\wh{GL}_{\frakf,\infty}^0\rtimes\bbZ$.
Therefore, we denote elements in $\wh{GL}_{\frakf,\infty}$ by
$(a,\wh{\sigma}^m)$ for $a\in\wh{GL}^0_{\frakf,\infty}$ and the
multiplication by
$(a,\wh{\sigma}^m)(a',\wh{\sigma}^{m'})=(a\wh{\sigma}^ma'
\wh{\sigma}^{-m},\wh{\sigma}^{m+m'})$.

For any $g\in GL_\infty^+(\gl_\infty)$, we choose a lifting
$\widetilde{\rho_g(\sigma)}\in\wh{GL}_{\frakf,\infty}$ of
$\rho_g(\sigma)\in GL_{\frakf,\infty}$. We check that
$\widetilde{\rho_g}:\wh{GL}_{\frakf,\infty}\to\wh{GL}_{\frakf,\infty}$
defined by
\[\widetilde{\rho_g}(a,\wh{\sigma}^m)=(\tilde{\rho}^0_g(a)
\widetilde{\rho_g(\sigma)}^m\wh{\sigma}^{-m},\wh{\sigma}^m)\] is a
group homomorphism and therefore is a lifting of $\rho_g$. Thus,
we must check that
\[\widetilde{\rho_g}((a,\wh{\sigma}^{m})(a',\wh{\sigma}^{m'}))=
\widetilde{\rho_g}((a,\wh{\sigma}^{m}))
\widetilde{\rho_g}((a',\wh{\sigma}^{m'})).\] By definition,
\[\begin{array}{ll}\mathrm{l.h.s.}&=\widetilde{\rho_g}((a\wh{\sigma}^ma'\wh{\sigma}^{-m},\wh{\sigma}^{m+m'}))\\

&=(\tilde{\rho}_g^0(a\wh{\sigma}^ma'\wh{\sigma}^{-m})\widetilde{\rho_g(\sigma)}^{m+m'}\wh{\sigma}^{-m-m'},\wh{\sigma}^{m+m'})\\

&=(\tilde{\rho}_g^0(a)\tilde{\rho}_g^0(\wh{\sigma}^ma'\wh{\sigma}^{-m})\widetilde{\rho_g(\sigma)}^{m+m'}\wh{\sigma}^{-m-m'},\wh{\sigma}^{m+m'}).

\end{array}\]
\[\begin{array}{ll}\mathrm{r.h.s.}&=(\tilde{\rho}^0_g(a)\widetilde{\rho_g(\sigma)}^m\wh{\sigma}^{-m},\wh{\sigma}^m)(\tilde{\rho}^0_g(a')\widetilde{\rho_g(\sigma)}^{m'}\wh{\sigma}^{-m'},\wh{\sigma}^{m'})\\

&=(\tilde{\rho}^0_g(a)\widetilde{\rho_g(\sigma)}^m\wh{\sigma}^{-m}\wh{\sigma}^m\tilde{\rho}^0_g(a')\widetilde{\rho_g(\sigma)}^{m'}\wh{\sigma}^{-m'}\wh{\sigma}^{-m},\wh{\sigma}^{m+m'})\\

&=(\tilde{\rho}^0_g(a)\widetilde{\rho_g(\sigma)}^m\tilde{\rho}^0_g(a')\widetilde{\rho_g(\sigma)}^{-m}\widetilde{\rho_g(\sigma)}^{m+m'}\wh{\sigma}^{-m-m'},\wh{\sigma}^{m+m'}).
                                 \end{array}\]
Therefore, we must check that
\begin{equation}\label{a.a.}
\tilde{\rho}_g^0(\wh{\sigma}^ma'\wh{\sigma}^{-m})=\widetilde{\rho_g(\sigma)}^m\tilde{\rho}^0_g(a')\widetilde{\rho_g(\sigma)}^{-m}.
\end{equation}
However, observe the two group homomorphisms
$\wh{GL}^0_{\frakf,\infty}\to\wh{GL}^0_{\frakf,\infty}$ given by
$a'\mapsto\tilde{\rho}_g^0(\wh{\sigma}^ma'\wh{\sigma}^{-m})$ and
$a'\mapsto\widetilde{\rho_g(\sigma)}^m\tilde{\rho}^0_g(a')
\widetilde{\rho_g(\sigma)}^{-m}$ are liftings of the group
homomorphism $GL^0_{\frakf,\infty}\to GL^0_{\frakf,\infty}$ given
by $a\mapsto \rho_g(\sigma^ma\sigma^{-m})$. Since
$H^1(GL^0_{\frakf,\infty},\bbC^\times)=0$, the two homomorphisms
must be the same. Therefore, (\ref{a.a.}) holds and
$\widetilde{\rho_g}$ defined as above is a group homomorphism.

To finish the proof, we will show that for each $g\in
GL_\infty^+(\gl_\infty)$, we can choose a particular lifting
$\widetilde{\rho_g(\sigma)}$ so that $\widetilde{\rho_g}=c_g$ is
just conjugation by $g$ if $g\in GL_{\frakf,\infty}$ and
$\widetilde{\rho_{gg'}}=\widetilde{\rho_g}\widetilde{\rho_{g'}}$.
To this end, we make use the following two claims. Recall that
$g\in GL_\infty^+(\gl_\infty)$ acts on $\calO_\bbK$ by formula
(\ref{action of GL+inftyglinfty}). We will write
$\calO_\bbK=\bbC\ppart\oplus s\calO_\bbK$.

\medskip

(i) For any $g\in GL_\infty^+(\gl_\infty)$, there exists some
$h\in GL^0_{\frakf,\infty}$ such that
$g|_{\bbC\ppart}=h|_{\bbC\ppart}$ and
$\rho_g(\sigma)=\rho_h(\sigma)$.

\medskip

(ii)For $h,h'\in GL^0_{\frakf,\infty}$, if
$h|_{\bbC\ppart}=h'|_{\bbC\ppart}$ and
$\rho_h(\sigma)=\rho_{h'}(\sigma)$, then
$c_h(\wh{\sigma})=c_{h'}(\wh{\sigma})$.

\medskip

If these two claims hold, then we let
$\widetilde{\rho_g(\sigma)}=c_{h}(\wh{\sigma})$ for some $h\in
GL^0_{\frakf,\infty}$ satisfying properties in Claim (i). By Claim
(ii), this is well-defined. We check that
\[\widetilde{\rho_{gg'}}((a,\wh{\sigma}^m))=\widetilde{\rho_g}(\widetilde{\rho_{g'}}((a,\wh{\sigma}^m))).\]
Assume that $h,h'\in GL^0_{\frakf,\infty}$ such that
$\rho_g(\sigma)=c_h(\sigma)$, $g|_{\bbC\ppart}=h|_{\bbC\ppart}$
and $\rho_{g'}(\sigma)=c_{h'}(\sigma)$,
$g'|_{\bbC\ppart}=h'|_{\bbC\ppart}$. Then
$gg'|_{\bbC\ppart}=\rho_g(h')h|_{\bbC\ppart}$ and
$\rho_{gg'}(\sigma)=c_{\rho_g(h')h}(\sigma)$. Now,
\[\begin{array}{ll}\mathrm{l.h.s.}&=(\tilde{\rho}^0_{gg'}(a)(c_{\rho_g(h')h}(\wh{\sigma}))^m\wh{\sigma}^{-m},\wh{\sigma}^m),
\end{array}\]
\[\begin{array}{lll}\mathrm{r.h.s.}&=\widetilde{\rho_g}((\tilde{\rho}^0_{g'}(a)(c_{h'}(\wh{\sigma}))^m\wh{\sigma}^{-m},\wh{\sigma}^m))&\\

&=(\tilde{\rho}^0_g(\tilde{\rho}^0_{g'}(a))\tilde{\rho}^0_g((c_{h'}(\wh{\sigma}))^m\wh{\sigma}^{-m})(c_{h}(\wh{\sigma}))^m\wh{\sigma}^{-m},\wh{\sigma}^m)&\\

&=(\tilde{\rho}^0_{gg'}(a)\tilde{\rho}^0_g(\tilde{h'}\wh{\sigma}^m\tilde{h'}^{-1}\wh{\sigma}^{-m})(c_{h}(\wh{\sigma}))^m\wh{\sigma}^{-m},\wh{\sigma}^m)&\\

&=(\tilde{\rho}^0_{gg'}(a)\tilde{\rho}^0_g(\tilde{h'})\tilde{\rho}^0_g(\wh{\sigma}^m\tilde{h'}^{-1}\wh{\sigma}^{-m})(c_{h}(\wh{\sigma}))^m\wh{\sigma}^{-m},\wh{\sigma}^m)&\\

&=(\tilde{\rho}^0_{gg'}(a)\tilde{\rho}^0_g(\tilde{h'})(c_h(\wh{\sigma}))^m\tilde{\rho}^0_g(\tilde{h'}^{-1})\wh{\sigma}^{-m},\wh{\sigma}^m)&
\mbox{
by }                                 (\ref{a.a.})\\

         &=\mathrm{l.r.s},
\mbox{ since } \tilde{\rho}^0_g(\tilde{h'}) \mbox{ is a lifting of
}\rho_g(h').

                              \end{array}\]

We also check that $\widetilde{\rho_h}=c_h$ for $h\in
GL_{\frakf,\infty}$. But this follows from
$\tilde{\rho}^0_h(a)=c_h(a)$ for $h\in GL_{\frakf,\infty},
a\in\wh{GL}^0_{\frakf,\infty}$ and our choice
$\widetilde{\rho_h(\sigma)}$.

It remains to prove the two claims we made. We begin with the
proof of Claim (i). Recall that we choose $\sigma$ such that
$\sigma(t^i)=t^{i+1}$ and $\sigma(s^it^j)=s^it^j$ for $i\geq 1$.
For $g\in GL_\infty^+(\gl_\infty)$, there exists $N \in \Z$ such
that $g^{-1}s^N\calO_\bbK\subset s\calO_\bbK$. Then one can find
some Tate vector space $V\subset gs\calO_\bbK$ such that
$gs\calO_\bbK=V\oplus s^N\calO_\bbK$. Therefore
$\calO_\bbK=g\bbC\ppart\oplus gs\calO_\bbK=g\bbC\ppart\oplus
V\oplus s^N\calO_\bbK$. We can define an
$h:\calO_\bbK\to\calO_\bbK$ by $h|_{s^N\calO_\bbK}=\mathrm{id}$,
$h|_{\bbC\ppart}=g|_{\bbC\ppart}$ and
$h(\sum\limits_{i=1}^{N-1}s^i\bbC\ppart)=V$. It is clear that
$h\in GL_{\frakf,\infty}$ and $g\sigma g^{-1}=h\sigma h^{-1}$. Now
we choose some $h_1\in GL_{\frakf,\infty}$, such that
$\deg(h)=\deg(h_1)$, $h_1|_{\bbC\ppart}=\mathrm{id}$ and
$h_1(s\calO_\bbK)=s\calO_\bbK$. Since $h_1\sigma h_1^{-1}=\sigma$,
we see that $hh_1^{-1}$ satisfies all the required properties.

Next, we prove the Claim (ii). It is enough to prove that if $h\in
GL^0_{\frakf,\infty}$, $h|_{\bbC\ppart}=\mathrm{id}$ and $\sigma
h\sigma^{-1}=h$, then
$\wh{\sigma}\tilde{h}\wh{\sigma}^{-1}=\tilde{h}$ for $\tilde{h}$
any lifting of $h$ in $\wh{GL}^0_{\frakf,\infty}$. It is easy to
prove that in this case $h(s\calO_\bbK)=s\calO_\bbK$. Recall the
definition of $GL_{\frakf,n}(\gl_\infty)$, which is a subgroup of
$GL_{\frakf,\infty}$ consisting of whose
$a:\gl_\infty[[s]]\to\gl_\infty[[s]]$ such that
$a|_{s^n\gl_\infty[[s]]}=\mathrm{id}$. Let
$L_0=\bbC[[t]][[s]]dtds$. Assume that $h\in
GL^0_{\frakf,n}(\gl_\infty)$. Then $\tilde{h}=(h,e)$ where
\[e\in\det(\frac{hL_0}{L_0\cap
s^N\calO_\bbK}|\frac{L_0}{L_0\cap
s^N\calO_\bbK})^\times=\det(\frac{hL_0}{L_0\cap(s^N\calO_\bbK+\bbC\ppart)}
|\frac{L_0}{L_0\cap(s^N\calO_\bbK+\bbC\ppart)})^\times\] It is
clear that $\sigma$ induces identity on this line. Likewise,
$\wh{\sigma}=(\sigma,o)$ where
\[o\in\det(\frac{\sigma
L_0}{L_0\cap s^N\calO_\bbK}|\frac{L_0}{L_0\cap
s^N\calO_\bbK})^\times=\det(\frac{\sigma L_0}{L_0\cap
s\calO_\bbK}|\frac{L_0}{L_0\cap s\calO_\bbK})^\times\] and $h$
induces identity on this line. Therefore,
$\tilde{h}\wh{\sigma}=\wh{\sigma}\tilde{h}$.
\end{proof}

\begin{thm}\label{H^3
of GLinftyinfty} (i)
$H^2(GL_{\infty,\infty},\bbC^\times)=\bbC^\times$.

(ii) The cohomology class $[D_2]\in
H^2(GL_{\frakf,\infty},\bbC^\times)$ is transgressive, that is, it
belongs to
$H^2(GL_{\frakf,\infty},\bbC^\times)^{GL_{\infty,\infty}}=E_2^{0,2}$,
and $d_2[D_2]=0$, so that $[D_2]\in E_3^{0,2}$.

(iii) $E_3^{3,0}\cong H^3(GL_{\infty,\infty},\bbC^\times)$ and
$d_3[D_2]$ is non-zero in $H^3(GL_{\infty,\infty},\bbC^\times)$.
We denote this class by $[E_3]$.
\end{thm}
\begin{proof}
Since
$H^\bullet(\widetilde{GL}_\infty(\gl_\infty),\bbC^\times)\cong\bbC^\times$
and $H^1(GL_{\frakf,\infty},\bbC^\times)\cong\bbC^\times$, (i)
follows from the spectral sequence associated to (\ref{extension
of GLinftyinfty}). (ii) follows from Theorem \ref{crossed module
corresponding to E3} and Proposition \ref{liftings}. We shall
prove (iii). Indeed,$E_3^{3,0}\cong
H^3(GL_{\infty,\infty},\bbC^\times)$ follows from
\[E_2^{1,1}=H^1(GL_{\infty,\infty},H^1(GL_{\frakf,\infty},\bbC^\times))=0\]
which in turn follows from
\[H^1(GL_{\infty,\infty},\bbC^\times)\hookrightarrow
H^1(\widetilde{GL}_\infty(\gl_\infty),\bbC^\times)=0, \quad
H^1(GL_{\frakf,\infty},\bbC^\times)=\bbC^\times\] The reason for
$d_3[D_2]\neq 0$ is the same as that for $[C^0_2]\neq 0$ and
$[D^0_2]\neq 0$.
\end{proof}

We have a similar theorem for $\wh{GL}_{\infty,\infty}$, whose
proof is even simpler.

\begin{thm}\label{H^3 of
whGLinftyinfty}(i) $H^2(\wh{GL}_{\infty,\infty},\bbC^\times)=0$.

(ii) The cohomology class $[D^0_2]\in
H^2(GL^0_{\frakf,\infty},\bbC^\times)$ is transgressive, that is,
it belongs to
$H^2(GL^0_{\frakf,\infty},\bbC^\times)^{\wh{GL}_{\infty,\infty}}
=E_2^{0,2}$, and $d_2[D^0_2]=0$, so that $[D^0_2]\in E_3^{0,2}$.

(iii) $E_3^{3,0}\cong H^3(\wh{GL}_{\infty,\infty},\bbC^\times)$
and $d_3[D^0_2]$ is non-zero in
$H^3(GL_{\infty,\infty},\bbC^\times)$. We denote this class by
$[\wh{E}_3]$. Then $[\wh{E}_3]$ is the pullback of $[E_3]$ along
$\wh{GL}_{\infty,\infty}\to GL_{\infty,\infty}$.
\end{thm}

\subsection{Computation of the cohomology classes of
gerbal
   representations}\label{classes corr}

In this section we will calculate the cohomology class
corresponding to the gerbal representation of $GL_{\infty,\infty}$
on the category $\calC_{\calO_\bbK}^{\on{ss}}$ of Clifford modules
constructed in Theorem \ref{main1}. We will also construct a
gerbal representation of $GL_{\infty,\infty}$ on another abelian
category which also gives this cohomology class.

\subsubsection{The cohomology class of the
gerbal representation constructed in Theorem \ref{main1}}
\label{gerbal2}

Let us recall that the category $\calC_{\calO_\bbK}^{\on{ss}}$
defined in \S \ref{cat 2-Cl} is a semisimple abelian category. The
irreducible objects are vacuum modules $M_L$ over the Clifford
algebra $\Cl_{\calO_\bbK}$. They are labeled by secondary
lattices $L\subset\calO_\bbK$. Let us also recall that
$\Hom_{\calC_{\calO_\bbK}^{\on{ss}}}(M_L,M_{L'})=\det(L|L')$, if
$L$ and $L'$ are commensurable, and $0$ otherwise. In Theorem
\ref{main1}, we constructed a gerbal representation $F$ of
$GL_{\infty,\infty}$ on $\calC_{\calO_\bbK}^{\on{ss}}$. The
obstruction to upgrade this gerbal representation to a genuine
representation is a cohomology class in
$H^3(GL_{\infty,\infty},\calZ(\calC_{\calO_\bbK}^{\on{ss}})^\times)$
(see Theorem \ref{H^3}). Since $\calC_{\calO_\bbK}^{\on{ss}}$ is
$\bbC$-linear, there is a natural embedding
$\bbC^\times\subset\calZ(\calC_{\calO_\bbK}^{\on{ss}})^\times$
which admits a splitting. The second main theorem of this paper is

\begin{thm}\label{main2} The
cohomology class corresponding to the gerbal representation of
$GL_{\infty,\infty}$ on $\calC_{\calO_\bbK}^{\on{ss}}$, as
constructed in Theorem \ref{main1}, is
\[[E_3]\in
H^3(GL_{\infty,\infty},\bbC^\times)\subset
H^3(GL_{\infty,\infty},\calZ(\calC_{\calO_\bbK}^{\on{ss}})^\times),\]
whose existence is proved in Theorem \ref{H^3 of GLinftyinfty}.
\end{thm}

The strategy of proving this theorem is to realize this gerbal
representation in another way, using gerbal pairs, so that we
could apply results from \S \ref{second lemma} to calculate the
cohomology class.

Let us recall that we have the $\bbC^\times$-central extension of
$\wh{GL}_{\frakf,\infty}$ (\ref{extension of GLfrakfinfty}),
\[1\to\bbC^\times\to\wh{GL}_{\frakf,\infty}\to
GL_{\frakf,\infty}\to 1\] On the other hand, we have the group
extension (\ref{extension of GLinftyinfty}),
\[1\to GL_{\frakf,\infty}\to
\widetilde{GL}_\infty(\gl_\infty)\to GL_{\infty,\infty}\to 1\]
Furthermore, there is an action of
$\widetilde{GL}_\infty(\gl_\infty)$ on $\wh{GL}_{\frakf,\infty}$
which leaves the central subgroup
$\bbC^\times\subset\wh{GL}_{\frakf,\infty}$ invariant. Therefore,
$(\widetilde{GL}_\infty(\gl_\infty),\wh{GL}_{\frakf,\infty})$ is a
gerbal pair, in the sense of Definition \ref{gerbal pair} (see
Theorem \ref{crossed module corresponding to E3}).

Since $GL_\infty^+(\gl_\infty)$ acts on $\calO_\bbK$ (see formula
(\ref{action of GL+inftyglinfty})), it acts on the category of
$\Cl_{\calO_\bbK}$-modules. It is clear that
$\calC_{\calO_\bbK}^{\on{ss}}$ is invariant under this action.
Indeed, $g\in GL_\infty^+(\gl_\infty)$ will send $M_L$ to
$M_{gL}$. Therefore, $\widetilde{GL}_\infty(\gl_\infty)$ acts on
$\calC_{\calO_\bbK}^{\on{ss}}$ via the surjection
$\widetilde{GL}_\infty(\gl_\infty)\to GL_\infty^+(\gl_\infty)$. We
will denote this action by $G$ in what follows. If for any $a\in
GL_{\frakf,\infty} \subset \widetilde{GL}_\infty(\gl_\infty)$ we
had an isomorphism
$G_a\cong\mathbf{1}_{\calC_{\calO_\bbK}^{\on{ss}}}$, then we would
have the sought-after second construction of gerbal representation
of $GL_{\infty,\infty}$ on $\calC_{\calO_\bbK}^{\on{ss}}$, which
would allow us to calculate its third cohomology class. However,
this is not true, since $G_a$ is not always isomorphic to
$\mathbf{1}_{\calC_{\calO_\bbK}^{\on{ss}}}$. Namely, from Example
\ref{ex} in \S \ref{dl}, we know that $aL$ and $L$ are not
necessarily commensurable with each other, and therefore
$G_a(M_L)=M_{aL}$ are not necessarily isomorphic to $M_L$. To
remedy this problem, we will define a new semisimple abelian
category $\wh{\calC}_{\calO_\bbK}^{\on{ss}}$, whose irreducible
objects are still labeled by the secondary lattices
$L\subset\calO_\bbK$, but whose morphism sets are enlarged a
little bit.

Let $\bbL$ be a lattice in $\bbK$ and $L,L'$ two secondary
lattices of $\bbL$ (see Definitions \ref{lat} and
\ref{2-lattice}). Let us recall that we can define $\det(L|L')$ if
$L,L'$ are pseudo commensurable with each other (see Definition
\ref{pseudo commensurability}). We will denote by
$\wh{\calC}_{\bbL}^{\on{ss}}$ the semisimple abelian category,
whose irreducible objects, denoted by $\wh{M}_L$, are labeled by
the secondary lattices $L\subset\bbL$ and whose morphism sets are
\[\Hom_{\wh{\calC}_{\bbL}^{\on{ss}}}(\wh{M}_{L},\wh{M}_{L'})=
\left\{\begin{array}{ll}\det(L|L')& L,
L' \mbox{ are pseudo commensurable},\\

 0         & \mbox{ otherwise.
}\end{array}\right.\] Clearly, there is a full embedding
$\Upsilon_\bbL:\calC_{\bbL}^{\on{ss}}\to\wh{\calC}_{\bbL}^{\on{ss}}$,
which sends $M_L$ to $\wh{M}_L$.

Let us observe that the same construction as in Theorem
\ref{main1} yields a gerbal representation of $GL_{\infty,\infty}$
on $\wh{\calC}_{\bbL}^{\on{ss}}$, which we will denote by
$\wh{F}$. The embedding $\Upsilon_\bbL$ commutes with the gerbal
representations of $GL_{\infty,\infty}$, i.e.
$\wh{F}_g\circ\Upsilon_\bbL\cong\Upsilon_\bbL\circ F_g$.
Therefore, by Proposition \ref{invariant}, it is enough to
calculate the cohomology class corresponding to the gerbal
representation $\wh{F}$ of $GL_{\infty,\infty}$ on
$\wh{\calC}_{\calO_\bbK}^{\on{ss}}$.

Now we give another realization of this gerbal representation.
Observe that there is a genuine representation of
$GL_\infty^+(\gl_\infty)$ on $\wh{\calC}_{\calO_\bbK}^{\on{ss}}$.
Namely, for $g\in GL_\infty^+(\gl_\infty)$, it will send
$\wh{M}_L$ to $\wh{M}_{gL}$ and $\det(L|L')$ to $\det(gL|gL')$.
Therefore, we obtain a genuine representation of
$\widetilde{GL}_\infty(\gl_\infty)$ on
$\wh{\calC}_{\calO_\bbK}^{\on{ss}}$ via the surjection
$\widetilde{GL}_\infty(\gl_\infty)\to GL_\infty^+(\gl_\infty)$,
which is denoted by $\wh{G}$ in what follows. Furthermore, our
definition of $\wh{\calC}_{\calO_\bbK}^{\on{ss}}$ makes, for any
$a\in GL_{\frakf,\infty}$,
$G_a\cong\mathbf{1}_{\wh{\calC}_{\calO_\bbK}^{\on{ss}}}$. Indeed,
recall the $\bbC^\times$-central extension
$\wh{GL}_{\frakf,\infty}\stackrel{\pi}{\to} GL_{\frakf,\infty}$.
Thus, $\pi^{-1}(a)$ is a $\bbC^\times$-torsor for any $a\in
GL_{\frakf,\infty}$.

\begin{lem}
There is a $\bbC^\times$-equivariant embedding
\begin{equation}
\label{assump}
\pi^{-1}(a)\to\Hom_{\GL(\wh{\calC}_{\calO_\bbK}^{\on{ss}})}
(G_a,\mathbf{1}_{\wh{\calC}_{\calO_\bbK}^{\on{ss}}})
\end{equation}
such that the assumptions of Proposition \ref{variation} hold.
\end{lem}

\begin{proof}
Let us fix a secondary lattice $L\subset\calO_\bbK$. For
simplicity, we will use $L_0=\bbC[[t]][[s]]$. By Proposition
\ref{another presentation}, $\pi^{-1}(a)$ may be identified with
$\det(aL_0|L_0)^\times$. We have constructed in Theorem
\ref{crossed module corresponding to E3} an action $\tilde{\rho}$
of $GL_\infty^+(\gl_\infty)$ on $\wh{GL}_{\frakf,\infty}$. This
action gives for each $a$ an isomorphism
\[\gamma_{a,g}:\det(aL_0|L_0)\stackrel{\tilde{\rho}_{g^{-1}}}{\to}
\det(\rho_{g^{-1}}(a)L_0|L_0)\stackrel{g}{\to}\det(agL_0|gL_0)\]
If $g$ fixes $L_0$, then $\gamma_{a,g}$ is just multiplication of
a non-zero number. We claim that in this case
$\gamma_{a,g}=\mathrm{id}$. To prove this, let $R$ be
the subalgebra of $\gl_\infty$ consisting of $A:K\to K$ such that
$A(\calO_K)\subset\calO_K$ (recall that
$K=\bbC\ppart,\calO_K=\bbC[[t]]$). Then $GL_\infty^+(R)$ is the
subgroup of $GL_\infty^+(\gl_\infty)$ that fixes $L_0$. Now
$\gamma_{a,g}\gamma_{a,g'}=\gamma_{a,gg'}$ for $g,g'\in
GL_\infty^+(R)$. But from Proposition \ref{group acyclic},
$\gamma_{a,g}=\mathrm{id}$ for $a\in GL_\infty^+(R)$.

Let $L$ be any secondary lattice of $\calO_\bbK$. By Lemma
\ref{sec lat}, we can choose $g\in GL_\infty^+(\gl_\infty)$ such
that $gL_0=L$, and therefore obtain an isomorphism
$\gamma_{a,g}:\det(aL_0|L_0)\cong\det(aL|L)$, which is independent
of the choice of $g$ by previous discussion. Therefore, we obtain
for any $a\in GL_{\frakf,\infty}$,
$L\in\wh{\calC}_{\calO_\bbK}^{\on{ss}}$, an isomorphism
\[\pi^{-1}(a)\to\det(aL_0|L_0)\to\det(aL|L)=
\Hom_{\wh{\calC}_{\calO_\bbK}^{\on{ss}}}(\wh{M}_{aL},\wh{M}_{L}).\]
It is easy to check that these isomorphisms give the desired map
\eqref{assump} such that the assumptions of Proposition
\ref{variation} hold.
\end{proof}

Applying Proposition \ref{variation}, we obtain a gerbal
representation $\wh{G}$ of $GL_{\infty,\infty}$ on
$\wh{\calC}_{\calO_\bbK}^{\on{ss}}$. The cohomology class
corresponding to this gerbal representation is $[E_3]\in
H^3(GL_{\infty,\infty},\bbC^\times)$ as claimed in Theorem
\ref{H^3 of GLinftyinfty}.

To complete the proof of Theorem \ref{main2}, we need the
following:

\begin{lem} The two gerbal representations, $\wh{F}$ and $\wh{G}$,
of $GL_{\infty,\infty}$ on $\wh{\calC}_{\calO_\bbK}^{\on{ss}}$ are
equivalent, that is, for any $g\in GL_{\infty,\infty}$,
$\wh{F}_g\cong \wh{G}_g$.
\end{lem}

\begin{proof} Let $g\in
GL_{\infty,\infty}$, and $L\subset\calO_\bbK$ be a secondary
lattice. It is clear that we only need to show that
$\wh{F}_g(\wh{M}_L)\cong \wh{G}_g(\wh{M}_L)$.

First, $\wh{F}_g(\wh{M}_L)\cong \wh{M}_{L'}$, where $L'$ is a
secondary lattice of $\calO_\bbK$ fitting into the following exact
sequence
\[0\to gL\cap\calO_\bbK\to L'\to
V_{\calO_\bbK\cap g\calO_\bbK,\calO_\bbK}\to 0.\] (We recall that
$V_{\calO_\bbK\cap g\calO_\bbK,\calO_\bbK}$ was introduced in \S
\ref{dg}.) On the other hand, $\wh{G}_g(\wh{M}_L)\cong
\wh{M}_{aL}$, where $a\in GL_\infty^+(\gl_\infty)$ such that
$(a,g)\in\widetilde{GL}_{\infty}(\gl_\infty)$. Let us recall that
the definition of $\widetilde{GL}_{\infty}(\gl_\infty)$ (see \S
\ref{groups of matrices}) implies that $(a,g)\in
GL_\infty^+(\gl_\infty)\times GL_{\infty,\infty}$ is in
$\widetilde{GL}_{\infty}(\gl_\infty)$ if and only if $\exists n$
such that $g(s^n\calO_\bbK)\subset\calO_\bbK$ and
$a|_{s^n\calO_\bbK}=g|_{s^n\calO_\bbK}$. Therefore, $aL\cap
as^n\calO_\bbK=gL\cap gs^n\calO_\bbK=L'\cap gs^n\calO_\bbK=L'\cap
as^n\calO_\bbK$. That is, $aL$ and $L'$ are pseudo commensurable
with each other. Therefore, $\wh{F}_g(\wh{M}_L)\cong
\wh{G}_g(\wh{M}_L)$.
\end{proof}

Theorem \ref{main2} is now proved. As a corollary, we obtain that
the third cohomology class defined by the determinantal
$B\C^\times$-extension $\bbG\bbL_{\infty,\infty}$ of
$GL_{\infty,\infty}$ is also equal to $[E_3]$.

\subsubsection{The gerbal
representation of $GL_{\infty,\infty}$ on the category of
representations of $\wh{GL}_{\frakf,\infty}$} \label{gerbal
representation}

We can also realize the cohomology class $[E_3]\in
H^3(GL_{\infty,\infty},\bbC^\times)$ as the one corresponding to
the gerbal representation of $GL_{\infty,\infty}$ on a category of
representations of $\wh{GL}_{\frakf,\infty}$.

Denote by $\rep_1(\wh{GL}_{\frakf,\infty})$ the category of level
one complex representations of $\wh{GL}_{\frakf,\infty}$, i.e,
those representations on which the central $\bbC^\times$ acts via
the standard character. As shown in \S \ref{second lemma}, we have
a gerbal representation of $GL_{\infty,\infty}$ on the category of
$\rep_1(\wh{GL}_{\frakf,\infty})$. The natural embedding
$\bbC^\times\subset\calZ(\rep_1(\wh{GL}_{\frakf,\infty}))^\times$
admits a splitting, and therefore the induced map
$H^3(GL_{\infty,\infty},\bbC^\times)\to
H^3(GL_{\infty,\infty},\calZ(\rep_1(\wh{GL}_{\frakf,\infty}))^\times)$
is injective. Combining this with Theorem \ref{H^3 of
GLinftyinfty}, we obtain that the cohomology class corresponding
to this gerbal representation is just $[E_3]\in
H^3(GL_{\infty,\infty},\bbC^\times)$.

\subsubsection{Level one
representations of $\wh{GL}_{\frakf,\infty}$}

In this subsection we sketch a construction of some level one
representations of $\wh{GL}_{\frakf,\infty}$, using modules over a
Clifford algebra. They form a subcategory of
$\rep_1(\wh{GL}_{\frakf,\infty})$, which is invariant under the
gerbal action of $GL_{\infty,\infty}$.  Therefore, the gerbal
representation of $GL_{\infty,\infty}$ on this subcategory also
realizes the class $[E_3]$. This subsection is independent from
the rest of the paper, and its main goal is to describe a
conjectural relation between the categories
$\rep_1(\wh{GL}_{\frakf,\infty})$ and
$\wh{\calC}_{\calO_\bbK}^{\on{ss}}$.

Recall that in \S \ref{central extensions} we constructed level
one representations of $\wh{GL}_\infty$ by using discrete Clifford
modules over $\Cl_K$. We would like to make similar constructions
here to obtain some interesting level one representations of
$\wh{GL}_{\frakf,\infty}$. Note however that these representations
are no longer discrete.

Recall the notation $\bbK=\bbC\ppart\ppars$ and
$\bbK^*=\bbC\ppart\ppars dtds$. The natural pairing on
$\bbK\oplus\bbK^*$ is given by
\[(f,\omega)=\res_{t=0}\res_{s=0}f\omega \qquad f\in\bbK,
\omega\in\bbK^*.\] Recall that $\calO_\bbK=\bbC\ppart[[s]]$ and
$\Cl_{\calO_\bbK}=\Cl(\calO_\bbK\oplus\bbK^*/\calO_\bbK^\perp)$ is
the corresponding Clifford algebra. We define the action of
$GL_{\frakf,\infty}$ on $\calO_\bbK$ via formula (\ref{action of
GL+inftyglinfty}) and on $\bbK^*/\calO_\bbK^\perp$ so that the
bilinear form on $\calO_\bbK\oplus\bbK^*/\calO_\bbK^\perp$ is
$GL_{\frakf,\infty}$-invariant.

Let $L$ be a secondary lattice of $\calO_\bbK$ (see Definition
\ref{2-lattice}). We have constructed in \S \ref{cat 2-Cl} the
vacuum module $M_L$ over $\Cl_{\calO_\bbK}$. However, since for
$a\in GL_{\frakf,\infty}$, $M_{aL}$ and $M_L$ are not necessarily
isomorphic as $\Cl_{\calO_\bbK}$-modules, the construction we used
in \S \ref{another descr} for $GL_\infty$ cannot not be applied
here. Instead, it turns out that a certain {\em completion} of
$M_L$ carries the action of $\wh{GL}_{\frakf,\infty}$.

Let $L_0=\bbC[[t]][[s]]$ and $L_0'=t^{-1}\bbC[t^{-1}][[s]]$, so
that $\calO_\bbK=L_0\oplus L_0'$. Then
\[M_{L_0}=\ind_{\bigwedge(L_0\oplus
L_0^\perp)}^{\Cl_{\calO_\bbK}}\bbC|0\rangle\cong\bigwedge(L'_0\oplus
L'^\perp_0).\] There is a natural topology on
$\bigwedge(L'_0\oplus L'^\perp_0)$. Namely, a basis of open
neighborhoods of $0$ can be chosen as $\bigwedge(L'_0\oplus
L'^\perp_0)(s^n\calO_\bbK\cap L'_0)$. The completion of $M_{L_0}$
with respect to this topology is denoted by $\wh{M}_{L_0,L'_0}$.
This is a completed Clifford module over a certain completed
Clifford algebra. Namely, we denote by
$\overline{\Cl}_{\calO_\bbK,L'_0}$ the completion of
$\Cl_{\calO_\bbK}$ with respect to the topology in which a basis
of open neighborhoods of $0$ is given by the subspaces $U+V$,
where $U$ is the left ideal generated by open subspaces of
$L_0\oplus L_0^\perp$ and $V$ is the right ideal generated by
$s^n\calO_\bbK\cap L'_0$. It is clear that the multiplication of
$\Cl_{\calO_\bbK}$ extends to a multiplication of
$\overline{\Cl}_{\calO_\bbK,L'_0}$. Let $\wh{\bigwedge(L_0\oplus
L_0^\perp)}$ be the subalgebra of
$\overline{\Cl}_{\calO_\bbK,L'_0}$ topologically generated by
$L_0\oplus L_0^\perp$. Then
\[\wh{M}_{L_0,L'_0}=\ind_{\wh{\bigwedge(L_0\oplus
L_0^\perp)}}^{\overline{\Cl}_{\calO_\bbK,L'_0}}\bbC|0\rangle.\]

The reason that $\wh{M}_{L_0,L'_0}$ is a representation of
$\wh{GL}_{\frakf,\infty}$ is based on the following

\begin{lem} Let
$L_1$ be a secondary lattice of $\calO_\bbK$ that is pseudo
commensurable with $L_0$ (i.e., $\exists n$ such that
$s^n\calO_\bbK\cap L_1=s^n\calO_\bbK\cap L_0$). Let
$\wh{\bigwedge(L_1\oplus L_1^\perp)}$ be the subalgebra of
$\overline{\Cl}_{\calO_\bbK,L'_0}$ that is topologically generated
by $L_1\oplus L_1^\perp$, and let
\[\wh{M}_{L_1,L'_0}=\ind_{\wh{\bigwedge(L_1\oplus
L_1^\perp)}}^{\overline{\Cl}_{\calO_\bbK,L'_0}}\bbC|0\rangle.\]
Then $\wh{M}_{L_1,L'_0}\cong\wh{M}_{L_0,L'_0}$ and there is a
canonical isomorphism
\[\Hom_{\overline{\Cl}_{\calO_\bbK,L'_0}}
(\wh{M}_{L_1,L'_0},\wh{M}_{L_0,L'_0})=\det(L_1|L_0)\] such that if
$L_2$ is another secondary lattice that is pseudo commensurable
with $L_0,L_1$, then similar diagram as in Lemma \ref{Hom set}
holds.
\end{lem}

Let us sketch the proof of this lemma. First, we identify
$\Hom_{\overline{\Cl}_{\calO_\bbK,L'_0}}(\wh{M}_{L_1,L'_0},
\wh{M}_{L_0,L'_0})$ with
$(\wh{M}_{L_0,L'_0})^{\wh{\bigwedge(L_1\oplus L_1^\perp)}}$, the
$\wh{\bigwedge(L_1\oplus L_1^\perp)}$-invariant subspace of
$\wh{M}_{L_0,L'_0}$. It is easy to see that
$(\wh{M}_{L_0,L'_0})^{\wh{\bigwedge(L_0\oplus
L_0^\perp)}}=\bbC|0\rangle$. Next, one can construct an injective
map
\[\det(L_1|L_0)\to(\wh{M}_{L_0,L'_0})^{\wh{\bigwedge(L_1\oplus
L_1^\perp)}}\] Indeed, one can apply Lemma \ref{Hom set} to
construct for any $m\geq n$, an embedding
\[\det(\frac{L_1}{L_1\cap
s^m\calO_\bbK}|\frac{L_0}{L_0\cap
s^m\calO_\bbK})\to\wh{M}_{L_0,L'_0}/(s^m\calO_\bbK\cap L'_0)
\wh{M}_{L_0,L'_0}\] in a compatible way. The projective limit of
this system of maps is the desired one. Therefore, one obtains an
embedding
\[\varphi:\det(L_1|L_0)\to\Hom_{\overline{\Cl}_{\calO_\bbK,L'_0}}
(\wh{M}_{L_1,L'_0},\wh{M}_{L_0,L'_0})\] Finally, one can check for
$c\in\det(L_1|L_0)^\times$, $\varphi(c)$ is an isomorphism.
Therefore, $(\wh{M}_{L_0,L'_0})^{\wh{\bigwedge(L_0\oplus
L_0^\perp)}}=\bbC|0\rangle$ forces that $\varphi$ is an
isomorphism. This completes the sketch of proof of the lemma.

\medskip

Now observe that the action of $GL_{\frakf,\infty}$ on
$\Cl_{\calO_\bbK}$ is continuous with respect to the topology we
have introduced. Therefore, $GL_{\frakf,\infty}$ acts on
$\overline{\Cl}_{\calO_\bbK,L'_0}$. The same construction as in \S
\ref{central extensions} applies here and gives
$\wh{M}_{L_0,L'_0}$ the structure of a representation of
$\wh{GL}_{\frakf,\infty}$. From this representation we construct a
subcategory of level one representations of
$\wh{GL}_{\frakf,\infty}$. Recall that $GL_\infty^+(\gl_\infty)$
acts on $\rep_1(\wh{GL}_{\frakf,\infty})$. Denote this action by
$H$. Therefore, for each $g\in GL_{\infty}^+(\gl_\infty)$, there
is a level one representations of $\wh{GL}_{\frakf,\infty}$,
$H_g(\wh{M}_{L_0,L'_0})$. The semisimple abelian category
generated by these objects, which is denoted by $\mathfrak
R\frake\frakp$, is the ``smallest'' abelian subcategory of
$\rep_1(\wh{GL}_{\frakf,\infty})$ that contains
$\wh{M}_{L_0,L'_0}$ and is invariant under the gerbal
representation of $GL_{\infty,\infty}$ on
$\rep_1(\wh{GL}_{\frakf,\infty})$. It is clear that the gerbal
representation of $GL_{\infty,\infty}$ on $\mathfrak
R\frake\frakp$ realizes $[E_3]$.

\medskip

Finally, we explain a conjectural connection between $\mathfrak
R\frake\frakp$ and $\wh{\calC}_{\calO_\bbK}^{\on{ss}}$. Observe
that in order to define the completed Clifford algebra
$\overline{\Cl}_{\calO_\bbK,L'_0}$, we had to choose a splitting
$\calO_\bbK=L_0\oplus L'_0$, which we used to define a topology on
$\Cl_{\calO_\bbK}$. If we choose a different splitting
$\calO_\bbK=L_0\oplus L''_0$, we will obtain a different completed
Clifford algebra $\overline{\Cl}_{\calO_\bbK,L''_0}$ and the
module $\wh{M}_{L_0,L''_0}$ over it. However, for different
choices of splitting $\calO_\bbK$, the corresponding completed
Clifford algebras are canonically isomorphic. Namely, they are all
isomorphic to the completed Clifford algebra modeled on the
orthogonal topological vector space $(L_0\oplus
L_0^\perp)\oplus(\calO_\bbK/L_0\oplus(\bbK^*/\calO_\bbK^\perp)/
L_0^\perp)$.
 Furthermore, both $M_{L_0,L'_0}$ and $M_{L_0,L''_0}$
are identified with the vacuum module induced from the trivial
module over $\bigwedge(L_0\oplus L_0^\perp)$. The real reason that
the splitting may potentially come into the play is that the
action of $GL_{\frakf,\infty}$ on this completed Clifford algebra
depends on the choice of the splitting. Therefore, a priori,
$\wh{M}_{L_0,L''_0}$ and $\wh{M}_{L_0,L'_0}$ may not be isomorphic
as representations of $\wh{GL}_{\frakf,\infty}$. However, we
expect that these two representations are indeed canonically
isomorphic. If this is the case, then we can drop the subscript
$L'_0$ and simply denote them by $\wh{M}_{L_0}$. Now the
conjectural connection between $\wh{\calC}_{\calO_\bbK}^{\on{ss}}$
and $\mathfrak R\frake\frakp$ is the following: there is a
$GL_\infty^+(\gl_\infty)$-equivariant equivalence of categories
$\wh{\calC}_{\calO_\bbK}^{\on{ss}}\to\mathfrak R\frake\frakp$
which sends the object $\wh{M}_{L_0}$ in
$\wh{\calC}_{\calO_\bbK}^{\on{ss}}$ to the same named object in
$\mathfrak R\frake\frakp$. Hence we obtain a homomorphism of
the corresponding gerbal representations of $GL_{\infty,\infty}$.

\end{document}